\documentclass{amsart}

\usepackage{amsmath}
\usepackage{amsfonts}
\usepackage{amssymb}
\usepackage{amsthm}
\usepackage{comment}
\usepackage{epsfig}
\usepackage{psfrag}
\usepackage{mathrsfs}
\usepackage{amscd}
\usepackage[all]{xy}
\usepackage{rotating}
\usepackage{lscape}
\usepackage{amsbsy}
\usepackage{verbatim}
\usepackage{moreverb}

\usepackage[usenames]{color}

\newtheorem{theorem}{Theorem}[section]
\newtheorem{prop}[theorem]{Proposition}
\newtheorem{lemma}[theorem]{Lemma}
\newtheorem{cor}[theorem]{Corollary}

\theoremstyle{definition}
\newtheorem{definition}[theorem]{Definition}

\newtheorem{notation}[theorem]{Notation}
\newtheorem{variation}[theorem]{Variation}
\newtheorem{example}[theorem]{Example}
\newtheorem{non-example}[theorem]{Non-example}
\newtheorem{remark}[theorem]{Remark}
\newtheorem{observation}[theorem]{Observation}
\newtheorem{terminology}[theorem]{Terminology}
\newtheorem{notation/term}[theorem]{Notation/Terminology}

\theoremstyle{remark}

	\newcommand{\nc}{\newcommand}
	\nc{\DMO}{\DeclareMathOperator}
	\nc{\tensor}{\otimes}
	\nc{\unit}{{1}}
	\DMO{\ass}{\sf{Alg}}
	\DMO{\conf}{\sf{Conf}}
	\DMO{\open}{\sf{Open}}
	\DMO{\Mod}{\mathsf{-Mod}}
	\DMO{\exit}{\mathsf{Exit}}
	\DMO{\cech}{\mathsf{Cech}}
	\DMO{\fun}{\mathsf{Fun}}
	\DMO{\sbar}{\mathsf{Bar}}
	\DMO{\coend}{\mathsf{Coend}}
	\DMO{\modpair}{\mathsf{ModPair}}
	\DMO{\maps}{\mathsf{Maps}}
	\DMO{\emb}{\mathsf{Emb}}
	\DMO{\Ab}{\mathsf{Ab}}
	\DMO{\chain}{\mathsf{Chain}}
	\DMO{\ob}{\mathsf{ob}}
	\DMO{\lan}{\mathsf{Lan}}
	\DMO{\orderI}{ {\pi_0\cI^{\sqcup}_{\mathsf{or}}}}
	\DMO{\id}{id}
	\nc{\diskover}{(\disk_1^{\partial,\fr})_{/[-\infty,\infty]}}

\DeclareMathOperator{\Aut}{\sf Aut}
\DeclareMathOperator*{\colim}{\sf colim}

\DeclareMathOperator{\hocolim}{\sf hocolim}
\DeclareMathOperator{\holim}{\sf holim}

\DeclareMathOperator{\End}{\sf End}

\DeclareMathOperator{\Fun}{\sf Fun}
\DeclareMathOperator{\Iso}{\sf Iso}
\DeclareMathOperator{\Map}{\sf Map}

\DeclareMathOperator{\m}{\sf Mod}
\DeclareMathOperator{\bimod}{\sf Bimod}

\DeclareMathOperator{\alg}{\mathsf{-alg}}

\DeclareMathOperator{\Alg}{\mathsf{Alg}}

\DeclareMathOperator{\op}{\sf op}

\DeclareMathOperator{\Top}{\mathsf{Top}}
\DeclareMathOperator{\Emb}{\mathsf{Emb}}
\DeclareMathOperator{\Diff}{\mathsf{Diff}}

\DeclareMathOperator{\disk}{\mathsf{Disk}}
\DeclareMathOperator{\fr}{\sf fr}

\DeclareMathOperator{\bman}{\mathsf{Mfld}(\cB)}
\DeclareMathOperator{\pbman}{\mathsf{pMfld}(\cB)}
\DeclareMathOperator{\pmfld}{\mathsf{pMfld}}

\DeclareMathOperator{\rfn}{\mathsf{Rfn}}

\DeclareMathOperator{\mfld}{\mathsf{Mfld}}

\DeclareMathOperator{\snglr}{\mathsf{Snglr}}
\DeclareMathOperator{\psnglr}{\mathsf{pSnglr}}
\DeclareMathOperator{\bsc}{\mathsf{Bsc}}

\DeclareMathOperator{\Sing}{\mathsf{Sing}}

\DeclareMathOperator{\BO}{{\mathsf BO}}

\def\ot{\otimes}

\DeclareMathOperator{\oo}{\infty}

\DeclareMathOperator{\hh}{\sf HH}

\DeclareMathOperator{\free}{\sf Free}

\newcommand{\ra}{\rightarrow}
\newcommand{\la}{\leftarrow}
\newcommand{\xra}{\xrightarrow}
\newcommand{\xla}{\xleftarrow}
\newcommand{\ov}{\overline}
\newcommand{\un}{\underline}
\newcommand{\w}{\widetilde}

\def\cA{\mathcal A}\def\cB{\mathcal B}\def\cC{\mathcal C}\def\cD{\mathcal D}
\def\cE{\mathcal E}\def\cF{\mathcal F}\def\cG{\mathcal G}
\def\cI{\mathcal I}\def\cL{\mathcal L}
\def\cM{\mathcal M}\def\cO{\mathcal O}\def\cP{\mathcal P}
\def\cS{\mathcal S}
\def\cU{\mathcal U}\def\cX{\mathcal X}

\def\DD{\mathbb D}
\def\EE{\mathbb E}

\def\NN{\mathbb N}
\def\RR{\mathbb R}\def\SS{\mathbb S}
\def\VV{\mathbb V}
\def\ZZ{\mathbb Z}

\def\sB{\mathsf B}\def\sC{\mathsf C}\def\sD{\mathsf D}

\def\sI{\mathsf I}\def\sL{\mathsf L}
\def\sM{\mathsf M}\def\sO{\mathsf O}
\def\sS{\mathsf S}

\def\bH{\mathbf H}

\def\fB{\mathfrak B}

\begin{document}

%

\title{Structured singular manifolds \\
and factorization homology}
\author{David Ayala}
\author{John Francis}
\author{Hiro Lee Tanaka}
\date{}

\address{Department of Mathematics\\Harvard University\\Cambridge, MA 02138-2901}
\email{dayala@math.harvard.edu}
\address{Department of Mathematics\\Northwestern University\\Evanston, IL 60208-2370}
\email{jnkf@northwestern.edu}
\address{Department of Mathematics\\Northwestern University\\Evanston, IL 60208-2370}
\email{htanaka@math.northwestern.edu}
\thanks{DA was partially supported by ERC adv.grant no.228082, and by the National Science Foundation under Award No. 0902639. JF was supported by the National Science Foundation under Award number 0902974. HLT was supported by a National Science Foundation graduate research fellowship, by the Northwestern University Office of the President, and also by the Centre for Quantum Geometry of Moduli Spaces in Aarhus, Denmark. The writing of this paper was finished while JF was a visitor at Paris 6 in Jussieu.}

\begin{abstract} We provide a framework for the study of structured manifolds with singularities and their locally determined invariants. This generalizes factorization homology, or topological chiral homology, to the setting of singular manifolds equipped with various tangential structures. Examples of such factorization homology theories include intersection homology, compactly supported stratified mapping spaces, and Hochschild homology with coefficients. Factorization homology theories for singular manifolds are characterized by a  generalization of the Eilenberg-Steenrod axioms. Using these axioms, we extend the nonabelian Poincar\'e duality of Salvatore and Lurie to the setting of singular manifolds -- this is a nonabelian version of the Poincar\'e duality given by intersection homology. We pay special attention to the simple case of singular manifolds whose singularity datum is a properly embedded submanifold and give a further simplified algebraic characterization of these homology theories. In the case of 3-manifolds with 1-dimensional submanifolds, this structure gives rise to knot and link homology theories akin to Khovanov homology.
\end{abstract}



\maketitle

\tableofcontents

\section{Introduction}

In this work, we develop a framework for the study of manifolds with singularities and their locally determined invariants. We propose a new definition for a structured singular manifold suitable for such applications, and we study those invariants which adhere to strong conditions on naturality and locality. That is, we study covariant assignments from such singular manifolds to other categories, such as chain complexes, satisfying a locality condition, by which the global values on a manifold are determined by the local values. This leads to a notion of factorization homology for structured singular manifolds, satisfying a monoidal generalization of excision for usual homology, and a characterization of factorization homology as a monoidal generalization of the usual Eilenberg-Steenrod axioms for usual homology. As such, this work specializes to give the results of~\cite{facthomology} in the case of ordinary smooth manifolds.

One motivation for this study is related to the Atiyah-Segal axiomatic approach to topological quantum field theory, in which one restricts to compact manifolds, possibly with boundary. In examples coming from physics, however, it is frequently possible to define a quantum field theory on noncompact manifolds, such as Euclidean space. Given a quantum field theory -- perhaps defined by a choice of Lagrangian -- and two candidate physical spaces $M$ and $M'$, where $M$ is realized as a submanifold of $M'$, the field theories on $M$ and $M'$ can enjoy a relation: a field can be restricted from $M'$ to $M$, and, dually, an observable on $M$ can extend (by zero) to $M'$. If one only considers closed and connected manifolds, then embeddings are quite restricted, $M$ must be equal to $M'$, and the consequent mathematical structure is that the appropriate symmetry group of $M$ should act on the fields/observables on $M$. Alternately, allowing for observables on not necessarily closed manifolds gives further mathematical structure, because then not all embeddings are equivalences, and this is one motivation for the theory of factorization homology, -- or topological chiral homology after Lurie~\cite{dag}, Beilinson-Drinfeld~\cite{bd}, and Segal~\cite{segallocal} -- as well as the Costello-Gwilliam mathematical formalism for the observables of a perturbative quantum field theory.

Much of the theory of topological quantum field theories extends to allow for manifolds with singularities. In particular, Lurie outlines a generalization of the Baez-Dolan cobordism hypothesis for manifolds with singularities in~\cite{cobordism}. Roughly speaking, this states that if one has a field theory ${\sf Bord}_n \ra \cC$ and one wishes to extend it to a field theory allowing for manifolds with certain prescribed $k$-dimensional singularities $C(N)$, the cone on a $(k-1)$-manifold, then such an extension is equivalent to freely prescribing the morphism in the loop space $\Omega^{k-1} \cC$ to which the manifold $C(N)$ should be assigned. This is closely related to the result in the Baas-Sullivan cobordism theories of manifold with singularities~\cite{baas}, in which allowing cones on certain manifolds effects the quotient of the cobordism ring by the associated ideal generated by the selected manifold.

This begs the question of what theory should result if one allows for noncompact manifolds with singularities, where one can again push-forward observables along embeddings or, likewise, restrict fields. In other words, what is the theory of factorization homology for singular manifolds? A requisite first step is to clarify exactly what an embedding of singular manifolds is, and what a continuous family of such embeddings is. Two of the basic concepts that determine the nature of smooth manifold topology are of an isotopy of embeddings and of a smooth family of manifolds. In particular, one can define topological spaces $\Emb(M,N)$ and ${\sf Sub}(M,N)$ of embeddings and submersions whose homotopy types reflect the smooth topology of $M$ and $N$ in a significant way; this is in the contrast with the space of all smooth maps $\Map(M,N)$, the homotopy type of which is a homotopy invariant of $M$ and $N$. E.g., the study of the homotopy type of $\Emb(S^1, \RR^3)$ is exactly knot theory, while the homotopy type of $\Map(S^1,\RR^3)\simeq \ast$ is trivial.

The first and foundational step in our paper is therefore to build into the theory of manifolds with singularities the same structure, where there is a natural space of embeddings between singular manifolds, as well as a notion of a smooth family of singular manifolds parametrized by some other singular manifold. In fact, we do much more. We define a topological category $\snglr_n$ (Terminology~\ref{terms}) of singular $n$-manifolds whose mapping spaces are given by suitable spaces of stratified embeddings; our construction of this category of singular manifolds is iterative, and is very similar to Goresky and MacPherson's definition of stratified spaces, see~\cite{goreskymacpherson2}, in terms of iterated cones of stratified spaces of lower dimension. We then prove a number of fundamental features of these manifolds: Theorem~\ref{collar=fin}, that singular manifolds with a finite atlas have a finite open handle presentation; Theorem~\ref{no-pre-mans}, that singular premanifolds are equivalent to singular manifolds, and thus all local homological invariants can be calculated from an atlas which is not maximal. In order to establish these results, we develop in conjunction other parallel elements from the theory of manifolds, including a theory of tangent bundles for singular manifolds fitted to our applications, partitions of unity, and vector fields and their flows.

With this setup in hand, one can then study the axiomatics of the local invariants afforded by factorization homology. These factorization homology theories have a characterization similar to that of Lurie in the cobordism hypothesis or of Baas-Sullivan in the cobordism ring of manifolds with singularities. Namely, to give a homology theory for manifolds with singularities it is sufficient to define the values of the theory for just the most basic types of open manifolds -- $\RR^n$, an independent selection of value for each allowed type of singularity. There is then an equivalence between these homology theories for singular manifolds and the associated singular $n$-disk algebras, which generalize the $\cE_n$-algebras of Boardman and Vogt~\cite{bv}. More precisely, let  $\bsc_n$ be the full subcollection of basic singularities in $\snglr_n$; i.e., $\bsc_n$ is the smallest collection such that every point in a singular $n$-manifold has an open neighborhood isomorphic to an object in $\bsc_n$. Taking disjoint unions of these basic singularities generates an $\infty$-operad $\disk(\bsc_n)$. Fix $\cC^\ot$ a symmetric monoidal quasi-category satisfying a minor technical condition. We prove:

\begin{theorem}\label{basic} There is an equivalence
\[
\int \colon \Alg_{\disk(\bsc_n)}(\cC^\ot)\leftrightarrows \bH(\snglr_n, \cC^\ot) \colon \rho
\]
between algebras for $\disk(\bsc_n)$
and $\cC^\ot$-valued homology theories for singular $n$-manifolds. The right adjoint is restriction and the left adjoint is
factorization homology.
\end{theorem}

The proof of this result makes use of the entire apparatus of singular manifold theory built herein. The essential method of the proof is by induction on a handle-type decomposition, and so we require the existence of such decompositions for singular manifolds -- the nonsingular version being a theorem of Smale based on Morse theory -- which is provided by Theorem~\ref{collar=fin}. Another essential ingredient in describing factorization homology is a push-forward formula: given a stratified fiber bundle $F:M\ra N$, the factorization homology $\int_MA$ is equivalent to a factorization homology of $N$ with coefficients in a new algebra, $F_*A$. This result is an immediate consequence of the equivalence between singular premanifolds and singular manifolds (Theorem~\ref{no-pre-mans}), by making use of the non-maximal atlas on $M$ associated to the map $F$.

A virtue of Theorem~\ref{basic} is that the axioms for a homology theory are often easy to check, and, consequently, the theorem is often easy to apply. For instance, this result allows for the generalization of Salvatore and Lurie's nonabelian Poincar\'e duality to singular manifolds. For simplicity, we give a more restrictive version of our general result.

\begin{theorem}\label{everything} Let $M$ a stratified singular $n$-manifold, and $A$ a stratified pointed space with $n$ strata, where the $i^{\rm th}$ stratum is $i$-connective for each $i$. There is an equivalence
\[\Map_{\sf c}(M,A) \simeq \int_M \rho^*A\]
between the space of compactly-supported stratum-preserving maps and the factorization homology of $M$ with coefficients in the $\bsc_n$-algebra associated to $A$.
\end{theorem}

Our result, Theorem~\ref{PD}, is even more general: instead of a mapping space, one can consider a local system of spaces, again subject to certain connectivity conditions, and the cohomology of this local system (which is the space of the compactly supported global sections) is calculated by factorization homology. This recovers usual Poincar\'e duality by taking the local system to be that of maps into an Eilenberg-MacLane space. Factorization homology on a nonsingular manifold is built from configuration spaces, and thus nonabelian Poincar\'e duality is very close to the configuration space models of mapping spaces studied by Segal~\cite{segal}, May~\cite{may}, McDuff~\cite{mcduff}, B\"oedigheimer~\cite{bodig}, and others. Intuitively, the preceding theorem allows such models to extend to describe mapping spaces in which the source is not a manifold per se. At the same time, this result generalizes the Poincar\'e duality of intersection cohomology, in that one such local system on a singular manifold is given by $\Omega^\infty\sI\sC_{\sf c}^\ast$, the underlying space of compactly supported intersection cochains for a fixed choice of perversity $p$. The theory of factorization homology of manifolds with singularities developed in this paper can thus be thought of as bearing the same relation to intersection homology as factorization homology bears to ordinary homology.

Theorem~\ref{basic} has a very interesting generalization in which one considers only a restricted class of singularity types together with specified tangential data. That is, instead of allowing singular $n$-manifolds with all possible singularities in $\bsc_n$, one can choose a select list of allowable singularities. Further, one can add extra tangential structure, such as an orientation or framing, along and connecting the strata. We call such data a category of basic opens, and given such data $\cB$, one can form $\mfld(\cB)$, the collection of manifolds modeled on $\cB$. While the theory of tangential structures on ordinary manifolds is a theory of $\sO(n)$-spaces, or spaces over $\BO(n)$, the theory of tangential structures on singular manifolds is substantially richer -- it is indexed by $\bsc_n$, the category of $n$-dimensional singularity types, which in particular has non-invertible morphisms.  Here is a brief list of examples of $\cB$-manifolds for various $\cB$. 
\begin{itemize}
\item Ordinary smooth $n$-manifolds, possibly equipped with a familiar tangential structure such as an orientation, a map to a fixed space $Z$, or a framing.  
\item Smooth $n$-manifolds with boundary, possibly equipped with a tangential structure on the boundary, a tangential structure on the interior, and a map of tangential structures in a neighborhood of the boundary -- in familiar cases, this map is an identity map.  More generally, $n$-manifolds with corners, possibly equipped with compatible tangential structures on the boundary strata.  
\item ``Defects'' -- which is to say properly embedded $k$-submanifolds of ordinary $n$-manifolds, possibly equipped with a tangential structure such as a framing of the ambient manifold together with a splitting of the framing along the submanifold, or a foliation of the ambient manifold for which the submanifold is contained in a leaf.  
Relatedly there is ``interacting defects'' -- which is to say $n$-manifolds together with a pair of properly embedded submanifolds which intersect in a prescribed manner, such as transversely, perhaps equipped with a coloring of the intersection locus by elements of a prescribed set of colors. 
\item Graphs all of whose valences are, say, 1491, possibly with an orientation of the edges and the requirement that at most one-third of the edges at each vertex point outward.  
\item Oriented nodal surfaces, possibly marked.    

\end{itemize}
We prove the following:

\begin{theorem}\label{B's} There is an equivalence between $\cC^\ot$-valued homology theories for $n$-manifolds modeled on $\cB$ and algebras for $\cB$
\[\int:\Alg_{\disk(\cB)}(\cC^\ot)\leftrightarrows \bH(\mfld(\cB), \cC^\ot):\rho\]
in which the right adjoint is restriction and the left adjoint is factorization homology.
\end{theorem}

This result has the following direct connection with extended topological field theories. 
Specifically, the set of equivalence classes of homology theories $\bH(\mfld_n,\cC^{\ot})$ is naturally equivalent to the set of equivalences classes of extended topological field theories valued in $\Alg_{(n)}(\cC^\ot)$.  
Let us make this statement more explicit.  
Here $\Alg_{(n)}(\cC^{\ot})$ the is a higher Morita-type symmetric monoidal $(\infty,n)$-category heuristically defined inductively by setting the objects to be framed $n$-disk algebras and the $(\infty,n-1)$-category of morphisms between $A$ and $B$ to be ${\Alg_{(n)}}(A,B) := \Alg_{(n-1)}\bigl(\bimod_{A,B}(\cC^\ot)\bigr)$, with composition given by tensor product over common $n$-disk algebras, and with symmetric monoidal structure given by tensor product of $n$-disk algebras. For a category $\cX$, let $\cX^\sim$ denote the underlying groupoid, in which noninvertible morphisms have been discarded. We have the following commutative diagram:
\[\xymatrix{
\Alg_{\disk_n}(\cC^{\ot})^{\sim}\ar[d]\ar[r]&\bigl(\Alg_{(n)}(\cC^{\ot})^{\sim}\bigr)^{\sO(n)}\ar[d]\\
\bH(\mfld_n, \cC^{\ot})^{\sim} \ar[r] & \Fun^\ot({\sf Bord}_n, \Alg_{(n)}(\cC^{\ot}))\\
}\]
The bottom horizontal functor $\bH(\mfld_n, \cC^{\ot})^{\sim} \ra \Fun^\ot({\sf Bord}_n, \Alg_{(n)}(\cC^{\ot}))$ assigns to a homology theory $\cF$ the extended topological field theory $Z_\cF$ whose value on a compact $(n-k)$-manifold (with corners) $M$ is the $k$-disk algebra $Z_\cF(M):= \cF\bigl(\RR^k\times \overset{\circ}{M}\bigr)$ where $\overset{\circ}{M}$ is the interior of $M$ and this value is regarded as a $k$-disk algebra using the Euclidean coordinate together with the functoriality of $\cF$. A similar picture relating homology theories for singular $\cB$-manifolds with topological quantum field theories defined on $\mathsf{Bord}_\cB$, the bordism $(\oo,n)$-category of compact singular $(n-k)$-manifolds equipped with $\cB$-structure.  

The assignments above express every such extended topological field theory as arising from factorization homology. This is a consequence of features of the other three functors: the left vertical functor is an equivalence by the $n$-disk algebra characterization of homology theories of~\cite{facthomology} and Theorem~\ref{basic}; the right vertical functor is an equivalence by the cobordism hypothesis~\cite{cobordism} whose proof is outlined by Lurie building on earlier work with Hopkins; the top horizontal functor being surjective follows directly from the definition of $\Alg_{(n)}(\cC^{\ot})$ and the equivalence $\Alg_{\disk_n}(\cC^\ot) \simeq \Alg_{\disk_n^{\fr}}(\cC^\ot)^{\sO(n)}$. Thus, for field theories valued in this particular but familiar higher Morita category, our theorem has the virtue of offering a simpler approach to extended topological field theories, an approach which avoids any discussion of $(\infty,n)$-categories and functors between them.

The homology theory also offers more data in an immediately accessible manner, since one can evaluate a homology theory on noncompact $n$-manifolds and embeddings, whereas the corresponding extension of a field theory to allow this extra functoriality is far from apparent. One also has noninvertible natural transformations of homology theories, while all natural transformations between extended field theories are equivalences.

\smallskip

Theorem~\ref{B's} is a useful generalization of Theorem~\ref{everything} because one might desire a very simple local structure such as nested defects which are not true singularities. The most basic examples of such are manifolds with boundary or manifolds with a submanifold of fixed dimension. Let $\mfld_n^\partial$ be the collection of $n$-manifolds with boundary and embeddings which preserve boundaries, and let $\sD_n^\partial$ be the associated collection of open basics, generated by $\RR^n$ and $\RR^{n-1}\times \RR_{\geq 0}$. A special case of the previous theorem is the following.

\begin{cor} There is an equivalence 
\[
\int\colon \Alg_{\disk_n^\partial}(\cC^\ot) \leftrightarrows\bH(\mfld_n^\partial, \cC^\ot)\colon \rho
\] 
between $\disk_n^\partial$-algebras in $\cC^\ot$ and $\cC^\ot$-valued homology theories for $n$-manifolds with boundary.
\end{cor}

\begin{remark} A $\disk_n^\partial$-algebra is essentially an algebra for the Swiss-cheese operad of Voronov~\cite{voronov} in which one has additional symmetries by the orthogonal groups $\sO(n)$ and $\sO(n-1)$.
\end{remark}

For framed $n$-manifolds with a $k$-dimensional submanifold with trivialized normal bundle, one then has the following:

\begin{cor} There is an equivalence 
\[
\int\colon \Alg_{\disk_{n,k}^{\fr}}(\cC^\ot)\leftrightarrows\bH(\mfld^{\fr}_{n,k}, \cC^\ot)\colon \rho
\] 
between $\disk_{n,k}^{\fr}$-algebras in $\cC^\ot$ and $\cC^\ot$-valued homology theories for framed $n$-manifolds with a framed $k$-dimensional submanifold with trivialized normal bundle. The datum of a $\disk_{n,k}^{\fr}$-algebra is equivalent to the data of a triple $(A,B,{\sf a})$, where $A$ is a $\disk_n^{\fr}$-algebra, $B$ is a $\disk_k^{\fr}$-algebra, and ${\sf a}:\int_{S^{n-k-1}}A\ra \hh^\ast_{\sD^{\fr}_k}(B)$ is a map of $\disk^{\fr}_{k+1}$-algebras.
\end{cor}

Specializing to the case of 3-manifolds with a 1-dimensional submanifold, i.e., to links, the preceding provides an algebraic structure that gives rise to a link homology theory. To a triple $(A, B, f)$, where $A$ is a $\disk_3^{\fr}$-algebra, $B$ is an associative algebra, and $f: \hh_\ast (A) \ra \hh^\ast (B)$ is a $\disk_2^{\fr}$-algebra map, one can then construct a link homology theory, via factorization homology with coefficients in this triple. This promises to provide a new source of such knot homology theories, similar to Khovanov homology. Khovanov homology itself does not fit into this structure, for a very simple reason: a subknot of a knot $(U, K) \subset (U', K')$ does not define a map between their Khovanov homologies, from ${\sf Kh}(U,K)$ to ${\sf Kh}(U',K')$. Link factorization homology theories can be constructed, however, using the same input as Chern-Simons theory, and these appear to be closely related to Khovanov homology -- these theories will be the subject of another work.

We conclude our introduction with an outline of the contents. In Section 2, we develop a theory of structured singular $n$-manifolds suitable for our applications, in particular, defining a suitable space of embeddings $\Emb(M,N)$ between singular $n$-manifolds; we elaborate on the structure of singular $n$-manifolds, covering the notion of a tangential structure and a collar-gluing. In Section 3, we introduce the notion of homology theory for structured singular $n$-manifolds and prove a characterization of these homology theories. Section 4 applies this characterization to prove a singular generalization of nonabelian Poincar\'e duality. Section 5 considers other examples of homology theories, focusing on the case of ``defects'', which is to say a manifold with a submanifold. Section 6 completes the proof of the homology theory result of Section 3, showing that factorization homology over the closed interval is calculated by the two-sided bar construction. Section 7 establishes certain technical features of the theory of singular $n$-manifolds used in a fundamental way in Sections 2 and 3.

The reader expressly interested in the theory of factorization homology is encouraged to proceed directly to Section 2.7 and Section 3, referring back to the earlier parts of Section 2 only as necessary, and to entirely skip the last section.

\subsection{Conventions}\label{conventions}

In an essential way, the techniques presented are specific to the smooth setting as opposed to the topological, piecewise linear, or regular settings -- the essential technical point being that in the smooth setting a collection of infinitesimal deformations can be averaged, so that smooth partitions of unity allow the construction of global vector fields from local ones, the flows of which give regular neighborhoods (of singular strata strata of singular manifolds), which in turn lend to handle body decompositions.

Unless the issue is topical, by a topological space we mean a compactly generated Hausdorff topological space.  
Throughout, $\Top$ denotes the category of (compactly generated Hausdorff) topological spaces--it is bicomplete.  We regard $\Top$ as a Cartesian category and thus as self-enriched.  

In this work, we make use of both $\Top$-enriched categories as well as the equivalent {\it quasi-category} model  of $\oo$-category theory first developed in depth by Joyal; see~\cite{joyal}. In doing so, we will be deliberate when referring to a $\Top$-enriched category versus a quasi-category, and by a map $\sC \to \cD$ from a $\Top$-enriched category to a quasi-category it will be implicitly understood to mean a map of quasi-categories $N^c\sC \to \cD$ from the coherent nerve, and likewise by a map $\cD \to \sC$ it is meant a map of quasi-categories $\cD \to N^c\sC$.
We will denote by $\cS$ the quasi-category of Kan complexes, vertices of which we will also refer to as spaces.  We will use the term `space' to denote either a Kan complex or a topological space --  the context will vanquish any ambiguity.

Quasi-category theory, first introduced by Boardman \& Vogt in~\cite{bv} as ``weak Kan complexes," has been developed in great depth by Lurie in~\cite{topos} and~\cite{dag}, which serve as our primary references; see the first chapter of~\cite{topos} for an introduction. For all constructions involving categories of functors, or carrying additional algebraic structure on derived functors, the quasi-category model offers enormous technical advantages over the stricter model of topological categories. The strictness of topological categories is occasionally advantageous, however. For instance, the topological category of $n$-manifolds with embeddings can be constructed formally from the topological category of $n$-disks, whereas the corresponding quasi-category construction will yield a slightly different result.\footnote{That is, the quasi-category will instead give a completion $\widehat{\mfld}_n$ of the quasi-category of $n$-manifolds in which the mapping spaces $\Map(M,N)$ are equivalent to the space of maps $\Map(\EE_M,\EE_N)$ of presheaves on disks; see \textsection\ref{factorizhom}.}

{\bf Acknowledgements.} We are indebted to Kevin Costello for many conversations and his many insights which have motivated and informed the greater part of this work. We also thank Jacob Lurie for illuminating discussions, his inspirational account of topological field theories, and his substantial contribution to the theory of quasi-categories. JF thanks Alexei Oblomkov for helpful conversations on knot homology. Finally, we would like to thank the referees for valuable comments.

\section{Structured singular manifolds}
In this section we develop a theory of singular manifolds in a way commensurable with our applications.
Our development is consistent with that of Baas~\cite{baas} and Sullivan, as well as Goresky and MacPherson~\cite{goreskymacpherson2}, and we acknowledge their work for inspiring the constructions to follow.
We also define the notion of \emph{structured} singular manifolds. Structures on a singular manifold generalize fiberwise structures on the tangent bundle of a smooth manifold, such as framings.  As we will see, a structure on a singular manifold can be a far more elaborate datum than in the smooth setting.

\begin{remark}
Our account of (structured) singular manifolds is as undistracted as we can manage for our purposes.  For instance, we restrict our attention to a category of singular manifolds in which morphisms are `smooth' (in an appropriate sense) open embeddings which preserve strata.  In particular, because our techniques avoid it, we do not develop a theory of transversality (Sard's theorem) for singular manifolds.  
\end{remark}

\subsection{Smooth manifolds}
We begin with a definition of a smooth $n$-manifold---it is a slight (but equivalent) variant on the usual definition. 

A \emph{smooth $n$-manifold} is a pair $(M,\cA)$ consisting of a second countable Hausdorff topological space $M$ together with a collection $\cA=\{\RR^n\xra{\phi}M\}$ of open embeddings which satisfy the following axioms.
\begin{itemize}
\item[{\bf Cover:}]
The collection of images $\{\phi(\RR^n) \mid \phi\in \cA\}$ is an open cover of $M$.

\item[{\bf Atlas:}]
For each pair $\phi ,\psi\in \cA$ and $p\in \phi(\RR^n)\cap\psi(\RR^n)$ there is a diagram of smooth embeddings
$
\RR^n\xla{f} \RR^n\xra{g} \RR^n
$
such that 
\[
\xymatrix{
\RR^n  \ar[r]^g  \ar[d]^f
&
\RR^n \ar[d]^\psi
\\
\RR^n \ar[r]^\phi
&
M
}
\]
commutes and the image $\phi f(\RR^n) = \psi g(\RR^n)$ contains $p$.  

\item[{\bf Maximal:}]
Let $\RR^n \xra{\phi_0} M$ an open embedding.  Suppose for each $\phi \in \cA$ and each $p\in \phi_0(\RR^n)\cap \phi(\RR^n)$ there is a diagram of smooth embeddings $\RR^n\xla{f_0}\RR^n \xra{f} \RR^n$ such that 
\[
\xymatrix{
\RR^n \ar[r]^f  \ar[d]^{f_0}
&
\RR^n \ar[d]^\phi
\\
\RR^n \ar[r]^{\phi_0}
&
M
}
\]
commutes and the image $\phi_0 f_0(\RR^n) = \phi f(\RR^n)$ contains $p$.  
Then $\phi_0 \in \cA$.  
\end{itemize}   

A \emph{smooth embedding} $f\colon (M,\cA)\to (M',\cA')$ between smooth $n$-manifolds is a continuous map $f\colon M\to M'$ for which $f_\ast \cA = \{f\phi \mid \phi \in \cA\} \subset \cA'$.  
Smooth embeddings are evidently closed under composition. 
For $M$ and $M'$ smooth $n$-manifolds, denote by $\Emb(M,M')$ the topological space of smooth embeddings $M\to M'$, topologized with the weak Whitney $C^\infty$ topology. 
We define the $\Top$-enriched category 
$
\mfld_n
$
to be the category whose objects are smooth $n$-manifolds and whose space of morphisms $M\to M'$ is the space $\Emb(M,M')$.  

Note that the topology of $\Emb(M,M')$ is generated by a collection of subsets $\{\sM_{O}(M,M')\}$ defined as follows.  
This collection is indexed by the data of atlas elements $\phi\in \cA_M$ and $\phi' \in \cA_{M'}$, together with an open subset $O\subset \Emb(\RR^n,\RR^n)$ for which for each $g\in O$ the closure of the image $\ov{g(\RR^n)} \subset \RR^n$ is compact.  
Denote $\sM_O(M,M') = \{ f \mid f \circ \phi \in \phi'_\ast(O)\}$.

\subsection{Singular manifolds}
We now define categories of singular manifolds.  
Our definition traces the following paradigm:  A singular $n$-manifold is a second countable Hausdorff topological space equipped with a maximal atlas by \emph{basics} (of dimension $n$).  A basic captures the notion of an $n$-dimensional singularity type---for instance, a basic is homeomorphic to a space of the form $\RR^{n-k}\times CX$ where $X$ is a compact singular manifold of dimension $k-1$ and $C(-)$ is the (open) cone.  
There is a measure of the complexity of a singularity type, called the \emph{depth}, and $CX$ has strictly greater depth than $X$. 
Our definition is by induction on this parameter. We refer the reader to \textsection\ref{section:singularExamples} for examples. 

In what follows, for $Z$ a topological space let $\cO(Z)$ denote the poset of open subsets of $Z$ ordered by inclusion. 
Define the \textit{open cone on $Z$} to be the topological space
	\[
	C(Z) := \colim (\ast \leftarrow \RR_{\leq 0} \times Z \to \RR \times Z).
	\]
Note in particular that $C(\emptyset)$ is not empty, but is a point. 
In the case $Z \neq \emptyset$, this agrees with the definition of the cone as the quotient $Z \times [0,\infty){/ \sim}$, with the equivalence relation $(z,0) \sim (z',0)$.

\begin{definition}\label{singular-manifolds}
By induction on $-1\leq k \leq n$ we simultaneously define the following:
\smallskip
\begin{itemize}
\item A $\Top$-enriched category $\bsc_{\leq n , k}$ and a faithful functor $\bsc_{\leq n , k}\xra{\iota} \Top$. 
\item A $\Top$-enriched category $\snglr_{\leq n , k}$ and a faithful functor $\snglr_{\leq n , k} \xra{\iota} \Top$.  Denote the maximal subgroupoid $\snglr_{\leq n , k}^c\subset \snglr_{\leq n , k}$ spanned by those $X$ for which $\iota X$ is compact.  We emphasize that every morphism in $\snglr_{\leq n , k}^c$ is an isomorphism.  
\item For each $j$ a continuous functor $R^j\colon \bsc_{\leq n , k} \to \bsc_{\leq n+j,k}$ over $\RR^j\times -\colon \Top \to \Top$. 
\item For each $j$ a continuous functor $R^j\colon \cO(\RR^j)\times \snglr_{\leq n , k}\to \snglr_{\leq n+j,k}$ over $\cO(\RR^j)\times \Top\xra{\times} \Top$.  We denote the value $R^j_{\RR^j}X$ simply as $R^jX$.  

\end{itemize}

For all $n$, define $\bsc_{\leq n,-1} = \emptyset$ to be the empty category and $\snglr_{\leq n,-1} = \{\emptyset\}$ to be the terminal category over $\emptyset \in \Top$.  The functors $R^j$ and $\iota$ are determined.

For $k\geq 0$, define $\bsc_{\leq n , k}$ as follows:
\begin{itemize}
\item[]
\begin{itemize}

\item[$ob~\bsc_{\leq n , k}$:]
\[
ob\bsc_{\leq n , k}~ =~ ob \bsc_{\leq n , k-1}~ \coprod ~ob \snglr^c_{\leq k-1,k-1}~.
\]
We denote the object of $\bsc_{\leq n , k}$ corresponding to the object $X\in \snglr^c_{\leq k-1,k-1}$ using the symbol $U^n_X$ or $U_X$ when $n$ is understood. 
Define 
\[
\iota \colon \bsc_{\leq n , k}\to \Top
\]
on objects as $U\mapsto \iota U$ and $U_X\mapsto \RR^{n-k}\times C(\iota X)$.

\item[$mor~\bsc_{\leq n , k}$:]
Let $X,Y\in \snglr^c_{\leq k-1,k-1}$ be objects and $U,V\in \bsc_{\leq n , k-1}$. 
Define
\begin{eqnarray}
\bsc_{\leq n , k}(U,V) &=& \bsc_{\leq n , k-1}(U,V)~,
\nonumber \\
\bsc_{\leq n , k}(U,U_X) &=& \snglr_{\leq n , k-1}(U,R^{n-k}R_{>0}X)~,
\nonumber \\
\bsc_{\leq n , k}(U_X,V) &=& \emptyset~.
\nonumber 
\end{eqnarray}
The space $\bsc_{\leq n , k}(U_X,U_Y)$ is defined as follows.  Consider the
space $\widetilde{\bsc}_{n,k}(U_X,U_Y)\subset \snglr_{\leq n , k-1}(R^{n-k}RX,R^{n-k}RY) \times \Emb(\RR^{n-k}, \RR^{n-k})$ consisting of those pairs $(\w{f},h)$ for which there exists a morphism
$g\in \snglr_{n-1,k-1}(R^{n-k}X,R^{n-k}Y)$ and an isomorphism
$g_0\in \snglr^c_{\leq k-1,k-1}(X,Y)$
which fit into the diagram of topological spaces
\[
\xymatrix{
\iota X  \ar[d]^{\{0,0\}\times 1_{\iota X}}  \ar[r]^{g_0}
&
\iota Y  \ar[d]^{\{h(0),0\}\times 1_{\iota Y}}
\\
\RR^{n-k}\times \RR_{\leq 0} \times \iota X  \ar[r]^{\w{f}_|}  \ar[d]
&
\RR^{n-k}\times \RR_{\leq 0} \times \iota Y  \ar[d]
\\
\RR^{n-k}\times \iota X  \ar[r]^g  \ar[d]
&
\RR^{n-k}\times \iota Y \ar[d]
\\
\RR^{n-k}  \ar[r]^h
&
\RR^{n-k}
}
\]
where the unlabeled arrows are the standard projections.
Consider the closed equivalence relation on $\w{\bsc}_{n,k}(U_X,U_Y)$ by declaring $(\w{f},h)\sim (\w{f}',h')$ to mean $h=h'$ and the restrictions agree $\w{f}_{|\RR^{n-k}\times \RR_{\geq 0} \times \iota X} = \w{f}'_{|\RR^{n-k}\times \RR_{\geq 0} \times \iota X}$.
Define
\[
\bsc_{\leq n , k}(U_X,U_Y) := \widetilde{\bsc}_{n,k}(U_X,U_Y)_{/ \sim}~.
\]

Composition in $\bsc_{\leq n , k}$ is apparent upon observing that if
$(\w{f},h)\sim (\w{f}',h')$ and if $(\w{f}'',h'')\sim (\w{f}''',h''')\in \widetilde{\bsc}_{n,k}(U_Y,U_Z)$, then $(\w{f}''\circ \w{f}, h'' \circ h)\sim (\w{f}'''\circ \w{f}', h'\circ h''')$. 
The functor $\iota \colon \bsc_{\leq n , k} \to \Top$ has already been defined on objects, and we define it on morphisms by assigning to $(f,h)\colon U^n_X \to U^n_Y$ the map $\RR^{n-k}\times C(\iota X) \to \RR^{n-k}\times C(\iota Y)$ given by $[(u,s,x)]\mapsto [f(u,s,x)]$ if $ s\geq 0$ and as $[(u,s,x)]\mapsto [h(u)]$ if $s\leq 0$ -- this map is constructed so to be well-defined and continuous.  Moreover, the named equivalence relation exactly stipulates that $\iota$ is continuous and injective on morphism spaces.  
As so, for $X\cong Y\in \snglr_{\leq k-1,k-1}^c$ isomorphic objects, we denote a morphism $U_X \xra{[(\w{f},h)]} U_Y$ simply as its induced continuous map $\iota U_X \xra{f} \iota U_Y$. 
Unless the issue is sensitive, we will often abuse notation and use the same letter $f$ to refer to the first coordinate of a representative of $f=[(\w{f},h)]$.

\item[R:]
The functor $R$ is determined on objects and morphisms by induction upon declaring $R^jU^n_X = U^{n+j}_X$ and $R^j(f,h) = (1_{\RR^j}\times f,1_{\RR^j} \times h)$.

\item[$ob~\snglr_{\leq n , k}$:]
An object $X$ of $\snglr_{\leq n , k}$ is a pair $(\iota X, \cA)$ consisting of a second countable Hausdorff topological space $\iota X$, together with a set $\cA=\{(U,\phi)\}$ called an {\em atlas}, whose elements consist of an object $U\in \bsc_{\leq n , k}$ together with an open embedding $\iota U \xra{\phi} \iota X$. The atlas must satisfy the following conditions: 
\begin{itemize}
\item[{\bf Cover:}]
The collection of images $\{\phi(\iota U)\subset \iota X\}$ is an open cover of $\iota X$.

\item[{\bf Atlas:}]
For each pair $(U,\phi),(V,\psi)\in \cA$ and $p\in \phi(\iota U)\cap\psi(\iota V)$ there is a diagram 
$
U\xla{f} W\xra{g} V
$
in $\bsc_{\leq n , k}$  such that 
\[
\xymatrix{
\iota W  \ar[r]^g  \ar[d]^f
&
\iota V  \ar[d]^\psi
\\
\iota U  \ar[r]^\phi
&
\iota X
}
\]
commutes and the image $\phi f(\iota W) = \psi g(\iota W)$ contains $p$.  

\item[{\bf Maximal:}]
Let $U_0$ be an object of $\bsc_{\leq n , k}$ and let $\iota U_0\xra{\phi_0} \iota X$ be an open embedding.  Suppose for each $(U,\phi)\in \cA$ and each $p\in \phi_0(\iota U_0)\cap \phi(\iota U)$ there is a diagram $U_0\xla{f_0} W \xra{f} U$ in $\bsc_{\leq n , k}$ such that 
\[
\xymatrix{
\iota W  \ar[r]^f  \ar[d]^{f_0}
&
\iota U  \ar[d]^\phi
\\
\iota U_0  \ar[r]^{\phi_0}
&
\iota X
}
\]
commutes and the image $\phi_0 f_0(\iota W) = \phi f(\iota W)$ contains $p$.  
Then $(U_0,\phi_0)\in \cA$.  

\end{itemize}
Define $\iota \colon \snglr_{\leq n , k}\to \Top$ on objects by $(\iota X,\cA)\mapsto \iota X$.

\item[$mor~\snglr_{\leq n , k}$:]
Let $X ,Y \in \snglr_{\leq n , k}$ be objects, written respectively as $(\iota X , \cA_X)$  and $(\iota Y,\cA_Y)$.  The space of morphisms $\snglr_{\leq n , k}(X,Y)$ has as its underlying set those continuous maps $\iota X\xra{f} \iota Y$ for which $f_\ast \cA_X := \{(U,f\phi)\mid (U,\phi)\in \cA_X\}\subset \cA_Y$.  
The topology on $\snglr_{\leq n , k}(X,Y)$ is generated by the collection of subsets $\{\sM_{O,(U,\phi),(V,\psi)}(X,Y)\}$ defined as follows.  
This collection is indexed by the data of atlas elements $(U,\phi)\in \cA_X$ and $(V,\psi)\in \cA_Y$, together with an open subset $O\subset \bsc_{\leq n , k}(U,V)$ for which for each $g\in O$ the closure of the image $\ov{g(\iota U)} \subset \iota V$ is compact.  
Denote $\sM_{O,(U,\phi),(V,\psi)}(X,Y) = \{ f \mid f \circ \phi \in \psi_\ast(O)\}$.

\vspace{8pt}

Composition is given on underlying sets by composing continuous maps.
To check that composition is continuous amounts to checking that for each $(\iota X,\cA)\in \snglr_{\leq n , k}$ the collection $\{\phi(\iota U)\mid (U,\phi)\in \cA\}$ is a basis for the topology of $\iota X$, as well as checking that $\iota X$ is locally compact.  These points are straight ahead and will be mentioned in Lemma~\ref{about-X}.  
The functor $\iota \colon \snglr_{\leq n , k} \to \Top$ is given on objects by $(\iota X,\cA)\mapsto \iota X$ and on morphisms by $f\mapsto f$.  These assignments are evidently functorial, continuous, and faithful.

\item[R:]
Define $R^j\colon \cO(\RR^j)\times \snglr_{\leq n , k} \to \snglr_{\leq n+j,k}$ on objects as $(Q,(\iota X,\cA))\mapsto (Q\times \iota X , \cA_Q)$ where $\cA_Q = \{(U',\phi')\}$ consists of those pairs for which for each $p\in \phi'(\iota U')$ there is an element $(U,\phi)\in \cA$ and a morphism $R^j U \xra{f} U'$ in $\bsc_{n+j,k}$ such that the diagram
\[
\xymatrix{
\RR^j \times \iota U  \ar[r]^f  \ar[d]^{pr}  
&
\iota U'  \ar[r]^{\phi'}
&
Q\times \iota X  \ar[d]^{pr}
\\
\iota U  \ar[rr]^\phi
&
&
\iota X
}
\]
commutes and $p\in \phi'\bigl(f(\RR^j\times \iota U)\bigr)$.  Clearly $Q\times \iota X$ is second countable and Hausdorff and $\cA_Q$ is an open cover.  It is routine to verify that $\cA_Q$ is a maximal atlas.
That $R^j$ describes a continuous functor is obvious.  

\end{itemize}
\end{itemize}
\end{definition}

We postpone the proofs of the following two lemmas to the end of~\S\ref{intuitions}.
\begin{lemma}\label{basic-full}
There is a standard fully faithful embedding
\[
\bsc_{\leq n , k} \subset \snglr_{\leq n , k}
\]
over $\Top$.  This embedding respects the structure functors $R^j$.  
\end{lemma}

For $k\leq k'$ notice the tautological fully faithful inclusions $\bsc_{\leq n , k}\subset \bsc_{\leq n , k'}$ over tautological inclusions $\snglr_{\leq n , k} \subset \snglr_{\leq n , k'}$.

\begin{lemma}\label{tautological}
Let $r$ be a non-negative integer.  
There are tautological fully faithful inclusions
\[
\bsc_{\leq n , k}\subset \bsc_{\leq n +r,k+r}
\]
over tautological inclusions 
\[
\snglr_{\leq n , k} \subset \snglr_{\leq n+r , k+r}
\]
over $\Top$.
Moreover, these are inclusions of components -- namely, there are no morphisms between an object of $\snglr_{\leq n , k}$ and an object of $\snglr_{\leq n +r ,k+r}\smallsetminus \snglr_{\leq n , k}$.  

\end{lemma}

\begin{terminology}\label{terms}
\begin{itemize}
\item[] 
\item Denote the complements 
\begin{itemize}
\item $\bsc_{n,k} = \bsc_{\leq n , k} \smallsetminus \bsc_{\leq n-1,k}$, 
\item $\bsc_{n,=k} = \bsc_{n,k} \smallsetminus \bsc_{n,k-1}$. 
\item $\snglr_{n,k} = \snglr_{\leq n , k} \smallsetminus \snglr_{\leq n-1,k}$,  
\item $\snglr_{n,k}^c = \snglr_{\leq n,k}^c \smallsetminus \snglr_{\leq n-1,k}$.
\end{itemize}

\item Denote 
\begin{itemize}
\item $\bsc_n=\bsc_{n,n}~,~{}~{}~{}~{}~{}~{}~{}~{}~{}~{}~\bsc_{\leq n} = \bsc_{\leq n ,n}$, 
\item $\snglr_n = \snglr_{n , n}~,~{}~{}~{}~{}~{}~{}~\snglr_{\leq n} = \snglr_{\leq n,n}$.  
\end{itemize}
\item We do not distinguish between an object of $\bsc_n$ and its image in $\snglr_n$.
\item We do not distinguish between objects of $\snglr_{n , k}$ and of $\snglr_n$.  
\item We do not distinguish between a morphism in $\snglr_n$ and its image under $\iota$.  
\item We will refer to an object $U\in \bsc_n$ as a \emph{basic}.
We will often denote the basic $U^n_{\emptyset^{-1}}\in \bsc_{n,0}$ by its underlying topological space $\RR^n$ which we understand to be equipped with its standard smooth structure.  
\item We will refer to an object $X\in \snglr_n$ as a \emph{singular $n$-manifold}.
\item We will say a singular $n$-manifold $X$ has \emph{depth $k$} if $X\in \snglr_{n,=k}$.  
\item We say a point $p\in X$ is of depth $k$ if there is a chart $(U,\phi)$, written as $\iota U = \RR^{n-k}\times C(\iota Y)$ so that $0\in \RR^{n-k}\subset \iota U$, for which $\phi(0) = x$. 
This notion of depth is well-defined -- this is an insubstantial consequence of Lemma~\ref{only-equivalences} for instance.

\item For $\mathfrak{P}$ a property of a topological space or a continuous map, say an object or morphism of $\snglr_n$ has property $\mathfrak{P}$ if its image under $\iota$ does.  
\item We adopt the convention that $\bsc_{\leq n , k} = \emptyset$ is empty and $\snglr_{\leq n , k} = \{\emptyset^n\}$ is terminal over the empty set in $\Top$ whenever $k<-1$.  

\end{itemize}

\end{terminology}

\subsection{Examples of singular manifolds}\label{section:singularExamples}

\example[Base cases and singular manifolds of dimension $0$]{
When $n=-1$, $\bsc_{-1}$ is the category with an empty set of objects; and $\snglr_{-1}$ is the terminal category, having a single object which we denote by $\emptyset^{-1}$.  The functor $\iota$ sends $\emptyset^{-1}$ to the empty topological space.

When $n=0$, $\bsc_{0,-1}$ is again the category with an empty set of objects, and $\snglr_{0,-1}^c = \snglr_{0,-1}$ is the terminal category whose object we denote by $\emptyset^0$. Hence we see that $\bsc_{0,0} = \{U_{\emptyset^{-1}}^0\} = \{\RR^0\}$ is the terminal category with $\iota \RR^0 = \ast \in \Top$. $\snglr_{0,0}$ is the category of countable sets and injections, while $\snglr_{0,0}^c$ is the category of finite sets and bijections. 
}

\example[Smooth manifolds]{
Singular manifolds of depth 0 and smooth manifolds are one and the same.
Specifically, $\bsc_{n,0}$ is the $\Top$-enriched category with a single object $U^n_{\emptyset^{-1}} = \RR^n$ and with space of morphisms $\Emb(\RR^n,\RR^n)$. 
$\snglr_{n,0} = \mfld_n$ is the $\Top$-enriched category of smooth $n$-manifolds and their smooth embeddings while $\snglr_{n,0}^c$ is the subcategory of compact $n$-manifolds and their isomorphisms.  The functor $\iota$ sends a smooth manifold $X$ to its underlying topological space.
}

\example[Singular manifolds of dimension 1]{
The first non-smooth example, $\bsc_{1,=1}$, is the $\Top$-enriched category with $\ob \bsc_{1,=1} = \{U^1_J\mid J \text{ a finite set}\}$ where the underlying space is the `spoke' $\iota U^1_J = C(J)$. This is the open cone on the finite set $J$, and we think of it as an open neighborhood of a vertex of valence $|J|$.

$\snglr_{1,1}$ is a $\Top$-enriched category summarized as follows.  Its objects are (possibly non-compact) graphs with countably many vertices, edges, and components.  
These types of graphs were utilized by S. Galatius in~\cite{galatius:graphs}.

For $X$ such a graph, denote by $\w{X}$ a smooth $1$-manifold obtained from $X$ by deleting the subset $V$ of its vertices and gluing on collars in their place, we call $\w{X}$ a \emph{dismemberment} of $X$ -- the construction of $\w X$ depends on choices of collars but is well-defined up to isomorphism rel $X\smallsetminus V$.  
(See \textsection\ref{section:dismemberment}.) 
There is a quotient map $\w{X} \to X$ given by collapsing these collar extensions to the vertices whence they came.  
A morphism $X\xra{f}Y\in \snglr_{1,1}$ is an open embedding of such graphs for which there are dismemberments $\w{X}$ and $\w{Y}$ which fit into a diagram
\[
\xymatrix{
\w{X}  \ar[r]^{\w{f}} \ar[d]
&
\w{Y}  \ar[d]
\\
X\ar[r]^f
&
Y
}
\]
where $\w{f}$ is smooth.  

Finally, $\snglr_{1,1}^c\subset \snglr_{1,1}$ is the subcategory whose objects are those such graphs which are compact, and whose morphisms are the isomorphisms among such.  
}

\begin{example}[Corners]
Let $M$ be an $n$-manifold with corners; for instance, $M$ could be an $n$-manifold with boundary.  
Each point in $M$ has a neighborhood $U$ which can be smoothly identified with $\RR^{n-k}\times [0,\infty)^k$ for some $0\leq k \leq n$.  
Notice that $\RR^{n-k}\times [0,\infty)^k\cong \RR^{n-k}\times C(\Delta^{k-1})$ is homeomorphic to the product of $\RR^{n-k}$ with the open cone on the topological $(k-1)$-simplex. Such an $M$ is an object of $\snglr_n$.

\end{example}

\begin{example}[Embedded submanifolds]
The data $P \subset M$ of a properly embedded $d$-dimensional submanifold of an $n$-dimensional manifold is an example of a singular $n$-manifold.
Each point of $M$ has a neighborhood $U$ so that the pair $(P\cap U\subset U)$ can be  identified with either $(\emptyset\subset\RR^n)$ or $(\RR^d\subset \RR^n)$.  
Note that the basic an open ball around $p \in P$ is homeomorphic to the product of $\RR^d$ with the open cone $C(S^{n-d-1})$. Hence the data $P \subset M$ is an object in $\snglr_{n,d}$.
\end{example}

\begin{example}
An object of $\bsc_{2,2}$ is $\RR^2$, $\RR\times C(J)$ with $J$ a finite set, or $C(Y)$ where $Y$ is a compact graph.  
Heuristically, an object of $\snglr_{2,2}$ is a topological space which is locally of the form of an object of $\bsc_{2,2}$.
The geometric realization of a simplicial complex of dimension $2$ is an example of such an object.  For instance, the cone on the $1$-skeleton of the tetrahedron $\Delta^3$, here there is a single point of depth $2$ which is the cone point and the subspace of depth $1$ points are is the cone on the vertices of this $1$-skeleton.  Another example is a nodal surface, here the depth $2$ points are the nodes and there are no depth $1$ points.  
\end{example}

\begin{example}[Simplicial complexes]
Let $\sS$ be a finite simplicial complex such that every simplex is the face of a simplex of dimension $n$. The geometric realization $|\sS|$ is a compact singular $n$-manifold. 
Indeed, inductively, the link of any simplex of dimension $(n-k)$ is a compact singular $(k-1)$-manifold.  It follows that a neighborhood of any point $p\in |\sS|$ can be identified with $\RR^{n-k}\times CX$ for $X=|\mathsf{Link}(\sigma_p)|$ the singular $(k-1)$-manifold which is the geometric realization of the link of the unique simplex $\sigma_p$ the interior of whose realization contains $p$.  
\end{example}

\begin{example}
Note that for any singular $(k-1)$-manifold $X$ there is an apparent inclusion ${\sf Iso}(X,X) \to \bsc_n(U^n_X,U^n_X)$.  
In particular, while there is a canonical homeomorphism of the underlying space $\iota \RR^n \cong \iota U^n_{S^{n-1}}$, there are fewer automorphisms of $\RR^n\in \bsc_n$ than of $U^n_{S^{n-1}}\in \bsc_n$.  Hence $\RR^n \ncong U^n_{S^{n-1}} \in \snglr_n$.  
\end{example}

\begin{example}[Whitney stratifications] The Thom-Mather Theorem~\cite{mather} ensures that a Whitney stratified manifold $M$ has locally trivializations of its strata of the form $R^k C(M)$ and thus is an example of a stratified manifold in the sense used here.
In particular, real algebraic varieties are examples of singular manifolds in the sense above -- though the morphisms between two such are vastly different when regarded as varieties versus as singular manifolds.  
\end{example}

\subsection{Intuitions and basic properties}\label{intuitions}
We discuss some immediate consequences of Definition~\ref{singular-manifolds}. It might be helpful to think of the $\Top$-enriched category  $\snglr_n$ as the smallest which realizes the following intuitions:
\begin{itemize}
\item Smooth $n$-manifolds are examples of singular $n$-manifolds.
\item If $X$ is a singular $n$-manifold then $\RR_{\geq 0} \times X$ has a canonical structure of a singular $(n+1)$-manifold.
\item For $P\subset X$ a compact singular $k$-submanifold of a singular $n$-manifold, then the quotient $X/P$ has a canonical structure of a singular $n$-manifold.
\item For $X$ a compact singular manifold, the map $\Aut(X,X) \to \Emb(CX , CX)$ is a homotopy equivalence.  
(This feature is of great importance and is responsible for the homotopical nature of the theory of singular manifolds developed in this article.)
\end{itemize}

The functor $\iota$ preserves the basic topology of a singular manifold $X$ one expects. We will use most of the following facts throughout the paper without mention:

\begin{lemma}\label{about-X}
Let $X=(\iota X,\cA)$ be a singular $n$-manifold.  Then 
\begin{enumerate}

\item  For each morphism $X\xra{f} Y$ the continuous map $\iota X \xra{\iota f} \iota Y$ is an open embedding.

\item The topological dimension of $\iota X$ is at most $n$.

\item The topological space $\iota X$ is paracompact and locally compact.  

\item The collection $\{\phi(\iota U)\mid (U,\phi)\in \cA\}$ is a basis for the topology of $\iota X$.

\item The atlas $\cA$ consists of those pairs $(U,\phi)$ where $U\in \bsc_n$ and $U\xra{\phi} X$ is morphism of singular manifolds.  

\item Let $O\subset \iota X$ be an open subset. There is a canonical singular $n$-manifold $X_O$ equipped with a morphism $X_O \to X$ over the inclusion $O\subset \iota X$.  

\item 
Let $\iota Y$ be a Hausdorff topological space.
Consider a commutative diagram of categories
\[
\xymatrix{
\cU  \ar[r]^q  \ar[d]^p
&
\snglr_n  \ar[d]^\iota
\\
\cO(\iota Y) \ar[r]^i
&
\Top
}
\]
in which $\cU$ is a countable ordinary
poset.  Suppose the functor $\ov{p}\colon \cU^\triangleright \to\cO(\iota Y)$, determined by $\infty\mapsto \iota Y$, is a colimit diagram. There is a unique (up to unique isomorphism) extension $\ov{q}\colon \cU^\triangleright \to \snglr_n$ of $q$ for which $i\ov{p} = \iota \ov{q}$.  

\end{enumerate}
\end{lemma}

\begin{proof}
All points are routine except possibly point~(7).  This will be examined in section~\textsection\ref{premanifolds} to come.  In the language there, there is an immediate extension $\widetilde{q}\colon \cU^\triangleright \to \psnglr_n$ to \emph{pre-}singular $n$-manifolds where the value $\widetilde{q}(\infty) = (\iota Y, \cA')$ with $\cA= \bigcup_{u\in \cU} \bigl(\ov{p}(u\to \infty)\bigr)_\ast \cA_{q(u)}$ where $\cA_{q(u)}$ the atlas of $q(u)$.  Then appeal to Corollary~\ref{maximal-atlas'}.
\end{proof}

One also has the basic operations desirable from basic manifold theory:
\begin{lemma}\label{facts}
\begin{enumerate}
\item[]
\item The $\Top$-enriched category $\snglr_n$ admits pullbacks and $\iota$ preserves pullbacks.  

\item The triple $(\snglr_n,\sqcup, \emptyset)$ is a $\Top$-enriched symmetric monoidal category over $(\Top,\amalg,\emptyset)$.  

\item\label{C}
There is an injective faithful embedding 
\[
C\colon \snglr_{n , k}^c \to \bsc_{n+1,k+1}
\]
over the open cone functor on $\Top$.  

\item\label{deleted}  There is a standard natural transformation 
\[
R_{>0} \to C
\]
of continuous functors $\snglr^c_{n , k} \to \snglr_{n+1,k+1}$ which lies over $\RR_{>0} \times - \to C-$.  

\item\label{star} For $P$ and $Q$ topological spaces, define the \emph{join} of $P$ and $Q$ to be the space $P\star Q = (P\times \Delta^1\times Q)_{/\sim}$, where $\bigl(p,(1,0)\bigr)\sim \bigl(p',(1,0)\bigr)~,~ \bigl(q,(0,1)\bigr)\sim \bigl(q',(0,1))\bigr)$.  The operation $-\star-$ is functorial in each variable and is coherently associative and commutative.  
There is a continuous functor 
\[
\snglr^c_{m,j}\times \snglr^c_{n , k} \xra{\star} \snglr^c_{m+n+1,j+k+1}
\]
over the join operation of topological spaces.  

\item\label{cross} There is a continuous functor
\[
\bsc_{m,j}\times\bsc_{n , k} \xra{\times} \bsc_{m+n,j+k}.
\]
over product of topological spaces.  

\item\label{cross'} 
There is a continuous functor
\[
\snglr_{m,j} \times \snglr_{n , k} \xra{\times} \snglr_{m+n,j+k}
\]
over product of topological spaces.

\end{enumerate}
\end{lemma}

\begin{proof}
All the proofs are routine so we only indicate the methods.  
\begin{enumerate}
\item Consider a pair of morphisms $X\xra{f} Z$ and $Y\xra{g} Z$. The pullback is $X\times_Z Y = (\iota X\times_{\iota Z} \iota Y, \cA')$ where $\cA'=\{(U,\phi)\mid (U,h\phi)\in \cA\}$ where $\cA$ is the atlas of $Z$ and $h\colon \iota X\times_{\iota Z} \iota Y \to \iota Z$ is the projection.  

\item Immediate.  

\item This is given by $X \mapsto U^{n+1}_X$.  
\item This is induced from the open embedding $\RR_{>0} \subset \RR$.
\item The join $\iota X\star \iota Y$ admits an open cover 
\[
\xymatrix{
\RR\times (\iota X\times \iota Y)  \ar[r]  \ar[d]  
&
\iota X\times\bigl(C(\iota Y)\bigr) \ar[d]
\\
\bigl(C(\iota X)\bigr)\times \iota Y  \ar[r]
&
\iota X\star \iota Y 
}
\]
which can be lifted to a diagram of singular manifolds witnessing an open cover, thereby generating a maximal atlas (see \textsection\ref{premanifolds} for how this goes).   
\item This is given inductively by the expression $(X,Y)\mapsto U^{m+n+1}_{X\star Y}$ with base case $(\RR^m,\RR^n)\mapsto \RR^{m+n}$ upon observing the formula $CP \times CQ\cong C(P\star Q)$ which is implemented by, say, the assignment $(s,p;t,q)\mapsto (s^2+t^2, q , (\frac{s-t+1}{2},\frac{t-s+1}{2}),q)$.  
\item An atlas for $X\times Y$ is the collection $\{(U', \phi')\}$ consisting of those pairs for which for each pair of morphisms $U\xra{f} X$ and $V\xra{g} Y$ for which $(f\times g)(\iota U \times \iota V) \subset \phi'(U')$ the composite $(\phi')^{-1} \circ (f\times g) \colon U\times V \to U'$ is a morphism of $\bsc_{m+n}$.
\end{enumerate}
\end{proof}

\begin{lemma}\label{local-basis}
For each $p\in \RR^{n-k}\subset \iota U$ there is a continuous map $\beta\colon \RR_{\geq 0} \to \bsc_n(U,U)$
for which 
$\beta_0=1_U$,
$\beta_t\circ \beta_s = \beta_{s+t}$, and
$\{\beta_t(\iota U)\mid t\in \RR_{\geq 0}\}$ is a local basis for the topology around $p \in \iota U$.  
\end{lemma}
\begin{proof}
By translation, assume $p=0$.  Using classical methods, choose such a continuous map $\beta'\colon \RR_{\geq 0} \to \Emb(\RR, \RR)$ for which $\beta_t$ restricts as the identity map on $\RR_{\leq 0}$ for each $t\in \RR_{\geq 0}$.  The lemma follows upon the homeomorphism $\iota U= \RR^{n-k}\times C(\iota X) \approx C(S^{n-k-1}\star \iota X) = \bigl(\RR \times (S^{n-k-1}\star \iota X)\bigr)_{/\sim}$.  

\end{proof}

\begin{proof}[Proof of Lemma~\ref{basic-full}]
Use induction on $k$.  For $k=0$ this is the inclusion of the full subcategory of $\mfld_n$ spanned by $\RR^n$.  
Let $U\in \bsc_{\leq n , k}$.  We wish to construct an object $\snglr_{\leq n , k}$ associated to $U$.  Inductively, we can assume $U = U^n_X$ with $X\in \snglr_{k-1,k-1}^c$.   So $\iota U = \RR^{n-k}\times C(\iota X)$.  Define the object $(\iota U , \cA)\in \snglr_{\leq n , k}$ where $\cA=\{(U',\phi)\mid U'\xra{\phi} U\in \bsc_{\leq n , k}\}$.   Clearly $\iota U$ is second countable and Hausdorff and $\cA$ is an open cover.  

To show $\cA$ is an atlas it is sufficient to prove that the collection $\{f(V)\mid V\xra{f} U \in \bsc_{\leq n , k}\}$ is a basis for the topology of $\iota U = \RR^n\times C(\iota X)$.  For this, proceed again by induction on $k$, the case $k=0$ being standard.  The general case following from the two scenarios.  Let $p\in O\subset \RR^{n-k}\times C(\iota X)$ be neighborhood.
\begin{itemize}
\item Suppose $p\in \RR^{n-k}$. There is a morphism $U\xra{f} U$ for which $p\in f(\iota U)\subset O$, this morphism given by choosing a smooth self-embedding of $\RR^{n-k}\times \RR_{\geq 0}$ onto an arbitrarily small neighborhood $(p,0)$.  
\item Suppose $p\in \RR^{n-k}\times \RR_{>0} \times \iota X$.  Because $X$ has an open cover by the sets $\phi(\iota U')$ with $U'\in \bsc_{k-1,k-1}$ then we can reduce to the case $p\in \RR^{n-k}\times \RR_{>0} \times \iota U'$ which follows by induction.
\end{itemize}

To show $\cA$ is maximal we verify that an open embedding $f\colon \iota U \to \iota V$ is in $\bsc_{\leq n , k}$ if for each $p\in \iota U$ there is a diagram $U\xla{g} W \xra{h} V$ in $\bsc_{\leq n , k}$ such that $h=fg$ and $p\in g(\iota W)$.  We prove this by induction on the depth $j$ of $V$ with the case $j=0$ being classical -- a continuous map is smooth if and only if it is smooth in a neighborhood of each point in the domain.  If the depth of $U$ is strictly less than that of $V$, then the result follows by induction upon inspecting the definition of a morphism in $\bsc_n$.  Suppose the depth of $U$ equals the depth of $V$ which equals $k$.  Write $U=U^n_X$ and $V=V^n_Y$.  By definition, there is a lift to a morphism $\w{f} \colon R^{n-j}R X \to  R^{n-j}R Y$ of singular $n$-manifolds of depth strictly less than $k$.  

This assignment $U\mapsto (\iota U,\cA)$ is evidently continuously functorial over $\Top$ and therefore is faithful.  We denote this object $(\iota U, \cA)$ again as $U$.  
If $f\in \Top(\iota U,\iota U)$ lies in $\snglr_{\leq n , k}(U,U)$ then necessarily the identity morphism $g(U=U)\in \cA$ and it follows that $g$ lies in $\bsc_{\leq n , k}$.  So $\bsc_{\leq n , k} \to \snglr_{\leq n , k}$ is full.

Clearly this embedding $\bsc_{\leq n , k}\subset \snglr_{\leq n , k}$ respects the structure functors $\iota $ and $R^j$.

\end{proof}

\begin{proof}[Proof of Lemma~\ref{tautological}]
This proof of the first assertion is typical of the arguments to come, so we give it in detail.
We simultaneously establish both inclusions using induction on $k$.  
While the base case should be $k=-1$, this case is too trivial to learn from -- the assertion is obviously true.  Lets examine the base case $k=0$.  Then $\bsc_{\leq n , 0}$ has one object $U^n_{\emptyset^{-1}}$ corresponding to the unique object $\emptyset \in \snglr_{-1,-1}^c$.  The space of endomorphisms of this object is $\Emb(\RR^n,\RR^n)$, with composition given by compositing of embeddings.  
Assign to this object the object $U^{n+r}_{\emptyset^{r-1}}\in \bsc_{\leq n +r , r}$ corresponding to $\emptyset\in \snglr_{\leq r-1, r-1}^c$.  Notice that $\iota U^n_{\emptyset^{-1}} = \RR^n \times C(\iota \emptyset^{-1}) = \RR^n = \RR^n\times C(\iota \emptyset^{r-1}) = \iota U^{n+r}_{\emptyset^{r-1}}$.  For the assignment on spaces of morphisms, observe the identifications $\bsc_{\leq n,0}(U^n_{\emptyset^{-1}},U^n_{\emptyset^{-1}}) =  \Emb(\RR^n,\RR^n) = \w{\bsc}_{\leq n +r,r}(U^{n+r}_{\emptyset^{r-1}},U^{n+r}_{\emptyset^{r-1}})  =  \bsc_{\leq n +r ,r}(U^{n+r}_{\emptyset^{r-1}},U^{n+r}_{\emptyset^{r-1}})$.  These assignments obviously describe a continuous functor over $\Top$ which is an injection on objects and an isomorphism on morphism spaces.  

Assume that the tautological fully faithful inclusion $\bsc_{\leq n , k}\subset \bsc_{\leq n+r,k+r}$.  
For now, denote this assignment on objects as $U\mapsto U'$, and on morphisms (justifiably) as  $f\mapsto f$.  
Let us now define the tautological inclusion $\snglr_{\leq n , k} \subset \snglr_{\leq n +r , k+r}$.  
Assign to $X=(\iota X, \cA) \in \snglr_{\leq n , k}$ the object $X'=(\iota X, \cA')\in \snglr_{\leq n +r,k+r}$ with the same underlying space but with atlas $\cA'=\{(U',\phi)\}$.  
The assignment on morphism spaces is tautological (hence the occurrence of the term).  
It is immediate that this inclusion $\snglr_{\leq n,k}\subset \snglr_{\leq n+r,k+r}$ takes place over $\Top$ and under $\bsc_{\leq n ,k}\subset \bsc_{n+r,k+r}$.  
Obviously, $X\in \snglr_{\leq n , k}^c$ if and only if $X'\in \snglr_{\leq n +r , k+r}^c$.  

By induction, assume the first statement of the lemma has been proved for any $k'<k$.  
We now establish the tautological inclusion $\bsc_{\leq n , k} \subset \bsc_{\leq n+r,k+r}$.
Let $U\in \bsc_{\leq n, k}$.  
If $U$ has depth less than $k$, then $U'$ has already been defined.  
Assume that the depth of $U$ is $k$.  
Then $U$ is canonically of the form $U=U^n_X$ for some $X\in \snglr_{k-1,k-1}^c$.  
Then $X'\in \snglr_{\leq k+r-1 , k+r-1}^c$ has been defined.  
Define $U' = U^{n+r}_{X'}\in \bsc_{\leq n+r,k+r}$.  
Notice that $\iota U' = \RR^{(n+r)-(k+r)}\times C(\iota X') = \iota U$.  
Let $f\in \bsc_{\leq n,k}(U,V)$ be a morphism.  
If the depth of $U$ or $V$ is less than $k$, then by induction we can regard $f$ as a morphism $U' \xra{f} V'$.  Assume both $U$ and $V$ have depth exactly $k$.  Write $U=U^n_X$ and $V=V^n_Y$ and choose a representative $(\w{f},h)$ of $f=[(\w{f},h)]$.  
Then both $\w{f}$ and $h$ are morphisms between singular manifolds of depth less than $k$.  
As so, we can regard them as morphisms $R^{(n+r)-(k+r)}RX' \xra{\w{f}} R^{(n+r)-(k+r)}RY'$ and $\RR^{(n+r)-(k+r)} \xra{h} \RR^{(n+r)-(k+r)}$ which in fact represent a morphism $U' \xra{f} V'$.  
This assignment on morphism spaces is tautologically a homeomorphism.
This completes the first statement of the proof.    

The second statement follows by a similar induction based on the fact the space of morphisms of topological spaces $Z\to \emptyset$ is empty unless $Z=\emptyset$ in which case it is a point.  

\end{proof}

\subsubsection{Specifying singularity type}
Given an arbitrary singular $n$-manifold $X$, one must contemplate `smooth' maps $\iota U\to \iota X$ as $U$ ranges over all singularity types (basics) of dimension $n$.  
Even when $n=1$, this is a large amount of data---for instance there is a $U$ for each natural number, corresponding to the valency of a graph's vertex. However, often one enforces control on allowed singularity types, considering manifolds which only allow for a certain class of singularities.  

Let $\sC$ be a $\Top$-enriched category.  A subcategory $\sL\subset \sC$ is a \emph{left ideal} if $d\in \sL$ and $c\xra{f}d \in \sC$ implies $c\xra{f}d\in \sL$.  Notice that a left ideal is in particular a full subcategory.

Let $\sB\subset \bsc_n$ be a left ideal.  Let $X$ be a singular $n$-manifold.  Denote by $X_\sB\subset X$ the sub-singular $n$-manifold canonically associated to the open subset $\iota X_\sB = \bigcup \phi(\iota U)\subset \iota X$ where the union is over the set of pairs $\{(U,\phi)\mid U\xra{\phi} X \text{ with } U\in \sB\}$. Because $\sB$ is a left ideal, the collection $\{\phi(\iota U)\mid \sB\ni U\xra{\phi} X\}$ is a basis for the topology of $\iota X_{\sB}$.  
Conversely, given a singular $n$-manifold $X$, the full subcategory $\sB_X\subset \bsc_n$, spanned by those $U$ for which $\snglr_n(U,X)\neq \emptyset$, is a left ideal.  

\begin{remark}
It is useful to think of a left ideal $\sB\subset \bsc_n$ as a list of $n$-dimensional singularity types, this list being \emph{stable} in the sense that a singularity type of an arbitrarily small neighborhood of a point in a member of this list is again a member of the list.  
\end{remark}

\begin{definition}\label{def:left-ideal}
Let $\sB\subset \bsc_n$ be a left ideal.  
A $\sB$-manifold is a singular $n$-manifold for which $X_\sB\xra{\cong} X$ is an isomorphism.  Equivalently, a singular $n$-manifold $X$ is a $\sB$-manifold if $\snglr_n(U,X) = \emptyset $ whenever $U\notin \sB$.  

\end{definition}

\begin{example}
The inclusion $\bsc_{n , k}\subset \bsc_n$ is a left ideal for each $0 \leq k \leq n$ and a $\bsc_{n , k}$-manifold is a singular $n$-manifold of depth at most $k$.  

\end{example}

\begin{example}
Let $\sB\subset \bsc_1$ be the full subcategory spanned by $\RR$ and $U_{\{1,2,3\}}$.  Then $\sB$ is a left ideal and a $\sB$-manifold is 
a (possibly open) graph whose vertices (if any) are exactly trivalent.  

\end{example}

\begin{example}\label{boundary}
Let $\sD_n^\partial\subset \bsc_n$ be the full subcategory spanned by the two objects $\RR^n$ and $U^n_{\ast}$ whose underlying space is $\RR^{n-1} \times \RR_{\geq 0}$.  This full subcategory is indeed a left ideal.  A $\sD_n^\partial$-manifold is precisely a smooth $n$-manifold with boundary.  

As a related example, let $\sD_n^{\un{\partial}}\subset \bsc_n$ be the full subcategory spanned by the objects $\{U^n_{\Delta^{k-1}}\}_{0\leq k \leq n}$ where it is understood that $\Delta^{-1}=\emptyset$ and $U^n_\emptyset = \RR^n$.  Because a neighborhood of any point in $\Delta^{k-1}$ is of the form $\RR^j \times C(\Delta^{j-1})$ for $j<k-1$, then this full subcategory is a left ideal.  A $\sD_n^{\un{\partial}}$-manifold is (one definition of) an $n$-manifold with corners.

\end{example}

\begin{example}[Embedded submanifolds]\label{example:Edn'}

Let $\sD_{n,d}^{\mathsf{Knk}}\subset \bsc_n$ be the left ideal with the two objects $\{\RR^n~,~U^n_{S^{n-d-1}}\}$.  The underlying space $\iota U^n_{S^{n-d-1}} = \RR^k\times CS^{n-d-1}$ 
is incidentally homeomorphic to $\RR^n$. 
However, the morphism spaces are as follows:

\begin{itemize}
\item $\sD_{n,d}^{\mathsf{Knk}}(\RR^n,\RR^n) = \Emb(\RR^n,\RR^n)$ -- the space of smooth embeddings,
\item $\sD_{n,d}^{\mathsf{Knk}}(U^n_{S^{n-d-1}},\RR^n) = \emptyset $, 
\item $\sD_{n,d}^{\mathsf{Knk}}(\RR^n,U^n_{S^{n-d-1}}) = \Emb(\RR^n,\RR^n \smallsetminus \RR^d)$ -- the space of smooth embeddings which miss the standard embedding $\RR^k\times\{0\}\subset \RR^n$,
\item $\sD_{n,d}^{\mathsf{Knk}}(U^n_{S^{n-d-1}},U^n_{S^{n-d-1}}) \subset  \Emb^{0}(\RR^n,\RR^n)$ -- the subset of those continuous embeddings $f$ which fit into a diagram of embeddings
\[
\xymatrix{
\RR^d  \ar[r]  \ar[d]^{f_|}
&
\RR^n \ar[d]^f
&
\RR^n\smallsetminus \RR^d  \ar[d]^{f_|} \ar[l]
\\
\RR^d  \ar[r] 
&
\RR^n
&
\RR^n\smallsetminus \RR^d \ar[l]
}
\]
in where the left and right vertical maps are smooth, and the middle vertical map is `conically smooth' -- by this we mean there is a diagram of continuous maps
\[
\xymatrix{
\RR^k\times (\RR\times S^{n-k-1})  \ar[r]^{\w{f}}  \ar[d]
&
\RR^k\times (\RR\times S^{n-k-1})   \ar[d]
\\
\RR^n  \ar[r]^f  
&
\RR^n
}
\]
such that the top horizontal map is smooth and each vertical map is the composite $\RR^k \times (\RR \times S^{n-k-1}) \to \RR^k \times C(S^{n-k-1}) \cong \RR^k \times \RR^{n-k}$. The first map is the quotient map to the open cone, and the last is the standard polar coordinates homeomorphism.
This set is topologized with the quotient topology on the evident subspace of smooth embeddings $\Emb(\RR^k\times \RR \times S^{n-k-1}, \RR^k\times \RR\times S^{n-k-1})$ with the weak $C^\infty$ Whitney topology.

\noindent
Evident from this description, there is the map of topological monoids $\sD^{\mathsf{Knk}}_{n,k}(U^n_{S^{n-k-1}}, U^n_{S^{n-k-1}}) \to \Emb^{\sf PL}(\RR^n,\RR^n)$, which extends to a topological functor
\[
\sD^{\sf Knk}_{n,d} \to \End^{\mfld_n^{\sf PL}}(\RR^n,\RR^n)~.
\]
In particular, there results a topological functor
\[
\mfld(\sD^{\sf Knk}_{n,d}) \to \mfld_n^{\sf PL} \times \mfld_d~.
\]

\end{itemize}
A $\sD^{\sf {Knk}}_{n,d}$-manifold is the data of 
\begin{itemize}
\item an $n$-dimensional $\mathsf{PL}$-manifold $M$,
\item a properly embedded $k$-dimensional $\mathsf{PL}$-submanifold $L\subset M$,
\item a smooth structure on $L$, 
\item a smooth structure on $M\smallsetminus L$, 
\item a smooth structure on the link $\mathsf{Link}_{L\subset M}$ for which the maps $L\leftarrow \mathsf{Link}_{L\subset M} \hookrightarrow M\setminus L$ are smooth. 
\end{itemize}
We refer to this data as a \emph{kink submanifold $L\subset M$}.  
An example of such data comes from a properly embedded smooth $k$-manifold in a smooth $n$-manifold, the smooth structure on the link amounting to the existence of tubular neighborhoods. 
Not all $\sD_{n,d}^{\mathsf{Knk}}$-manifolds are isomorphic to ones of this form -- this difference will be addressed as Example~\ref{example:Ekn}.  
\end{example}

\subsection{Premanifolds and refinements}\label{premanifolds}
Assume we are given a Hausdorff topological space $\iota X$ with a countable open cover by singular $n$-manifolds. Further assume the transition maps are morphisms in $\snglr_n$. While the resulting atlas is not a priori maximal, we would still like to accommodate such objects.

\begin{definition}
In Definition~\ref{singular-manifolds} the hypothesis {\bf Maximal} can be dropped. The resulting $\Top$-enriched category $\psnglr_{n , k}$ is called the {\em category of singular premanifolds}.
Evidently, there is the fully faithful inclusion $\snglr_{n , k}\subset \psnglr_{n , k}$ over the faithful functors $\iota$ to $\Top$.   
\end{definition}

\begin{definition}\label{refinement}
A \textit{refinement} is a morphism $\dddot{X} \xra{r} X$ of singular $n$-premanifolds for which the map of underlying spaces $r \colon \iota \dddot{X} \xrightarrow{\cong} \iota X$ is a homeomorphism.  
Refinements are clearly closed under composition.  
Define the \emph{category of refinements} as the subcategory
\[
\rfn_n \subset \psnglr_n
\]
consisting of the same objects and with morphisms the refinements.  
\end{definition}

Lemma~\ref{basic-full}, Lemma~\ref{about-X}, and Lemma~\ref{facts} are valid for $\psnglr$ in place of $\snglr$.

\begin{example}\label{cover-refinement}
Let $X=(\iota X,\cA)$ be a singular manifold.  An open cover $\cU$ of $\iota X$ canonically determines a refinement $X_\cU \to X$ where the atlas of $X_\cU$ is the subset $\cA_\cU = \{(U,\phi)\mid \phi(U)\subset O\in \cU\}\subset \cA$.  In this way it is useful to regard the data of a refinement of a singular manifold as an open cover.  
\end{example}

\begin{lemma}\label{refs-pullback}
Consider a pullback diagram in $\psnglr_n$
\[
\xymatrix{
{\dddot{Y}} \ar[r]^{r_|}  \ar[d]^{\dddot{f}}
&
Y  \ar[d]^f
\\
{\dddot{X}}  \ar[r]^r
&
X
}
\]
in where $r$ is a refinement.  Then $r_|$ is a refinement as well.
\end{lemma}
\begin{proof}
This is immediate from the description of the pullback in the proof of Lemma~\ref{facts}.  
\end{proof}

There is the evident fully faithful inclusion $\snglr_n \subset \psnglr_n$.  

\begin{prop}\label{maximal-atlas}
There is a localization
\[
\psnglr_n \leftrightarrows \snglr_n
\]
where the right adjoint is the inclusion.  
\end{prop}
\begin{proof}
Let $\dddot{X}=(\iota X,\cA)$ be a singular premanifold.  The value of the left adjoint on $\dddot{X}$ is the singular manifold $X=(\iota X,\cA')$ where $\cA'$ consists of those $(U,\phi)$ for which there is a refinement $\dddot{U} \xra{r} U$ and a morphism $\dddot{\phi}\colon \dddot{U} \to X$ over $\phi$.  
Clearly $\cA$ is an open cover of $\iota X$.  Let $(U,\phi),(V,\psi)\in \cA$.  The underlying space of the pullback of singular premanifolds $\dddot{U}\times_{\dddot{X}} \dddot{V}$ is $\phi(\iota U) \cap \psi(\iota V)$.  It follows that $\cA$ is an atlas and that $\dddot{X} \to X$ is a refinement.  For $X\to X'$ a refinement, then $\dddot{X} \to X \to X'$ is a refinement.  From Lemma~\ref{refs-pullback}, for each morphism $U\xra{\phi} X'$ the morphism from the pullback $U\times_{X'} \dddot{X} \to U$ is a refinement.  It follows that $(U, \phi)\in \cA$.  This proves that $\cA$ is maximal.

Suppose $\dddot{X} \to X'$ be a refinement.  For any $U\xra{\phi} X'$ the morphism $U\times_{X'} \dddot{X} \to U$ is a refinement.  It follows that $(U,\phi)\in \cA$ and thus the continuous map $\iota X' \to \iota X$ is a refinement of singular premanifolds.  This proves that the adjunction is a localization.  
\end{proof}

\begin{cor}\label{maximal-atlas'}
For each singular premanifold $\dddot{X}$  there exists an essentially unique singular manifold $X$ equipped with a refinement $\dddot{X} \to X$. 
\end{cor}

Hereafter, unless the issue is topical we will not distinguish in notation or language between a singular-premanifold and its canonically associated singular manifold.

\subsection{Stratifications}\label{stratifications}
For $P$ a poset, by a {\em $P$-stratified space $Z$}, or simply a stratified space, we mean a continuous map $Z \to P$ where $P$ is given the poset topology (so closed sets are those which are upward closed). Here we explain how a singular $n$-manifold can be viewed as a $[n]$-stratified space.

To do this, we define a continuous functors $(-)_j \colon \snglr_{n} \to \snglr_{\leq j}$ for each $n$, equipped with continuous natural transformations by closed embeddings $\iota (-)_j \hookrightarrow \iota$ for each integer $j$.  
For $X$ a singular $n$-manifold we refer to $X_j$ as its  \emph{$j^{\rm th}$ stratum}.  

We will accomplish this by double induction, first on the parameter $n$, then on the parameter $k$. 
For $n<0$ then $\snglr_n=\{\emptyset^n\}$ and declare $(\emptyset^n)_j=\emptyset^j\in\snglr_{\leq j}$, the closed inclusion is obvious.

Assume $(-)_j$ and the closed inclusion $\iota (-)_j \to \iota$ have been defined for $\snglr_{n'}$ whenever $n'<n$.  We now define these data for $\snglr_n$.  
We do this by induction on $k$.
For $k=-1$ then $\snglr_{n,-1}=\{\emptyset^n\}$ and declare $(\emptyset^n)_j = \emptyset^j\in \snglr_{\leq j}$.

Assume $(-)_j$ and the closed inclusion $\iota (-)_j\to \iota$ have been defined on $\snglr_{n,k'}$ whenever $k'<k$.  We now define these data for $\snglr_{n,k}$.  
We first do this for $\bsc_{n,k}$.  
Let $U=U^n_Y\in \bsc_{n,k}$.  
If $U$ has depth less than $k$ then $U_j$ and the closed inclusion $\iota U_j \to \iota U$ have already been defined.  
Assume the depth of $U$ equals $k$.
Define $U_j$ through the expression
\[
(U^n_Y)_j = U^j_{Y_{(k-1)-(n-j)}}~,
\]
the righthand side of which has been defined by induction since the dimension of $Y$ is $k-1<n$. 
We point out the following slippery cases: if $(k-1)-(n-j) = -1$, righthand side is $U^j_{\emptyset^{-1}}=\RR^j$; if $(k-1)-(n-j)<-1$, we invoke our convention that the righthand side is $\emptyset^j$ the empty $j$-manifold (see Terminology~\ref{terms}). 
This assignment $(-)_j$ is evidently functorial and continuous on $\bsc_{n,k}$.  
By our inductive assumption, there is a closed embedding $\iota Y_{(k-1)-(n-j)} \hookrightarrow \iota Y$, which in turn induces the closed inclusion
\[
\iota U_j = \iota U^j_{Y_{(k-1)-(n-j)}} = \RR^{n-k}\times C(\iota Y_{(k-1)-(n-j)}) \hookrightarrow \RR^{n-k}\times C(\iota Y) = \iota U^n_Y = \iota U~.  
\]

To an arbitrary singular manifold $(\iota X, \cA)$ we assign the pair $(\iota X_j,\cA_j)$ where $\iota X_j = \bigcup \phi(U_j)\subset \iota X$ is the union over $(U,\phi)\in \cA$, and $\cA_j = \{(U_j,\phi_{|U_j})\mid (U,\phi)\in \cA'\}$.  It is routine to check that $(\iota X_j, \cA_j)$ is a singular $j$-manifold.  
This assignment $(-)_j$ is evidently functorial and continuous on $\snglr_{n,k}$.  
The manifest inclusion $\iota X_j \subset \iota X$ is closed because it is locally closed.

\begin{remark}
For $j\geq n$, the construction of the functor $\snglr_n \to \snglr_{\leq j}$ agrees with the tautological inclusion of Lemma~\ref{tautological}.  
\end{remark}

Clearly, $j\leq j'$ implies $\iota X_j\subset \iota X_{j'}$.  The open complement $\iota X_{j'} \smallsetminus \iota X_j$ canonically inherits the structure of a singular submanifold $X_{j, j'}\subset X_{j'}$ of dimension at most $j'$.  Tracing through the construction of $(-)_j$, the depth of $X_{j, j'}$ is at most $j'-j-1$.  
It follows that $X_{j-1,j}$ is a smooth $j$-manifold for each $j\leq n$. In particular if $X$ has depth $k$ then $X_{n-k}$ is a smooth $k$-manifold.

We defer the proof of the following proposition to~\S\ref{embeddings}.  It is in the proof of this proposition that the deliberate locally cone strucure is used.  To state the proposition requires some vocabulary.

\begin{terminology}
For the definition of a \emph{conically smooth} map referred to below, see~\S\ref{embeddings}.  
For the time being, think of a conically smooth map as the singular version of a smooth map between ordinary manifolds.  
The definition of a conically smooth fiber bundle follows exactly the definition of a smooth fiber bundle -- it is a conically smooth map $\partial E \to B$ which locally has the form of a projection $U \times Z \to U$ with transition maps by isomorphisms of $Z$. 
For $\partial E \to B$ a conically smooth fiber bundle with compact fibers, there is another conically smooth fiber bundle $C_B(\partial E)$ called the {\em fiberwise cone of $\partial E \to B$}. This is defined locally through the functorial construction $U\times Z\mapsto U\times C(Z)$.  
A fiberwise cone structure on a conically smooth fiber bundle $E\to B$ is a conically smooth fiber bundle $\partial E \to B$ with compact fibers, together with an isomorphism $C_B(\partial E) \cong E$  of conically smooth fiber bundles over $B$.   
Note that a fiberwise cone structure on $E\to B$ determines a \emph{cone-section} $B\to E$ and an isomorphism $E\smallsetminus B \cong R(\partial E)$ over $B$.   
\end{terminology}

\begin{prop}\label{tubular-neighborhood}
Let $X$ be a singular $n$-manifold of depth $k$.  
Then there is a fiberwise cone $\widetilde{X}_{n-k} \to X_{n-k}$ and a morphism $\w{X}_{n-k} \xra{f} X\in \snglr_n$
for which the composition with the cone-section is the standard closed embedding $\iota X_{n-k}\to \iota X$.  
\end{prop}

\begin{cor}\label{pieces}
Let $X$ be a singular $n$-manifold of depth $k$. There is a pullback diagram in $\snglr_n$ 
\[
\xymatrix{
R(\partial \widetilde{X}_{n-k})  \ar[r]  \ar[d]
&
X\smallsetminus X_{n-k}  \ar[d]
\\
\widetilde{X}_{n-k} \ar[r]
&
X
}
\]
whose image under $\iota$ is a pushout diagram.

\end{cor}

We have constructed for each singular $n$-manifold $Y$ a continuous map 
$S\colon \iota Y \to [n]$ determined by $S^{-1} \{i\mid i\leq j\} = \iota Y_j$.  
We summarize the situation as the following proposition.  

\begin{prop}\label{stratified}
There is a standard factorization
\[
\xymatrix{
\snglr_n  \ar[rr]^\iota   \ar[dr]^{[\iota]}
&
&
\Top
\\
&
\Top_{[n]}  \ar[ur]
&
}
\]
through $[n]$-stratified topological spaces. The unlabeled arrow is given by forgetting the stratification.  
Moreover, the stratified space $[\iota]X$ is conically stratified and the $j^{\rm th}$ open stratum $\iota X_{j-1, j}$ is a smooth $j$-manifold.  
\end{prop}

\subsection{Collar-gluings}
We highlight a class of diagrams which will play an essential role to come. Fix a dimension $n$.  

\begin{definition}\label{collar-gluing}
A \emph{collar-gluing} is a singular $(n-1)$-manifold $V$ together with a pullback square of singular $n$-manifolds
\[
\xymatrix{
RV  \ar[r]  \ar[d]
&
X_+  \ar[d]^{f_+}
\\
X_-\ar[r]^{f_-}  
&
X
}
\]
for which $f_-(\iota X_-)\cup f_+(\iota X_+) = \iota X$.  
We denote the data of a collar-gluing as $X=X_-\cup_{RV} X_+$. 
\end{definition}

Manifestly, a collar-gluing $X=X_-\cup_{RV} X_+$ 
determines the open cover $\{f_\pm (\iota X_\pm)\}$ of $\iota X$ and thus canonically determines a refinement.

Let us temporarily denote by 
\[
\bsc_n \subset \snglr_n^\cI\subset \snglr_n
\]
the smallest full subcategory
for which $X=X_-\cup_{RV} X_+$ with $X_\pm , RV\in  \snglr_n^\cI$ implies $X \in \snglr_n^\cI$.  Explicitly, an object of $\snglr_n^\cI$ is a singular $n$-manifold which can be written as a finite iteration of collar-gluings.

\begin{definition}[Finite singular manifolds]\label{finite}
Say an atlas $\dddot{\cA}$ of a singular manifold $X$ is \emph{finite} if 
\begin{enumerate}
\item the set $\dddot{\cA}$ is finite,
\item for each pair $(U,\phi),(V,\psi)\in \dddot{\cA}$ and each $p\in \iota X$ there is an element $(W,\eta)\in \dddot{\cA}$ with $p\in \eta(\iota W) \subset \phi(\iota U) \cap \psi(\iota V)$~.  
\end{enumerate}
Say a singular manifold is \emph{finite} if it admits a finite atlas.
Denote by $\mathsf{Snglr}_n^{\sf fin}\subset \mathsf{Snglr}_n$ the full subcategory spanned by the finite singular manifolds. Similarly $\mathsf{pSnglr}_n^{\sf fin}\subset \mathsf{p
Snglr}_n$ is the full subcategory spanned by finite singular pre-manifolds.
\end{definition}

\begin{prop}\label{cw}
The underlying space $\iota X$ of a finite singular manifold $X$ is homotopy equivalent to a finite CW complex.  
\end{prop}
\begin{proof}
By induction on the depth $k$ of $X$.  The case $k=0$ is handled by way of standard Morse theory.  From Corollary~\ref{pieces} we can write $X = \widetilde{X}_{n-k} \cup_{R(\partial \widetilde{X}_{n-k})} X\smallsetminus X_{n-k}$.  In particular, there is the pushout diagram
\[
\xymatrix{
\{0\}\times \iota (\partial \widetilde{X}_{n-k})  \ar[r]  \ar[d]
&
\iota (X\smallsetminus X_{n-k})  \ar[d]
\\
\iota \widetilde{X}_{n-k} \ar[r]
&
\iota X
}
\]
is a homotopy pushout.
There is a deformation retraction of $\iota \widetilde{X}_{n-k}$ onto $\iota X_{n-k}$.  By induction all of $\iota \widetilde{X}_{n-k}$, $\iota (X\smallsetminus X_{n-k})$, and $\iota (\partial \widetilde{X}_{n-k})$ admit a finite CW complex structure.  The claim follows.  

\end{proof}

Clearly, if $X_\pm$ and $V$ are each finite then so is $X$.  Therefore $\snglr_n^\cI\subset \snglr_n^{\sf fin}$.  

\begin{theorem}\label{collar=fin}
The inclusion $\snglr_n^\cI \xra{\cong} \snglr_n^{\sf fin}$ is an equivalence of categories.  
\end{theorem}

\begin{proof}
Let $X = (\iota X,\cA)$ be a singular $n$-premanifold with $\cA$ finite.  We prove that $X\in \snglr_n^\cI$ by induction on the depth of $X$.  If $X$ has depth zero, then $X$ is an ordinary smooth $n$-manifold and the result follows from classical results.  For instance, $X$ admitting a finite atlas implies it is the interior of a compact $n$-manifold with boundary.  Then use Morse theory. 

Suppose $X$ has depth $k$.  From Corollary~\ref{pieces}, the diagram
\[
\xymatrix{
R(\partial \widetilde{X}_{n-k})  \ar[r]  \ar[d]
&
X\smallsetminus X_{n-k}  \ar[d]
\\
\widetilde{X}_{n-k} \ar[r]
&
X
}
\]
describes a collar-gluing $X = \widetilde{X}_{n-k}\cup_{R(\partial \widetilde{X}_j)} X\smallsetminus X_{n-k}$.  The collection 
\[
\{(U_{n-k},\phi_k )\mid (U,\phi)\in \cA \text{ and } \phi(U)\cap X_{n-k}\neq \emptyset \}
\]
is a finite atlas for $X_{n-k}$.  For each $p\in \iota X_{n-k}$ there is a morphism $U^n_Z \xra{\psi} \widetilde{X}_{n-k}$ for which $p\in\psi(U^n_Z)$.  Because $Z$ is compact, it too admits a finite atlas.  It follows that both $\widetilde{X}_{n-k}$ and $\partial \widetilde{X}_{n-k}$ admit finite atlases.  Likewise, the open cover $\{\phi(\iota U\smallsetminus \iota U_{n-k})\mid (U,\phi)\in \cA\}$ of $X\smallsetminus X_{n-k}$ admits a finite refinement by basics thereby exhibiting an atlas.  The result follows.  
\end{proof}

\subsection{Categories of basics}\label{categories-of-basics}
This section is the culmination of our presentation of the geometric objects studied in this article.  
Here we define a category of singular $n$-manifolds equipped with a specified structures, an object of which we call a \emph{structured singular manifolds}.   

For an ordinary smooth $n$-manifold, a \emph{structure} is a continuously varying $B$-structure on the fibers of the tangent bundle, where $B\to {\sf BGL}(\RR^n)$ is a fibration.  Singular manifolds allow for much more creativity in their choices of structure.  This is due to the fact that, while (the coherent nerve of) $\bsc_{n,0}$ is an $\infty$-groupoid (i.e., a Kan complex), the quasi-category $\bsc_n$ is far from being a groupoid.

The idea is that a \emph{structure} on a singular manifold $X$ is a lift of the tangent classifier $X \xra{\tau_X} \bsc_n$ to a right fibration over $\bsc_n$. To implement this idea, we first establish the notion of a tangent classifier.

\subsubsection{Tangent classifiers}

For $\cC$ a quasi-category, let $\cP(\cC)$ denote the quasi-category of right fibrations over $\cC$.  
We often denote a right fibration $E\to \sC$ by $E$.
There is a Yoneda functor $\cC \to \cP(\cC)$, see~\cite{topos},~\textsection5.1.1 and \textsection5.1.3, given on vertices as $c\mapsto \widehat{c}:=(\cC_{/c} \to \cC)$.
This Yoneda functor is fully faithful (meaning the map of homotopy types $\cC(c,d)\to  \cP(\cC)(\widehat{c},\widehat{d})$ is an isomorphism).  

Recall our convention that we do not distinguish in notation between a $\Top$-enriched category and the quasi-category which is its coherent nerve.  
The inclusion $\bsc_n \subset \snglr_n$ induces a map of quasi-categories
$
\cP(\snglr_n) \to \cP(\bsc_n)
$
given by pullback $(E \to \snglr_n)\mapsto (\bsc_n\times_{\snglr_n} E)$.  Precomposing with the Yoneda map gives a map of quasi-categories $\widehat{(-)}\colon\snglr_n \to \cP(\bsc_n)$, denoted on objects as
\begin{equation}\label{hat}
X\mapsto (\widehat{X} \xra{\tau_X} \bsc_n)~.  
\end{equation}
We will often denote this value simply as $\widehat{X}$ and refer to $\tau_X$ as the \emph{tangent classifier}.   

Note there is an equivalence of Kan complexes $|\widehat{X}| \simeq \Sing(\iota X)$.
To see this, consider the sub-quasi-category of $\bsc_n$ whose edges $U\to V$ are morphisms that extend to embeddings of compact spaces $\ov{\iota U} \hookrightarrow \ov{\iota V}$. Over this lives a cofinal sub-quasi-category of $\widehat{X}$. Because each $\iota U$ is contractible,
\[
|\widehat{X}| = \colim(\widehat{X} \xra{\ast} \cS) \xla{\simeq} \colim(\widehat{X} \xra{\tau_X} \bsc_n\xra{\iota} \cS) \xra{\simeq} \Sing(\iota X)~.
\]

\begin{remark}
The situation $X\mapsto (\widehat{X} \xra{\tau_X} \bsc_n)$ is a generalization of a familiar construction in classical differential topology.  
Suppose $X$ is a smooth manifold. The map $\tau_X$ factors as $\widehat{X} \xra{\tau_X} \bsc_{n,0} \to \bsc_n$.  As we just witnessed, the quasi-category $\widehat{X}$ is a Kan complex and is equivalent to $\iota X$.  Moreover, $\bsc_{n,0}$ is a Kan complex and is equivalent to ${\sf BO}(n)$.  
Through these equivalences, the map $\widehat{X} \xra{\tau_X} \bsc_{n,0}$ is equivalent to the familiar tangent bundle classifier $X\xra{\tau_X} {\sf BO}(n)$.  

If $X$ is not smooth, then $\widehat{X}$ is not a priori a Kan complex and the map $\widehat{X} \to \bsc_n$ retains more information than any continuous map from the underlying space $\iota X$.  For instance, the map $\tau_X$ does not classify a fiber bundle per se since the `fibers' are not all isomorphic.  

\end{remark}

\begin{remark}
The maps of quasi-categories $\snglr_n \to \cP(\bsc_n)$ is very far from being fully faithful. For instance, in the case of smooth manifolds, this functor is equivalent to $\mfld_n \ra \cS_{/\BO(n)}$ to spaces over the classifying space of the orthogonal group.  Surgery theory provided (non-trivial) obstructions to lifting an object on the righthand side to an object on the left, as well as obstructions to comparing two such lifts.  

\end{remark}

\subsubsection{Categories of basics}

\begin{definition}\label{cat-o-basics}
A \emph{category of basics (of dimension $n$)} is a right fibration 
\[
\cB\to \bsc_n.
\]  
We refer to an object $U\in \cB$ as a \emph{basic}.
Let $\cB$ be a category of basics.  Define the quasi-category
\[
\mfld(\cB) = \snglr_n\times_{\cP(\bsc_n)} (\cP(\bsc_n)_{/\cB})~.
\]
and refer to its vertices as \emph{$\cB$-manifolds}.  Likewise, we refer to the vertices in the quasi-categoroy $\pmfld(\cB) = \psnglr_n\times_{\cP(\bsc_n)} \bigl(\cP(\bsc_n)\bigr)_{\cB}$ as $\cB$-premanifolds.  
We will use $\iota$ to also denote the composition $\iota \colon \mfld(\cB) \xra{\rm proj} \snglr_n \xra{\iota} \Top$.  
\end{definition}

Explicitly, a $\cB$-manifold is the data of a pair $(X,g)$ where $X$ is a singular $n$-manifold and $\widehat{X} \xra{g} \cB$ is a map of right fibrations over $\bsc_n$.  Unless the structure $g$ is notationally topical, we will denote a $\cB$-manifold simply as a letter $X$.

\begin{remark}
The straightening-unstraightening construction is an equivalence of quasi-categories between $\cP(\bsc_n)$ and the quasi-category $\mathsf{Fun}(\bsc_n^{\op},\cS)$ of functors to Kan complexes.  As so, the datum of a right fibration is equivalent to that of a map of quasi-categories $\bsc_n^{\op} \to \cS$.  For $\cB$ a category of basics, denote by $\Theta$ the corresponding functor.  We can reformulate the datum $g$ of a $\cB$-manifold $(X,g)$ as a point in the limit $g\in \lim(\widehat{X}^{\op} \xra{\tau_X} \bsc_n^{\op} \xra{\Theta} \cS)$.  That is, for each $U\to X$ an element $g_U\in \Theta(U)$ and for each $U\xra{f} V \to X$ a path between $g_U$ and $f^\ast g_V$, and likewise for each sequence $U_0\xra{f_1} \dots \xra{f_p} U_p \to X$.  
In the case that these paths (and maps from higher simplices) are constant, this is a strict notion of a structure.  
\end{remark}

\begin{example}
Let $\Theta\colon \bsc_n^{\op} \to \Top$ be a fibrant continuous functor. The unstraightening construction applied to $\Theta$ gives a right fibration 
\[
\cB_\Theta \to \bsc_n
\]
whose fiber over $U$ is a Kan complex equivalent to $\mathsf{Sing}\bigl(\Theta(U)\bigr)$.  

\end{example}

\begin{example}\label{over-X}
Let $X$ be a singular $n$-manifold.  Then $\widehat{X} \to \bsc_n$ is a category of basics.  The data of a singular $n$-manifold $Y$ together with a morphism $Y\xra{f} X$ determines the $\widehat{X}$-manifold $(\widehat{Y},\widehat{f})$.  
In general, it is not the case that every $\widehat{X}$-manifold arises in this way. 
\end{example}

Let $\cC$ be a quasi-category.  Say a sub-quasi-category $\cL \subset \cC$ is a \emph{left ideal} if the inclusion $\cL \to \cC$ is a right fibration.  Alternatively, say $\cL\subset \cC$ is a left ideal if it is a full subcategory and for each $c\in \cL$ the homotopy type $\cC(c',c)\neq\emptyset$ being non-empty implies $c'\in \cL$.  Being a full sub-quasi-category, a left ideal is determined by its set of vertices.  
Note that the coherent nerve of a left ideal $\sL \subset \sC$ of $\Top$-enriched categories is a left ideal of quasi-categories.

\begin{remark}
We point out a consistency of terminology.  
For $\sB\subset \bsc_n$ a left ideal of $\Top$-enriched categories, the two Definitions~\ref{def:left-ideal} and~\ref{cat-o-basics} of a $\sB$-manifold agree.  
\end{remark}

\begin{example}\label{example:left-ideal}
Let $\sB\subset\bsc_n$ be a left ideal of $\Top$-enriched categories. The map of coherent nerves
\[
\sB\to \bsc_n
\]
is a category of basics.  So for $\cB\to \sB$ a right fibration over a $\Top$-enriched left ideal of $\bsc_n$, the composition $\cB \to \sB \to \bsc_n$ is a category of basics.  
Conversely, for $\sB\to \bsc_n$ a category of basics, the full subcategory $\sB \subset \bsc_n$, spanned by those $U$ for which the fiber $\sB_U$ is non-empty, is a left ideal of $\Top$-enriched categories.  
\end{example}

\begin{definition}\label{definition:D_n}
Define the category of basics 
\[
\sD_n\to \bsc_n
\]
as the coherent nerve of $\bsc_{n,0}$.  
\end{definition}

\begin{example}[Framed 1-manifolds with boundary]\label{I}
Consider the left ideal $\cI'\subset \bsc_1$ whose set of objects is $\{\RR,U^1_\ast \}$.  We point out that $\iota U^1_\ast = \RR_{\geq 0}$.
Define by $\cI\to \cI'$ the (unique) right fibration whose fiber over $\RR$ is a point, $\Delta^0$, and whose fiber over $U^1_\ast$ is $\Delta^0\sqcup \Delta^0$, the discrete simplicial set with two vertices.  Explicitly, we write the vertices of $\cI$ as $\RR$, $\RR_{\geq -\infty}$, and $\RR_{\leq \infty}$.  An edge between two is a smooth open embedding which preserves orientation.  
An $\cI$-manifold is an oriented smooth $1$-manifold with boundary, $[-\infty,\infty]$ being an important example.

We will denote by $\mfld_1^{\partial, \fr}$ the quasi-category of $\cI$-manifolds. We will denote the full subcategory consisting of disjoint unions of basics by $\disk_1^{\partial, \fr}$. We will come back to this category in \textsection\ref{section:interval}.
\end{example}

\begin{example}\label{fr}
Denote by the right fibration  $\sD_n^{\fr} := (\sD_n)_{/\RR^n} \to \sD_n$.  It is unstraightening of the continuous functor $\sD_n^{\op} \to \Top$ given by $\sD_n(-,\RR^n)$ whose value on $\RR^n$ (the only object) is $\Emb(\RR^n,\RR^n) \simeq \sO(n)$.  As so, it is useful to regard a vertex of $\sD_n^{\fr}$ as a framing on $\RR^n$, that is, a trivialization of the tangent bundle of $\RR^n$.  Because $\sD_n$ is a Kan complex, the slice simplicial set $\sD_n^{\fr} \simeq \ast$ is contractible.  
A $\sD_n^{\fr}$-manifold is a smooth $n$-manifold together with a choice of trivialization of its tangent bundle.  A morphism between two is a smooth embedding together with a path of trivializations from the given one on the domain to the pullback trivialization of the target.  
Composition is given by compositing smooth embeddings and concatenating paths.  

\end{example}

\begin{example}
Let $G\xra{\rho} {\sf GL}(\RR^n)$ be a map of topological groups.  There results a Kan fibration between Kan complexes ${\sB}G \to {\sf BGL}(\RR^n)$.  
Define the right fibration $\sD_n^G \to \sD_n$ through the equivalence of Kan complexes ${\sf BGL}(\RR^n) \xra{\simeq} \sD_n$.
A $\sD_n^G$-manifold is a smooth $n$-manifold with a (homotopy coherent) $G$-structure on the fibers of its tangent bundle.  A morphism of such a smooth embedding together with a path from the fiberwise $G$-structure on the domain to the pullback $G$-structure under the embedding.  

Examples of such a continuous homomorphism are the standard maps from ${\sf Spin}(n)$, $\sO(n)$, ${\sf SO}(n)$.  The case $\ast \to {\sf GL}(\RR^n)$ of the inclusion of the identity subgroup gives the category of basis $\sD_n^{\fr}$ of Example~\ref{fr}. 

\end{example}

\begin{example}\label{example:Ekn}
Recall the left ideal $\sD_{n,k}^{\mathsf{Knk}}$ of Example~\ref{example:Edn'}.  Consider the topological category $\sD'_{n,k}$ over $\sD_{n,k}^{\mathsf{Knk}}$ consisting of the same (two) objects $\RR^n$ and $U^n_{S^{n-k-1}}$, and with morphism spaces given as
\begin{itemize}
\item $\sD'_{n,k}(\RR^n,\RR^n) = \sD_{n,k}^{\mathsf{Knk}}(\RR^n,\RR^n) = \Emb(\RR^n,\RR^n)\xla{\simeq} \sO(n)$,
\item $\sD'_{n,k}(\RR^n, U^n_{S^{n-k-1}}) = \sD_{n,k}^{\mathsf{Knk}}(\RR^n,U^n_{S^{n-k-1}}) = \Emb(\RR^n , \RR^n\smallsetminus \RR^k)\xla{\simeq} S^{n-k-1} \times \sO(n)$,
\item $\sD'_{n,k}(U^n_{S^{n-k-1}},\RR^n) = \sD_{n,k}^{\mathsf{Knk}}(U^n_{S^{n-k-1}},\RR^n) = \emptyset$,
\item $\sD'_{n,k}(U^n_{S^{n-k-1}},U^n_{S^{n-k-1}}) \subset \Emb(\RR^n, \RR^n)$ -- the subspace consisting of those $\RR^n\xra{f}\RR^n$ for which there are the containments $f(\RR^k)\subset \RR^k$ and $f(\RR^n\smallsetminus \RR^k) \subset \RR^n\smallsetminus \RR^k$.  
\end{itemize}
Composition for $\sD'_{n,k}$ is evident.  
There is the obvious topological functor $\sD'_{n,k} \to \sD^{\sf Knk}_{n,k}$, which gives a map of quasi-categories.  
This map of quasi-categories is equivalent to a right fibration 
\[
\sD_{n,k} \to \sD_{n,k}^{\sf Knk}
\]
which is unique up to canonical equivalence of right fibrations over $\sD^{\sf Knk}_{n,k}$ (this is implemented as fibrant replacement in the model structure of right fibrations given in~\cite{topos}).
More explicitly, the right fibration $\sD_{n,k} \to \sD^{\sf Knk}_{n,k}$ can be obtained as the unstraightening construction applied to the coherent nerve of the topological functor
\[
\bigl(\sD^{\sf Knk}_{n,k}\bigr)^{\op} \to \Top
\]
given by assigning to $\RR^n$ a singleton $\ast$, and to $U^n_{S^{n-k-1}}$ the space
\[
\sD^{\sf Knk}_{n,k}(U^n_{S^{n-k-1}}, U^n_{S^{n-k-1}}) /\sD'_{n,k}(U^n_{S^{n-k-1}}, U^n_{S^{n-k-1}}) 
\]
where the topological monoid $\sD'_{n,k}(-,-)$ acts by post composition -- this space has the homotopy type of $\Diff(S^{n-k-1}) / \sO(n-k)$ and is interpreted as the space of smooth structures on $\RR^k\times C(S^{n-k-1})$ which agree with the given conically smooth structure.  

A $\sD_{n,k}$-manifold is a smooth $n$-manifold equipped with a properly embedded smooth $k$-manifold. Provided $n-3\leq k \leq n$, then $\sO(n-k)\xra{\simeq} \Diff(S^{n-k-1})$ is an equivalence (due to the proof of Smale's conjecture) and thus the right fibration $\sD_{n,k}\xra{\simeq} \sD_{n,k}^{\mathsf{Knk}}$ is an equivalence of quasi-categories.  For this range of $k$ then, there is no distinction between a $\sD_{n,k}$-manifold and a $\sD_{n,k}^{\mathsf{Knk}}$-manifold.  A case of particular interest is $(n,k) = (3,1)$.  

\end{example}

\begin{example}
Let $A$ be a set whose elements we call colors.  Denote by $\sD_{n,k}^A\to \sD_{n,k}$ the right fibration whose fiber over $\RR^n$ is a point and whose fiber over $U^n_{S^{n-k-1}}$ is the set $A$.  
A $\sD_{n,k}^A$-manifold is a collection $\{L_\alpha\}_{\alpha\in J}$ of pairwise disjoint properly embedded smooth $k$-manifolds of a smooth $n$-manifold $M$.

As a related example, let $S\subset A\times A$ be a subset and denote by $\sD_n^S$ the category of basics over the objects $\RR^n$ and $U^n_{\ast \sqcup \ast}$ whose fiber over the first object is the set $A$ of colors, and whose fiber over the second object is $S$.    The edges (and higher simplicies) are given from the two projections $S \to A$.  A $\sD_n^S$-manifold is a smooth $n$-manifold together with a hypersurface whose complement is labeled by $A$, and the colors of two adjacent components of this complement are specified by $S$.
We call such a geometric object a ($S$-indexed) \emph{defect} of dimension $n$.  

\end{example}

\begin{example}\label{corners}
Recall the left ideal $\sD_n^{\un{\partial}}$ of Example~\ref{boundary}.  
Consider the continuous functor $(\sD_n^{\un{\partial}} )^{\op} \to \mathsf{POSet}$ given by $U^n_{\Delta^{k-1}} \mapsto \{S\in  \pi_0(\iota (U^n_{\Delta^{k-1}})_j)\mid 0\leq j \leq n\}\cong (s^{\{1,\dots,n\}})^{\op}$ the poset of components of the open strata with order relation $S\leq T$ if $S\subset \ov{T}$ is contained in the closure -- this poset happens to be a cube with ${n}\choose{j}$ path components of the $j^{\rm th}$ open strata.  Because morphisms in $\bsc_n$ induce stratified maps of underlying stratified spaces, this assignment is evidently functorial.  Denote by $\sD_{\langle n \rangle} \to \sD_n^{\un{\partial}}$ the Grothendieck construction on this functor.  Upon regarding it as a quasi-category via the coherent nerve construction, $\sD_{\langle n \rangle}$ is a category of basics.  A $\sD_{\langle n \rangle}$-manifold is an $n$-manifold with corners equipped with coloring of its strata.  In particular, while the tear drop $[-1,1]_{/-1\sim 1}$ is a $1$-manifold with corners (in the sense of Example~\ref{boundary}), it is not a $\sD_{\langle 1 \rangle}$-manifold.   
Often, this $\langle~\rangle$-notion of a manifold with corners is easier to work with.  For instance, the space of endomorphisms of an object over $U^n_{\Delta^{n-1}}$ is contractible.  

\end{example}

\begin{example}\label{Ekn-framed}
Recall the left ideal $\sD_{n,k}$ of Example~\ref{example:Ekn}.  
Consider the standard homeomorphism $\xi\colon \RR^n = \RR^k\times \RR^{n-k} \xra{\cong}  \iota U^n_{S^{n-k-1}} \RR^k\times C(S^{n-k-1})$ given by $\xi(u,v) = (u,\lVert v\rVert , \frac{v}{\lVert v \rVert})$.  The restriction of the inverse $\zeta = \xi^{-1}_|\colon U^n_{S^{n-k-1}} \smallsetminus \RR^k \longrightarrow \RR^n$ is a morphism in $\snglr_{n,0}$ -- that is, a smooth embedding of ordinary $n$-manifolds.  

Denote by $\sD_{n,k}^{\fr}\to \sD_{n,k}$ the right fibration obtained as the unstraightening of the map $\sD_{n,k}^{\op} \to \cS$ given on objects as  $\RR^n \mapsto \sD_{n,k}(\RR^n , \RR^n)$, $U^n_{S^{n-k-1}}\mapsto \sD_{n,k}(U^n_{S^{n-k-1}} , U^n_{S^{n-k-1}})$, on endomorphisms in the evident way, and on the remaining class of morphisms as 
\[
(\RR^n \xra{f} U^n_{S^{n-k-1}}) \mapsto \Bigl(\sD_{n,k}(U^n_{S^{n-k-1}} , U^n_{S^{n-k-1}}) \xra{\zeta_\ast f^\ast } \sD_{n,k}(\RR^n , \RR^n)\Bigr)
\]
which is clearly continuous in the argument $f$.  
Observe that $\sD_n^{\fr}$ embeds into $\sD_{n,k}^{\fr}$ as the fiber over $\RR^n$.  The fiber of $\sD_{n,k}^{\fr}$ over the other vertex $U^n_{S^{n-k-1}}$ is what one can justifiably name the Kan complex of \emph{framings} of $U^n_{S^{n-k-1}}$.  Via~\textsection\ref{orthogonal}, this terminology is justified through the identification of the fiber over $U^n_{S^{n-k-1}}$ being canonically equivalent to a (twisted) product $\sO(k)\widetilde \times \Aut(S^{n-k-1})$ which is the Kan complex of choices of `trivializations of the tangent stalk of $U^n_{S^{n-k-1}}$ at it center'.    

Consider the map of quasi-categories $\sD_{n,k} \to \Delta^1$ determined by $\{\RR^n\}\mapsto 0$, $\{U^n_{S^{n-k-1}}\}\mapsto 1$.  The composite map $\sD_{n,k}^{\fr} \to \sD_{n,k} \to \Delta^1$ is an equivalence of quasi-categories.  This is analogous to (and restricts to) $\sD_n^{\fr} \simeq \ast$.  

A $\sD_{n,k}^{\fr}$-manifold is a framed $n$-manifold and a properly embedded $k$-submanifold equipped with a splitting of the ambient framing along the tangent bundle of this submanifold.  

\end{example}

\begin{example}
Fix a smooth embedding $e\colon S^{k-1}\sqcup S^{l-1} \to S^{n-1}$.  Regard the datum of $e$ as a singular $(n-1)$-manifold, again called $e$, whose underlying space is $S^{n-1}$ with singularity locus the image of $e$.  
Consider the left ideal $\sD_{k,l,n}^e\to \bsc_n$ whose set of objects is $\{\RR^n, U^n_{S^{n-k-1}}, U^n_{S^{n-l-1}}, U^n_{e}\}$.  
A $\sD_{k,l,n}^e$-manifold is a smooth $n$-manifold $M$ together with a pair of properly embedded smooth submanifolds $K,L\subset M$ of dimensions $k$ and $l$, respectively, whose intersection locus is discrete and of the form specified by $e$.  
As a particular example, if $k+l=n$ and $e$ is the standard Hopf link, then a $\sD_{k,l,n}^e$-manifold is a pair of submanifolds (of dimensions $k$ and $l$) of a smooth manifold which intersect transversely as a discrete subset.  

\end{example}

\begin{example}
Recall the category of basics $\sD_{n,k}$ of Example~\ref{example:Ekn}.  
There is a standard map of enriched categories $\sD_{2,1} \to \sD_{3,1}$ given by the assignments of objects $\RR^2\mapsto \RR^3$ and $U^2_{S^0} \mapsto U^3_{S^1}$, and on spaces of morphisms as 
\begin{itemize}
\item
$\sD_{2,1}(\RR^2,\RR^2) = \Emb(\RR^2,\RR^2) \xra{-\times \RR} \Emb(\RR^3,\RR^3) = \sD_{3,1}(\RR^3,\RR^3)$,
\item
$\sD_{2,1}(U^2_{S^0},U^2_{S^0})  \to
\sD_{3,1}(U^3_{S^1},U^3_{S^1})$ given from the identification $\RR\times\bigl(\RR\times C(S^0)\bigr) \cong \RR\times C(S^1)$,
\item
$\sD_{2,1}(\RR^2,U^2_{S^0}) = \Emb(\RR^2,\RR^2\smallsetminus \RR^1) \xra{-\times \RR} \Emb(\RR^3,\RR^3\smallsetminus \RR^2) \hookrightarrow \Emb(\RR^3,\RR^3\smallsetminus \RR) = \sD_{3,1}(\RR^3,U^3_{S^1})$.
\end{itemize}
Modify the map of quasi-categories $\sD_{2,1} \to \sD_{3,1}$ to a right fibration.  
For concreteness (though unnecessary) the fiber of this right fibration over $\RR^3$ is $S^2 =  \sO(3)/\sO(2)$, over $U^3_{S^1}$ is $S^1 = \sO(2)/\sO(1) \simeq \Diff(S^1)/\Diff(S^0)$, and over an edge 
with a prescribed lift of its target is $\ZZ\times \ZZ \times S^2 \simeq \mathsf{hofib}(S^0 \to S^1)\times \sO(3)/\sO(2)$.  
Regard a link in $\RR^3$ as an $\sD_{3,1}$-manifold $L$.  A reduction of the structure category (think ``structure group'') of $L$ to $\sD_{2,1}$ -- that is, a lift of the tangent map
\[
\xymatrix{
&
\widetilde{\sD}_{2,1}  \ar[d]
\\
\widehat{L}  \ar[r]  \ar@{.>}[ur]  
&
\sD_{3,1}~,
}
\]
is the data of a non-vanishing vector field on $\RR^3$ which is transverse to the link.  
This lift is equivalent to a right fibration associated to a $\sD_{2,1}$-manifold if and only if the link $L$ is equivalent to the trivial link.

\end{example}

\subsubsection{Collar-gluings}\label{section-collar-gluings}

Recall the structure functor $\bsc_k \xra{R^{n-k}} \bsc_n$ -- it gives a right fibration of quasi-categories.  
Denote the pullback in quasi-categories
\begin{equation}\label{Bk}
\xymatrix{
\cB_k  \ar[rr]^{R^{n-k}}  \ar[d]
&
&
\cB  \ar[d]  
\\
\bsc_k  \ar[rr]^-{R^{n-k}} 
&
&
\bsc_n
}
\end{equation}
thereby defining the quasi-category and map 
\begin{equation}\label{Mk}
\mfld(\cB_k) \xra{R^{n-k}} \mfld(\cB)~.
\end{equation}  
Employ the notation $\cB_{<r} = \coprod_{0\leq k <r} \cB_k$ and likewise for $\mfld(\cB_{<r}) = \coprod_{0\leq k <r} \mfld(\cB_k)$.  
There are canonical maps $R^-\colon \cB_{<r} \to \cB$ and $R^-\colon \mfld(\cB_{<r}) \to \bman$.

\begin{definition}\label{B-gluings}
A \emph{collar-gluing of $\cB$-manifolds} is a collar-gluing $X=X_-\cup_{RV} X_+$ of singular $n$-manifolds together with a map $\widehat{X} \xra{g} \cB$ over $\bsc_n$.  
Notice that a collar-gluing of $\cB$-manifolds $X=X_-\cup_{RV}X_+$ determines and, up to canonical equivalence, is determined by a diagram of $\cB$-manifolds
\[
\xymatrix{
R(V,g_0)  \ar[r]^{i_+}  \ar[d]^{i_-}
&
(X_+,g_+)  \ar[d]^{f_+}
\\
(X_-,g_-)\ar[r]^{f_-}  
&
(X,g)
}
\]
in where $(V,g_0)$ is a $\cB_{n-1}$-manifold and for which $f_-(\iota X_-)\cup f_+(\iota X_+) = \iota X$.  
We denote a collar-gluing of $\cB$-manifolds as $X=X_-\cup_{RV} X_+$.
\end{definition}

Temporarily denote by $\cB\subset \mfld(\cB)^\cI\subset \mfld(\cB)$ the smallest full sub-quasi-category closed under collar-gluings.  The following corollary is an immediate consequence of Theorem~\ref{collar=fin}.
\begin{cor}\label{B-collar=fin}
There is an equivalence of quasi-categories
\[
\mfld(\cB)^\cI \simeq \mfld(\cB)^{\sf fin}~.
\]  
\end{cor}

\section{Homology theories}
In this section we define the notion of a \emph{homology theory for $\cB$-manifolds}.
Roughly, a homology theory for $\cB$-manifolds (valued in the symmetric monoidal category $\cC^\otimes$) is a functor $H\colon \mfld(\cB) \to \cC$ for which 
\begin{itemize}
\item $\sqcup \mapsto \otimes$,
\item $\{{\rm isotopies}\}\mapsto \{{\rm equivalences}\}$,
\item $H$ satisfies an excision axiom (see Definition~\ref{homology-theory}).
\item Sequential colimits are preserved.
\end{itemize}
This is in fact a generalization of the usual Eilenberg-Steenrod axioms, as we will see in \textsection\ref{classics}. Recall that in the Eilenberg-Steenrod axioms, the excision property can be phrased as sending pushout diagrams to certain colimit diagrams. While this is a fine definition if the domain is the category of spaces, the category of $\cB$-manifolds admits very few pushouts.  As so, collar-gluing diagrams (Definition~\ref{B-gluings}) will play the role of pushout diagrams.

Our definition of a homology theory for $\cB$-manifolds is intended to accommodate many examples, which typically depend on the specifics of $\cB$.  
While at first consideration a homology theory is a huge amount of data, being an assignment of an object of $\cC$ for each $\cB$-manifold, our main result is a consolidation of this information.
Specifically, Theorem~\ref{characterization} characterizes homology theories for $\cB$-manifolds as $\disk(\cB)$-algebras (Definition~\ref{oplus}).  
This is analogous to the situation in classical algebraic topology where a generalized homology theory is determined by its value on a point.
In practice a $\disk(\cB)$-algebra is a manageable amount of data.  The examples in~\textsection\ref{section:examples} will illustrate this.

\subsection{Symmetric monoidal structures and $\disk(\cB)$-algebras}
In this section we observe that disjoint union defines a symmetric monoidal structure on $\mfld(\cB)$ and we define the notion of a $\disk(\cB)$-algebra.  We make use of Lurie's theory of $\infty$-operads, specifically of symmetric monoidal quasi-categories; see~\cite{dag}, \textsection2.  We will also utilize the coherent nerve construction to pass from a $\Top$-enriched symmetric monoidal category to a symmetric monoidal quasi-category.  

Recall Definition 2.0.0.7 from~\cite{dag} that a symmetric monoidal quasi-category is a quasi-category $\cC^\otimes$ equipped with a map $\cC^\otimes \xra{p} \mathsf{Fin}_\ast$ to the nerve of the category of pointed finite sets, such that 
\begin{itemize}
\item $p$ is a coCartesian fibration 
\item Let $J$ be a finite set and let $J_+ = J \cup \{+\}$ denote the finite pointed set with basepoint $+$. For every $j \in J$, define $\rho_j\colon J_+ \to \{j\}_+$ to be the map of based finite sets for which $j\mapsto j$ and $j\neq j'\mapsto +$. The induced map of fibers
\[
\prod_{j\in J} (\rho_j)_\ast \colon \cC^\otimes_{J_+} \xra{\simeq} \prod_{j\in J} \cC^\otimes_{\{j\}_+}
\]
is an equivalence of quasi-categories.

\end{itemize}
A {\em symmetric monoidal functor} $\cC^\otimes \to \cD^\otimes$, or {\em a map of symmetric monoidal quasi-categories}, is a map of quasi-categories over $\mathsf{Fin}_\ast$ which sends coCartesian edges to coCartesian edges.  
For $\cC^\otimes$ a symmetric monoidal quasi-category, we dente the fiber over $\ast_+\in \mathsf{Fin}_\ast$  by $\cC$ and refer to it as the \emph{underlying quasi-category} of $\cC^\otimes$.  

Via the straightening-unstraightening equivalence, a symmetric monoidal quasi-category is equivalent to a functor $\cC^\otimes\colon \mathsf{Fin}_\ast \to \mathsf{Cat}_\infty$ to the quasi-category of quasi-categories for which $\cC^\otimes(J) \xra{\cong} \prod_{j\in J} \cC^\otimes(\{j\})$.

\begin{example}
Recall the quasi-category $\cS$ of Kan complexes.  Denote by $\cS^\times \to \mathsf{Fin}_\ast$ the $\infty$-operad of \emph{spaces under Cartesian product} which is defined as follows.  Denote by $\mathsf{Fin}_\ast^\times\to \mathsf{Fin}_\ast$ the Grothendieck construction on the functor assigning to $J_+$ the category $\cP(J)^{\op}$ which is the opposite of the poset of subsets of $J$.  
A $p$-simplex of $\cS^\times$ over $\Delta^p\xra{\sigma} \mathsf{Fin}_\ast$ is a map $A\colon \Delta^p \times_{\mathsf{Fin}_\ast} \mathsf{Fin}_\ast^\times \to \cS$ for which for each $0\leq i \leq p$ the restriction 
\[
d_{\{i\}}^\ast A\colon \cP\bigl(\sigma(i)\bigr)^{\op} = \Delta^{\{i\}}\times_{\sigma(i)} \cP\bigl(\sigma(i)\bigr)^{\op} \to \cS
\]
witnesses $d^\ast_{\{i\}}A(J)$ as the product $\prod_{j\in J} d^\ast_{\{i\}} A(\{j\})$.  
Then $\cS^\times \to \mathsf{Fin}_\ast$ is an $\infty$-operad. (For details, see~\cite{dag},~\textsection2.4.1.)  

So a vertex of $\cS^\times$ over $J_+$ is a $J$-indexed collection of Kan complexes.  An edge of $\cS^\times$ over $J_+\to K_+$ is an $J$-indexed collection of Kan complexes $(X_j)_{j\in J}$, a $K$-indexed collection of Kan complexes $(Y_k)_{k\in K}$, and for each $k\in K$ a map $\prod_{f(j)=k} X_j \to Y_k$.  
\end{example}

\begin{example}[coCartesian symmetric monoidal quasi-categories]\label{coCart}
Let $\cC$ be a quasi-category.  Define the $\infty$-operad $\cC^\amalg$ as follows.  
Denote by $\widetilde{\mathsf{Fin}}_\ast \to \mathsf{Fin}_\ast$ the Grothendieck construction on the functor assigning to $J_+$ the set $J$, regarded as a category with only identity morphisms.  
A $p$-simplex of $\cC^\amalg$ over $\Delta^p\xra{\sigma} \mathsf{Fin}_\ast$ is a map $\cU\colon \Delta^p\times_{\mathsf{Fin}_\ast} \widetilde{\mathsf{Fin}}_\ast  \to \cC$.  
It is shown in~\cite{dag}~\textsection2.4.1 that $\cC^\amalg \to \mathsf{Fin}_\ast$ is an $\infty$-operad.

Explicitly a vertex of $\cC^\amalg$ over $J_+$ is an $J$-indexed collection of objects of $\cC$.  An edge of $\cC^\amalg$ over $J_+\xra{f} K_+$ is an $J$-indexed collection of objects $(c_j)_{j\in J}$, a $K$-indexed collection of objects $(d_k)_{k\in K}$, and for each pair $(j,k)$ with $f(j)=k$ a map $c_j\to d_k$.  

The underlying category of $\cC^\amalg$ is $\cC$.  
If $\cC$ admits coproducts there is a map of quasi-categories $\cC^\amalg \to \cC$ given by $\bigl(J,(c_j)\bigr)\mapsto \coprod_{j\in J} c_j$ and $\cC^\amalg$ is a symmetric monoidal quasi-category.  
In this case, we refer to $\cC^\amalg$ as a \emph{coCartesian symmetric monoidal quasi-category}.  

\end{example}

We record the following example expressly for use in~\S\ref{free-examples}.  
\begin{example}\label{free-guy}
Recall that a morphism $J_+\xra{f} K_+\in \mathsf{Fin}_\ast$ is \emph{inert} if $f^{-1}(k)$ is a singleton for each $k\in K$.  Clearly, inert morphisms are closed under composition.  
Denote $\cC^\amalg_{\sf{inert}} = (\mathsf{Fin}_\ast)_{\sf inert}\times_{\mathsf{Fin}_\ast} \cC^\amalg$.  The canonical map $\cC^\amalg_{\sf inert} \to \mathsf{Fin}_\ast$ is the free $\infty$-operad on $\cC$.   
\end{example}

\begin{example}
Let $(\sC,\otimes,1)$ be a symmetric monoidal $\Top$-enriched category.  
Define the $\Top$-enriched category $\sC^\otimes\to \mathsf{Fin}_\ast$ over finite pointed sets, whose objects are pairs $(J,(c_j)_{j\in J})$ consisting of a finite set $J$ and an $J$-indexed sequence of objects $c_j\in \sC$.  
Define the space of morphisms
\[
\sC^\otimes \bigl((J,(c_j)),(K,(d_k))\bigr) = \coprod_{J_+\xra{f} K_+} \prod_{k\in K}  \sC\Bigl(\bigotimes_{\{f(j)=k\}} c_j ~,~ d_k\Bigr)~.
\]
The map $\sC^\otimes \to \mathsf{Fin}_\ast$ is obvious.  
Composition in $\sC^\otimes$ over $J_+\xra{f} K_+\xra{g}L_+$ is given by 
	\[
	\bigl((u_k)_{k\in K},(v_l)_{l\in L}\bigr)\mapsto \bigl(v_l\circ (\otimes_{\{g(k)=l\}} u_k)\bigr)_{l\in L}.
	\]  
Then the coherent nerve of $\sC^\otimes\to \mathsf{Fin}_\ast$ is a symmetric monoidal quasi-category.  
\end{example}

\begin{definition}[\cite{dag} 2.1.3.1]\label{factorization-definition}
Let $\cC^\otimes$ and $\cO^\otimes$ be symmetric monoidal quasi-categories.  
A \textit{$\cO^\otimes$-algebra in $\cC^\otimes$}, or simply an $\cO$-algebra in $\cC$ if the $\otimes$-structures are understood, is a map of coCartesian fibrations over $\mathsf{Fin}_\ast$
\[
A\colon \cO^\otimes \to \cC^\otimes.
\]
The quasi-category of such functors is denoted $\operatorname{Alg}_\cO(\cC^\ot)$, or sometimes $\Fun^\otimes(\cO^\otimes , \cC^\otimes)$; it is evidently functorial in the variables $\cO^\otimes$ and $\cC^\otimes$.  
\end{definition}

Recall that $(\snglr_n, \sqcup, \emptyset)$ is a symmetric monoidal $\Top$-enriched category.  There is an obvious morphism of symmetric monoidal quasi-categories  $\snglr_n^{\sqcup} \to \cP(\bsc_n)^\amalg$ given by $(J,(X_j)_{j\in J})\mapsto (J,(\widehat{X}_j)_{j\in J})$.  
Denote by 
\[
\mathsf{SngDisk}_n \subset (\snglr_n,\sqcup, \emptyset)
\]
the smallest sub-symmetric monoidal $\Top$-enriched category whose underlying $\Top$-enriched category contains $\bsc_n$.  Now fix a category of basics $\cB$.  

\begin{definition}\label{oplus}
Define 
\[
\mfld(\cB)^{\sqcup} = \snglr_n^{\sqcup} \times_{\cP(\bsc_n)^\amalg} (\cP(\bsc_n)_{/\cB})^\amalg
\]
which is a coCartesian fibration over $\mathsf{Fin}_\ast$ and therefore a symmetric monoidal quasi-category.  
Define the sub-symmetric monoidal quasi-category 
\[
\disk(\cB)^{\sqcup} =  \mathsf{SngDisk}_n^{\sqcup}\times_{\cP(\bsc)^\amalg} (\cP(\bsc_n)_{/\cB})^\amalg ~{}~{}~ \subset~{}~  \mfld(\cB)^{\sqcup}~.
\]
Denote the underlying quasi-category of $\disk(\cB)^{\sqcup}$ as $\disk(\cB)$.  
We refer to $\operatorname{Alg}_{\disk(\cB)}(\cC^\ot)$ as the category of $\disk(\cB)$-algebras (in $\cC$).  
\end{definition}

\begin{notation}
Recall the category of basics $\sD_n$ from Definition~\ref{definition:D_n} and the various elaborations given in the examples following. We denote 
	\[
	\disk_n = \disk(\sD_n),
	\qquad
	\disk_n^{\fr} = \disk(\sD_n^{\fr}),
	\qquad
	\disk_n^{\partial} = \disk(\sD_n^{\partial}),
	\qquad
	\disk_{n,k}^{\fr} = \disk(\sD_{n,k}^{\fr})
	\]
and likewise for other elaborations on $\sD$.
\end{notation}

\begin{example}\label{n-disk-algebras}
A $\disk^{\fr}_n$-algebra in $\cC^\ot$ is an algebra over the little $n$-disk operad in $\cC^\ot$.  In particular, a $\disk^{\fr}_1$-algebra in $\cC^\ot$ is an $A_\infty$-algebra in $\cC^\ot$.  

\end{example}

\begin{example}\label{deligne-boundary}
Consider the category of basics $\sD^{\partial,\fr}_n$ which is a framed version of Example~\ref{boundary}.  
A $\disk_n^{\partial,\fr}$-algebra is the data $(A,B,{\sf a})$ of a $\disk^{\fr}_n$-algebra $A$, a $\disk^{\fr}_{n-1}$-algebra $B$, and an action 
	\[
	{\sf a}\colon A\longrightarrow \hh^\ast_{\sD^{\fr}_{n-1}}(B) := {\m_B^{\disk^{\fr}_{n-1}}}(B,B)
	\] 
which is a map of $\disk^{\fr}_n$-algebras; this reformulation is the higher Deligne conjecture, proved in this generality in~\cite{dag} and~\cite{thomas}.
\end{example}

Notice that the underlying quasi-category of $\mfld(\cB)^{\sqcup}$ is $\mfld(\cB)$.  Explicitly, an object of $\mfld(\cB)^{\sqcup}$ is a finite set $J$ together with an $J$-index collection $(X_j,g_j)_{j\in J}$ of $\cB$-manifolds.  A morphism $\bigl(J,(X_j,g_j)\bigr) \to \bigl(K,(Y_k,h_k)\bigr)$ is a map $J_+\xra{f} K_+$ together with morphisms $\bigsqcup_{f(j)=k} (X_j,g_j) \to (Y_k,h_k)$ of $\cB$-manifolds for each $k\in K$.  
Such an object $\bigl(J,(U_j)\bigr)$ is in $\disk(\cB)^{\sqcup}$ if for each $j\in J$ the $\cB$-manifold $U_j$ is a disjoint union of basics.  

The quasi-category $\cP(\bsc_n)_{/\cB}$ admits finite coproducts.  
There is a commutative diagram of quasi-categories
\[
\xymatrix{
\mfld(\cB)^{\sqcup}  \ar[r]  \ar[d]^\sqcup  
&
(\cP(\bsc_n)_{/\cB})^\amalg  \ar[d]^\amalg
\\
\mfld(\cB)   \ar[r]
&
\cP(\bsc_n)_{/\cB}
}
\]
where the left downward arrow is given by disjoint union $(J,(X_j,g_j)_{a \in J})\mapsto \bigsqcup_{j\in J} (X_j,g_j)$ and the right downward arrow is given by coproduct $(A,(E_j,g_j)_{j\in J})\mapsto \coprod_{j\in J} (E_j,g_j)$.

\begin{variation}
In Definition~\ref{oplus} one could replace $\snglr$ by $\psnglr$ to obtain the symmetric monoidal quasi-category $\pmfld(\cB)^{\sqcup}$.  
The equivalence of quasi-categories $\pmfld(\cB)\simeq \mfld(\cB)$ of Theorem~\ref{no-pre-mans} induces an equivalence of symmetric monoidal quasi-categories 
\begin{equation}\label{oplus-equiv}
\pmfld(\cB)^{\sqcup} \simeq \mfld(\cB)^{\sqcup}~.
\end{equation}
Likewise, there is the symmetric monoidal quasi-category $(\mfld(\cB)^{\sf fin})^{\sqcup}$.
\end{variation}

\begin{remark}\label{not-coprod}
We warn the reader that, while the symmetric monoidal structure on $\mfld(\cB)^{\sqcup}$ comes from disjoint union of $\cB$-manifolds, this symmetric monoidal category $\mfld(\cB)^{\sqcup}$ is \emph{not} a coCartesian symmetric monoidal category.  Indeed, disjoint union is not a coproduct in $\mfld(\cB)$, which admits very few colimits since all morphisms are monomorphisms.  

\end{remark}

\subsection{Factorization homology}\label{factorizhom}
Every object $X \in \mfld(\cB)$ represents a functor $\mfld(\cB)^{\op} \to \cS$. Restriction along $\disk(\cB) \subset \mfld(\cB)$ defines a presheaf
	\[
	\EE_X \colon \disk(\cB)^{\op} \to \cS.
	\]
Covariantly, given $A$ a $\disk(\cB)$-algebra in a symmetric monoidal quasi-category $\cC$, we denote again by $A: \disk(\cB) \to \cC$ the restriction to the underlying quasi-category. We will define factorization homology of $X$ with coefficients in $A$ as the coend constructed from these two functors, but we first make an assumption on the target category $\cC$.

\begin{definition}[$\ast$]
A symmetric monoidal quasi-category $\cC^\otimes$ satisfies condition $(\ast)$ if the underlying category $\cC$ admits (small) colimits and for each $c \in \cC$, the functor $\cC \xra{- \otimes c} \cC$ preserves filtered colimits and geometric realizations.
In particular, for each vertex $c\in \cC$ and each space (=Kan complex) $K$ the constant map $K\xra{c} \cC$ admits a colimit.  We denote this colimit as $K\ot c \in \cC$ and refer to the map $-\ot -\colon \cS \times \cC \to \cC$ as the \emph{tensor over spaces}.
\end{definition}

\begin{remark} 
Note that since $\cC$ admits all small colimits, it is tensored over spaces. Also, condition $(\ast)$ implies the stronger property that $-\tensor c$ preserves all {\it sifted} colimits. (See Corollary 5.5.8.17 of~\cite{topos}.)
\end{remark}

\begin{example} The symmetric monoidal quasi-category ${\sf Ch}^{\sqcup}$ of chain complexes with direct sum satisfies condition $(\ast)$. In particular, the functor $V\oplus -$ preserves geometric realizations (although it does not preserve more general colimits, such as coproducts).
\end{example}

\begin{definition}[Factorization homology]\label{defn:factorization-homology}
Let $\cC$ be a symmetric monoidal quasi-category which satisfies $(*)$.  
Let $A\in \Alg_{\disk(\cB)}(\cC^\ot)$ and let $X\in \mfld(\cB)$ regarded as a vertex in the underlying quasi-category of $\mfld(\cB)^{\sqcup}$.   
We define 
\begin{equation}\label{fact-def}
\int_X A~ :=~ \EE_X\underset{\disk(\cB)}\bigotimes A  ~\simeq ~\colim_{\bigsqcup_{k\in K} U_k \to  X} ~{}~\bigotimes_{k\in K} A(U_k)~{}~ \in \cC
\end{equation}
to be the coend of 
	$
	\disk(\cB)^{\op} \times \disk(\cB) 
	\xra{\EE_X \times A} 
	\cS \times \cC  \xra{\otimes} \cC,
	$	
the second map being the tensor over spaces. 
In (\ref{fact-def}) we have equivalently written this coend as a colimit over the quasi-category $\disk(\cB)_{/X}:= \disk(\cB) \times_{\mfld(\cB)} \mfld(\cB)_{/X}$.

We refer to the object $\int_X A\in \cC$ as the \textit{factorization homology of $X$ with coefficients in $A$}.  
\end{definition}

\begin{remark}\label{operadic}
Under the assumption~($\ast$) on $\cC^\otimes$, the expression defining factorization homology of $X$ is equivalent to Lurie's {\em operadic left Kan extension} of $\disk(\cB)^{\sqcup}\xra{A}\cC^\otimes$ along $\disk(\cB)^{\sqcup} \to \mfld(\cB)^{\sqcup}$, evaluated on $X$. 
(We refer readers to \textsection3.1.2 of~\cite{dag} for a general treatment.) We will make use of the universal property of operadic left Kan extensions throughout, so we comment briefly on this equivalence.

Using that $\disk(\cB)^{\sqcup} \to \mathsf{Fin}_\ast$ is a coCartesian fibration, the defining expression for the weak operadic left Kan extension, evaluated at $X$, can be written
\begin{equation}\label{op-colim}
\colim_{\bigl(K,(U_k)\bigr) \xra{\sf act} X} ~{}~\bigotimes_{k\in K} A(U_k)
\end{equation}
where the colimit is taken over $(\disk(\cB)^{\sqcup}_{\sf act})_{/X} := \disk(\cB)^{\sqcup}_{\sf act} \times_{\mfld(\cB)^{\sqcup}_{\sf act}} (\mfld(\cB)^{\sqcup}_{\sf act})_{/X}$. The notation $`\sf act'$ signifies that we only consider {\em active} morphisms.  (Recall that a morphism $f$ in $\mathsf{Fin}_\ast$ is called {\em active} if $f^{-1}(+) = +$, where $+$ is the base point of the source and target of $f$. The active morphisms in a symmetric monoidal quasi-category are those morphisms which are coCartesian lifts of these $f$. See Definition 2.1.2.1 of~\cite{dag}.)

Because $\disk(\cB)^{\sqcup} \to \mathsf{Fin}_\ast$ is coCartesian, then so is the pullback to active morphisms $\disk(\cB)^{\sqcup}_{\sf act} \to (\mathsf{Fin}_\ast)_{\sf act}$.  
Notice that the fiber over $\{1\}_+ \in (\mathsf{Fin}_\ast)_{\sf act}$ of this latter map is the underlying category $\disk(\cB)$.
The object $\{1\}_+ \in (\mathsf{Fin}_\ast)_{\sf act}$ is terminal and therefore the inclusion $\disk(\cB)\subset \disk(\cB)^{\sqcup}_{\sf act}$ is cofinal.  
Because $\disk(\cB)_{/X} \to \disk(\cB)$ is a right fibration, it follows that the morphism $\disk(\cB)_{/X} \to (\disk(\cB)^{\sqcup}_{\sf act})_{/X}$ is cofinal as well.  

This colimit~(\ref{op-colim}) can be rewritten as a geometric realization.  
Assumption~($\ast$) ensures that the operation $-\ot c$ preserves weak operadic colimit diagrams with shape $\Delta^{\op}$. 
It follows that the expression~(\ref{op-colim}) computes the operadic left Kan extension.

\end{remark}

\begin{remark}
By the universal property of operadic left Kan extension, one can also characterize factorization homology as a left adjoint. 
Specifically, note that the inclusion of symmetric monoidal quasi-categories $\disk(\cB)^{\sqcup} \subset \bman^{\sqcup}$ induces a restriction functor
\[
\rho\colon \Fun^\otimes(\bman^{\sqcup}, \cC^\otimes)  \to \operatorname{Alg}_{\disk(\cB)}(\cC^\ot)~.
\]
Then the pair $\int$ and $\rho$ define an adjunction
\begin{equation}\label{adj}
\int \colon \Alg_{\disk(\cB)}(\cC^\ot) \leftrightarrows  \Fun^\otimes(\bman^{\sqcup}, \cC^\otimes) \colon \rho.
\end{equation}
\end{remark}

\begin{example}
It is definitional that $\int_U A \xra{\simeq} A(U)\in \cC$ for any basic open $U\in \cB$.  
\end{example}

\begin{notation}
The symbol $\int_X A$ was only defined for $X$ a $\cB$-manifold. However, if $X\in\cP\bigl(\disk(\cB)\bigr)$ is any right fibration, then the same expression defining $\int_X A$ as a colimit is valid.  In this way, we extend notation and obtain a functor $\int\colon \Alg_{\disk(\cB)}(\cC^\ot) \to  \Fun^\otimes(\pbman^{\sqcup}, \cC^\otimes)$ to $\cB$-premanifold-algebras.
\end{notation}

\begin{lemma}\label{premanifold-factorization}
Let $\dddot{X}$ be a $\cB$-premanifold and let $\dddot{X} \to X$ be a refinement. The induced morphism
\[
\int_{\dddot{X}} \xra{\simeq} \int_X
\]
is an equivalence of functors $\Alg_{\disk(\cB)}(\cC^\ot) \to \cC$.  
\end{lemma}
\begin{proof}
This follows immediately from~(\ref{oplus-equiv}) which was a consequence of Theorem~\ref{no-pre-mans}.  
\end{proof}

\begin{remark}
We discuss the relation of our work to that of Costello and Gwilliam in~\cite{kevinowen}. In Costello and Gwilliam's definition of a factorization algebra, one does not endow the set of opens $U \subset M$ with a topology (let alone fix a fibration over the space of embeddings), and one imposes a cosheaf-like condition to ensure a local-to-global principle. In contrast, we work with categories which see the topology of embeddings $U \to M$, and we have no need to impose a cosheaf condition---in light of 
Theorem~\ref{no-pre-mans}
factorization homology automatically satisfies such a condition. 
(Namely, for every cover of a basic $U\in \cB$, which we think of as a refinement $\dddot{U} \to U$, the edge
$
\int_{\dddot{U}} A \xra{\simeq} A(U)
$
is an equivalence  in $\cC$.)

It is precisely the right fibration condition of Definition~\ref{cat-o-basics} which guarantees this cosheaf property for the structures we consider, and which generally lends the content of this article toward homotopical techniques. This comes at the cost of excluding potentially interesting geometric examples (such as in holomorphic or Riemannian settings), where the natural maps to $\bsc_n$ are not right fibrations.
\end{remark}

\begin{remark}
Factorization homology appears closely related to the blob complex of
Morrison-Walker, both defined by a coend construction, at least for
their values on {\it closed} manifolds. The blob complex, or cobordism
hypothesis construction \`a la Lurie, can take as input an
$n$-category ${\frak B}^n A$, which intuitively has an single
$k$-morphism for $k< n$, and whose $n$-morphisms are equivalent to
$A$. One can define the functor $\int A$ on {\it closed} $n$-manifolds
using $\fB^n A$ rather than $A$.  That is, there exists a commutative
diagram
\[
\xymatrix{
\disk^{\fr}_n\alg(\cC^\ot)\ar[drr]_-{{\frak
B}^n}\ar[rr]^\int
&&
\Fun^\ot(\mfld_n^{\fr},\cC)\ar[rr]^-{\rm
restrict}
&&
\Fun^\ot(\mfld_n^{\fr, {\sf closed}}, \cC)
\\
&&
{\sf Cat}_n(\cC^\ot)\ar[rru]
&&
}\]
where $\mfld^{\fr,{\sf closed}}_n$ is the full groupoid of $\mfld_n$
consisting of closed $n$-manifolds and diffeomorphisms. However, there
does {\it not} exist a functor from ${\sf Cat}_n(\cC^\ot)$ to
$\Fun^\ot(\mfld_n^{\fr},\cC)$ making the above diagram commute. This
is because the functor ${\frak B}^n$ is not fully faithful (the space
$\Map(\fB^nA,\fB^nC)$ is quotient of $\Map(A,C)$ by the $(n-1)$-fold
delooping of the group of units of $C$), whereas the factorization
homology functor $\int$ is fully faithful, in particular by Theorem
\ref{characterization}.
\end{remark}

\subsection{Push-forward}\label{push-forward-section}
We introduce a technique making the expression~(\ref{fact-def}) defining factorization homology far more computable.

Fix two categories of basics $\cB_0\to \bsc_d$ and $\cB\to \bsc_n$. Any map $\cB_0\xra{F} \mfld(\cB)$ determines a map of symmetric monoidal quasi-categories
$
\disk(\cB_0)^{\sqcup}\xra{F^{\sqcup}} \mfld(\cB)^{\sqcup}.
$
Operadic left Kan extension defines the right vertical map of symmetric monoidal quasi-categories 
\[
\xymatrix{
\disk(\cB_0)^{\sqcup}  \ar[r]  \ar[d]^{F^{\sqcup}}  
&
\mfld(\cB_0)^{\sqcup}  \ar[d]^{L^\ot_F}
\\
\mfld(\cB)^{\sqcup}  \ar[r]  
&
\cP(\cB)^\amalg~.
}
\]
The passage to $\cP(\cB)^\amalg$ guarantees there are enough colimits for the Kan extension to exist, but 
we seek conditions under which the Kan extension factors through $\mfld(\cB)^{\sqcup}$. We will prove in Lemma~\ref{creation} that the following definition suffices.
We use the notion of \emph{conical smoothness}, which is a suitable notion of smoothness for continuous maps between underlying singular manifolds (of possibly differing dimensions), visit~\S\ref{conically-smooth} for a definition.  

\begin{definition}
A \emph{$\cB_0$-family (of $\cB$-manifolds)} is a pair $(F,\gamma)$ where $F: \cB_0 \to \mfld(\cB)$ is a map of quasi-categories and $\gamma: \iota \circ F \to \iota $ is a natural transformation \emph{by conically smooth maps}.
We will often denote the data $(F,\gamma)$ of a $\cB_0$-family simply by the letter $F$.
\end{definition}

\begin{example}[Product bundles]\label{product-bundles}
Let $\cB'\to \bsc_{n-d}$ be a category of basics, and let $q\colon \cB_0\times \cB' \to \cB$ be a map over $\bsc_d\times \bsc_{n-d} \to \bsc_n$.  Fix a $\cB'$-manifold $(X,g')$.  The assignment 
$
(V,h) \mapsto \bigl(V\times X , q(h,g')\bigr)
$ 
describes a map of quasi-categories $\mathsf{Pr}_X\colon \cB_0\to \mfld(\cB)$ which is evidently over (the coherent nerve of) $\bsc_d \xra{-\times X} \snglr_n$.  Moreover, there is the natural projection $\iota (V\times X) \to \iota V$.  
A particular example of this is for $\cB_0 = \cB_{n-k}$ and $\cB'=\cB_k$ with $q\colon \cB_{n-k}\times \cB_k \to \cB$ the standard map over $\bsc_{n-k}\times \bsc_k \xra{\times} \bsc_n$ -- recall~(\ref{Bk}) from~\S\ref{section-collar-gluings} for this subscripted notation.  
In this way, each $\cB_{n-k}$-manifold $Y$ determines a map of categories of basics $\mathsf{Pr}_Y\colon \cB_k \to \mfld(\cB)$ which in turn determines a restriction map denoted
\[
\int_Y \colon \Alg_{\disk(\cB)}(-) \xra{\int} \Alg_{\mfld(\cB)}(-) \xra{\mathsf{Pr}_Y^\ast} \Alg_{\disk(\cB_k)}(-)~.  
\]

Now fix a $\cB_0$-manifold $P$. One can elaborate the above example to give examples of bundles over $P$ with fibers $\cB'$-manifolds. Given a map $\cB_0 \times \cB' \to \cB$ as before, this is an example of a $(\cB_0)_{/P}$-family of $\cB$-manifolds.
Because we never use such examples, we will not develop them in proper detail.  

\end{example}

\begin{example}[Collar-gluings]\label{collar-I}
Recall the category of basics $\cI$ from Example~\ref{I}, whose set of objects is $ob~\cI = \{\RR_{\geq -\infty}, \RR_{\leq \infty}, \RR\}$.  Note that the $\cI$-manifold $[-\infty,\infty]$ is not a basic.  

Let $(X,g)=(X'_-,g'_-)\cup_{(RV,g_0)} (X'_+,g'_+)$ be a collar-gluing of $\cB$-manifolds.  Choose a conically smooth map $h\colon \iota X \to [-\infty,\infty]$ so that $h_{| \iota RV} \colon \RR \times \iota V \to [-\infty,\infty]$ is the projection followed by a standard smooth monotonic map $\RR \to [-\infty,\infty]$ which restricts to a diffeomorphism $(-1,1)\cong \RR$.     
(Such a conically smooth map can be constructed using conically smooth partitions of unity (see Lemma~\ref{part-o-1}).)  The choice of a standard diffeomorphism $\RR\cong (-1,1)$ gives a morphism $RV \to R_{(-1,1)}V \subset RV$ which we use to rewrite the collar-gluing $(X,g)=(X_-,g_-)\cup_{(RV,g_0)} (X_+,g_+)$ with $(X_- , g_-) = (X'_- \smallsetminus (\RR_{\geq 1}\times V) , (g'_-)_{| X_-})$ and likewise $(X_+,g_+) = (X'_+\smallsetminus (\RR_{\leq -1}\times V) , (g'_+)_{|X_+})$.  
Then 
	\[
	h^{-1}\RR_{\geq -\infty}= \iota X_-, 
	\qquad
	h^{-1}\RR_{\leq \infty}= \iota X_+, 
	\qquad	{\text{and}}
	\qquad
	h^{-1}\RR = \iota RV
	\]
	so one has a  $\cI$-family of $\cB$-manifolds given by 
\begin{itemize}
\item $G \colon \cI \to \mfld(\cB)$ given by $\RR_{\geq-\infty}\mapsto  \iota X_-$, $\RR_{\leq \infty}\mapsto  \iota X_+$, and $\RR \mapsto \iota RV$; where the assignment on morphisms of $\cI$ is obvious.  
\item $\gamma\colon \iota G\to \iota $ is given on an object $U$ as $\gamma_U = h_{|\iota U}$.  
\end{itemize}
We will use this observation to prove Theorem~\ref{thm:excision}.
\end{example}

\begin{lemma}\label{creation}
Let $F$ be a $\cB_0$-family of $\cB$-manifolds.  Let $P$ be a $\cB_0$-manifold regarded as a vertex of the underlying quasi-category of $\mfld(\cB_0)^{\sqcup}$. There is a 
natural lift of the value of the Kan extension
$
L_F(P)
$
to an object of $\mfld(\cB)$.  
\end{lemma}

\begin{proof}

Define the composition $f\colon \cB_0\xra{F}\mfld(\cB) \to \snglr_n$.  
We will describe a $\cB$-manifold $(f(P),g_0)$ lifting $L_F(P)$.  
We first describe the singular $n$-manifold $f(P)$ lifting the left Kan extension of $\disk(\cB_0) \xra{f^\sqcup} \snglr_n^\sqcup$ along $\disk(\cB_0)^\sqcup \subset \mfld(\cB_0)^\sqcup$.  
As discussed in Remark~\ref{operadic}, we compute this operadic left Kan extension as the left Kan extension of the map on underlying quasi-categories $\disk(\cB_0) \to \snglr_n$ along $\disk(\cB_0)\subset \mfld(\cB_0)$.  

We define the singular $n$-premanifold $f(P)$ as follows.  
Consider the topological space
\[
\iota f(P) :=\colim_{V\to P} \iota f(V)
\]
which is the strict colimit over the category whose objects are morphisms from basics $V\to P$ and whose morphisms are factorized maps $V \to W \to P$.  We take this colimit in the category of (not necessarily compactly generated Hausdorff) topological spaces.  
\begin{itemize}
\item[]
\begin{itemize}
\item[Hausdorff:] The natural transformation $\gamma$ determines a continuous map
$
\iota \gamma\colon \iota f(P) \to \iota P~.  
$
Let $p,q\in \iota f(P)$.  If $\iota \gamma(p)\neq \iota \gamma(q)$ then because the collection $\{\iota V \to \iota P\}$ forms a basis for the Hausdorff topological space $\iota P$, then there are separating neighborhoods of $p$ and $q$.  If $\iota \gamma(p) = \iota \gamma(q)$ then there is a morphism $V\to P$ for which the image of $\iota f(V) \to \iota f(P)$ contains $p$ and $q$.  
Because each such $\iota f(V) \to \iota f(P)$ is an open embedding, and because $\iota f(V)$ is Hausdorff, then there are separating neighborhoods of $p$ and $q$.  

\item[Second Countable:] The collection $\{ \iota f(V)\to \iota f(P)\}$ forms an open cover of $\iota F(P)$.  
Because $P$ admits a countable atlas, $\iota f(P)$ is second countable.  

\item[Atlas:] The collection $\cA = \{(U,\phi)\mid U\xra{\phi} f(V)\in \snglr_n\text{ with } V\to P\in \snglr_d\}$ clearly gives an open cover of $f(P)$ and forms an atlas for $f(P)$.  
\end{itemize}
\end{itemize}

Though it is not necessarily the case that the atlas $\cA$ is maximal, this is accounted for through Theorem~\ref{no-pre-mans}.  
From Lemma~\ref{covers=colims} the open cover from $\cA$ realizes the right fibration $\widehat{f(P)}\in \cP(\bsc_n)$ as the necessary colimit computing the value on $P$ of the left relevant Kan extension.  

The canonical map
\[
\Map_{\bsc_n}(\widehat{f(P)},\cB) \simeq \Map_{\bsc_n}(\colim \widehat{f(V)} , \cB) \xra{\simeq} \lim \Map(\widehat{f(V)},\cB)
\]
is an equivalence.  There results a canonical section $g_0\colon \widehat{f(P)} \to \cB$ whose restrictions to each $f(V)$ is the given section for $F(V)$.  

\end{proof}

\begin{definition}
Let $\cC^\otimes$ be a symmetric monoidal quasi-category.  
and let $F\colon \cB_0\to \bman$ be a map of quasi-categories.  
Let $A\colon \disk(\cB)^{\sqcup} \to \cC^\otimes$ be a $\disk(\cB)$-algebra.  
We refer to the composition
\[
F_*A \colon \cB_0^{\sqcup} \xra{F^{\sqcup}} \bman^{\sqcup}  \xra{\int A} \cC^\otimes~.
\]
as the \textit{push-forward of $A$ along $F$}.  
Evident by construction, push-forward is functorial with respect to algebra maps in argument $A$.  
\end{definition}

\begin{theorem}[Pushforward formula]\label{pushforward-formula}
Let $A$ be a $\disk(\cB)$-algebra in $\cC$.  Let $F$ be a $\cB_0$-family of $\cB$-manifolds and let $P$ be a $\cB_0$-manifold.
Then there is a canonical equivalence
\[
\int_P F_* A \xra{\simeq} \int_{L_F(P)} A
\]
in $\cC$.  
\end{theorem}

\begin{proof}
This follows from transitivity of operadic left Kan extensions (see Corollary~3.1.4.2 of~\cite{dag}).  Namely, for a fixed $\disk(\cB)$-algebra $A$, the diagram of symmetric monoidal quasi-categories
\[
\xymatrix{
\cB_0^{\sqcup}  \ar[rr] \ar[d]^{F^{\sqcup}}
&
&
\mfld(\cB_0)^{\sqcup}  \ar[dl]^{L^\ot_F}  \ar[dd]^{\int F_\ast A}
\\
\mfld(\cB)^{\sqcup}  \ar[r]
&
\cP(\cB)^\amalg  \ar[dr]^{\int A}
&
\\
\disk(\cB)^{\sqcup}  \ar[rr]^A  \ar[u]  
&
&
\cC^\otimes
}
\]
commutes,
where all arrows in the triangle are defined using operadic left Kan extension.  
For the reader uneasy with the (immense) machine of operadic left Kan extension, we will observe the following string of equivalences from unwinding definitions.  Fix a $\cB_0$-manifold $P$.
\begin{eqnarray}
\int_{L_F(P)}A 
&\simeq&
\colim_{\bigsqcup_{j\in J} U_j \longrightarrow  L_F(P)} ~{}~\bigotimes_{j\in J} ~A(U_j)  
\nonumber \\
&\simeq& \colim_{\bigsqcup_{k\in K} V_k \longrightarrow P} ~{}~\colim_{\bigsqcup_{j\in J} U_j \longrightarrow \bigsqcup_{k\in K} F(V_k)} ~{}~ \bigotimes_{j\in J} ~A(U_j)
\nonumber \\
&\simeq&  \colim_{\bigsqcup_{k\in K} V_k \longrightarrow P} ~{}~\bigotimes_{k\in K}~{}~ \colim_{\bigsqcup_{j'\in J'_k} U_{j'} \longrightarrow F(V_k)} ~{}~\bigotimes_{j'\in J'_k} A(U_{j'})
\nonumber \\
&\simeq & \colim_{\bigsqcup_{k\in K} V_k \longrightarrow P}~{}~ \bigotimes_{k\in K}~ \int_{F(V_k)} A
\nonumber \\
&\simeq& \colim_{\bigsqcup_{k\in K} V_k \longrightarrow P}~{}~ \bigotimes_{k\in K}~ F_\ast A(V_k)
\nonumber \\
&\simeq& \int_P  F_\ast A ~.
\nonumber
\end{eqnarray}
\end{proof}

\subsection{Factorization homology over a closed interval}\label{section:interval}

Recall the category of basics $\cI$ of Example~\ref{I}.  
Any $\cB$-manifold $X$ with a collar-gluing can be thought of as an $\cI$-family of $\cB$-manifolds. (See Example~\ref{collar-I}.) So the factorization homology of $X$ can be computed by pushing forward to the closed interval $[-\infty,\infty]$, and taking factorization homology of $[-\infty,\infty]$ with coefficients in the resulting $\disk_1^{\partial,\fr}$-algebra. Thankfully, the structure of such an algebra is simple.

\begin{example}\label{example:disk1algebra}
Let $\cC$ be the (nerve of the) category of vector spaces over $k$, with monoidal structure given by $\tensor_k$. An object of $\Alg_{\disk^{\partial , \fr}_1}(\cC^\ot)$ is the following data:
	\begin{enumerate}
	\item {\em A unital associative algebra $A_\RR = A(\RR)$.}  An embedding $\RR\sqcup \RR \to \RR$ determines the multiplication map.  The inclusion of the empty manifold into $\RR$ determines the unit of $A_\RR$.  Associativity follows from factoring an oriented embedding $(\RR \sqcup \RR \sqcup \RR) \to \RR$ through an oriented embedding $\RR \sqcup \RR \to \RR$ in evident ways.
	\item {\em Two vector spaces $M_- = A(\RR_{\geq -\infty})$ and $M_+ = A(\RR_{\leq \infty})$ each receiving a map $*_i: k \to M_i$.} 
	The maps $*_i$ correspond to the inclusion of the empty manifold into $\RR_{\geq -\infty}$ or into $\RR_{\leq \infty}$. The index $i$ is over the values $-$ and $+$.
	\item {\em A right module action $\mu_-: M_- \tensor A_\RR \to M_-$ compatible with the map $*_-$}, in the sense that the following diagram commutes:
		\[
		\xymatrix{
		k \ar[r]^{1_{A_\RR}} \ar[drrr]^{*_-}
		&A \ar[rr] ^{*_- \tensor id_{A_\RR}}
		&&M_- \tensor A_\RR \ar[d]^{\mu_-} \\
		&&& M_-~.
		}
		\]
The action $\mu_-$ is determined by an embedding $\RR_{\geq-\infty}\sqcup \RR \to \RR_{\geq - \infty}$.  		
And likewise for $M_+$, which is a left module over $A_\RR$ rather than a right module.
	\end{enumerate}
\end{example}

In general, the object $A_\RR:=A(\RR)$ is a unital $A_\infty$ algebra and the objects $M_-:=A(\RR_{\geq -\infty})$ and $M_+:=A(\RR_{\leq \infty})$ are right- and left- modules over this algebra, respectively. 
We claim that the factorization homology of $A$ over a closed interval is precisely the (geometric realization) of the two-sided bar construction $\sbar_\bullet(M_-,A_\RR,M_+)$:

\begin{prop}\label{lemma:interval}
Let $A$ be a $\cI$-algebra. There is a natural equivalence
	\begin{equation}\label{bar-homology}
	\int_{[-\infty,\infty]} A
	~{}~\simeq~{}~
	A(\RR_{\geq -\infty}) \underset{A(\RR)}\otimes A(\RR_{\leq \infty})
	\end{equation}
between the factorization homology of the closed interval and the two-sided bar construction of the left $A(\RR)$-module $A(\RR_{\leq \infty})$ and the right $A(\RR)$-module $A(\RR_{\geq -\infty})$.
\end{prop}

Recall from its defining expression that the lefthand side of~(\ref{bar-homology}) is computed as a colimit over the quasi-category $\diskover$.  We study this quasi-category directly.

The underlying space of an arbitrary object $X\in \disk_n^{\partial,\fr}$ is canonically of the form 
	\[
	\iota X = \RR_{\geq -\infty}^{\sqcup \epsilon_-} \sqcup \RR^{\sqcup p} \sqcup \RR_{\leq \infty}^{\sqcup \epsilon_+}~.
	\]
For $X$ to admit a morphism to $[-\infty,\infty]$ we must have $\epsilon_\pm \in \{0,1\}$.  
We thus obtain a map
	\[
	\epsilon : \diskover \to \Delta^1 \times \Delta^1
	\]
described as the assignment $(X\to [-\infty,\infty]) \mapsto (\epsilon_-,\epsilon_+)$.

\begin{lemma}\label{lemma:cofinaldisk}
The inclusion of the fiber
	\[
	\epsilon^{-1}(1,1) \subset \diskover
	\]
is cofinal.
\end{lemma}

\begin{proof}
Let $Z=(X\xra{f}[-\infty,\infty])\in \diskover$ be an arbitrary vertex. 
We must prove that the quasi-category
	\[
	\cD = \epsilon^{-1}(1,1) \times_{\diskover} \bigl(\diskover\bigr)_{Z/}
	\]
has a weakly contractible classifying space. (See for instance Theorem 4.1.3.1 of Lurie's~\cite{topos}.)
Define the object $Z'=(X' \xra{f'} [-\infty,\infty]) \in \epsilon^{-1}(1,1)$ as follows.  
The $\cI$-manifold 
	\[
	X'  = \RR_{\geq -\infty}^{\sqcup 1-\epsilon_-} \sqcup X \sqcup \RR_{\leq \infty}^{\sqcup 1 - \epsilon_+}~.
	\]
Notice the inclusion $X \subset X'$ of $\cI$-manifolds.  
By construction, for $Y \in \epsilon^{-1}(1,1)$, any morphism $X \to Y\in \disk^{\partial,\fr}_1$ canonically factors as $X \subset X' \to Y$.

Choose any morphism $X'\xra{f'} [-\infty,\infty]$ so that a composite $X\subset X'\xra{f'} [-\infty,\infty]$ is $f$ -- such a morphism $f'$ exists as evident by replacing $f$ with $f\circ g$ where $X\xra{g}X$ is isotopic to the identity map and the closure $\ov{g(\iota X)} \subset \iota X$ is compact.  
By construction $Z\to Z'$ is an edge in $\diskover$ and thus is a vertex of $\cD$.  
A similar argument indicated for the existence of $f'$ shows that $(Z\to Z')\in \cD$ is initial.  
It follows that $|\cD|$ is contractible.  
\end{proof}

\begin{lemma}
There is an equivalence of quasi-categories $\epsilon^{-1}(1,1) \xra{\simeq} \Delta^{\op}$.

\end{lemma}

\begin{proof}
The map is given by $(X\xra{f} [-\infty,\infty])\mapsto \pi_0\bigl([-\infty,\infty]\smallsetminus f(\iota X)\bigr)$, the path components of the compliment of the image, which is equipped with a linear order from that of $[-\infty,\infty]$; this linearly ordered set is evidently finite and is never empty since $[-\infty,\infty]$ is not an element of $\disk^{\partial,\fr}_1$.
It is straight forward to verify that this indeed describes a functor of quasi-categories.
That this map is an equivalence is standard.  

\end{proof}

\begin{proof}[Proof of Proposition~\ref{lemma:interval}]
We contemplate the composition
	\[
	\Delta^{\op} \xla{\simeq} \epsilon^{-1}(1,1) 
	\to \diskover \xra{A} \cC~.
	\]
The value of this map on $[p]$ is canonically equivalent to $A(\RR_{\geq - \infty} \sqcup \RR^{\sqcup p} \sqcup \RR_{\leq \infty}) \simeq M_-\ot A_\RR^{\ot p}\ot M_+$.  
The value of this map on morphisms gives the simplicial structure maps of the two-sided bar construction $\sbar_\bullet(M_-,A_{\RR},M_+)$. 
We have established the string of canonical equivalence
\[
\int_{[-\infty,\infty]} A = \colim\bigl( \diskover  \xra{A} \cC\bigl) \xla{\simeq}\colim\bigl(\Delta^{\op} \xra{\sbar_\bullet(M_-,A_\RR,M_+)} \cC\bigr) = M_-\ot_{A_\RR} M_+
\]
in which the leftward equivalence is from Lemma~\ref{lemma:cofinaldisk}.

\end{proof}

\subsection{Factorization homology satisfies excision}

 By combining Proposition~\ref{lemma:interval} with the push-forward formula of Theorem~\ref{pushforward-formula}, we see that factorization homology satisfies the {\em excision property}:

\begin{theorem}[Excision]\label{thm:excision}
Let $X=X_-\cup_{RV} X_+$ be a collar-gluing of $\cB$-manifolds.  
Then there is a natural equivalence
\[
\int_{X_-} A~{}~{}~  \bigotimes_{\int_{RV} A}~{}~{}~ \int_{X_+} A  ~{}~{}~  \simeq ~{}~{}~ \int_X A~.
\]
\end{theorem}
\begin{proof}
As discussed in Example~\ref{collar-I}, collar-gluing determines an $\cI$-family $G\colon \cI \to \mfld(\cB)$ with $G(\RR_{\geq -\infty}) = X_-$, $G(\RR_{\leq \infty}) = X_+$, and $G(\RR) = RV$.  
We establish the string of equivalences
\begin{eqnarray}
\int_{X_-} A \bigotimes_{\int_{RV} A} \int_{X_+} A 
& \simeq &
\int_{G(\RR_{\geq -\infty})}A \bigotimes_{\int_{G(\RR)} A} \int_{G(\RR_{\leq \infty})} A 
\nonumber \\
&\simeq &
\int_{\RR_{\geq -\infty}}G_\ast A \bigotimes_{\int_{\RR} G_\ast A} \int_{\RR_{\leq \infty}} G_\ast A 
\nonumber \\
& \xra{\simeq} &
\int_{[-\infty,\infty]} G_\ast A
\nonumber \\
&\simeq &
\int_{L_G([-\infty,\infty])} A
\nonumber \\
&\xra{\simeq} &
\int_X A~.  
\nonumber
\end{eqnarray}
The first equivalence defines the left-hand side.  The second and fourth are the pushforward formula of~\ref{pushforward-formula}.  The third is Proposition~\ref{lemma:interval}.  The final equivalence follows because the canonical map $H([-\infty,\infty]) \to X$ is a refinement of $\cB$-manifolds, and therefore an equivalence by Theorem~\ref{no-pre-mans}.   

\end{proof}

\subsection{Homology theories}
In classical algebraic topology, the excision axiom of a homology theory can be phrased as sending pushout squares to certain colimits in chain complexes.  
This axiom cannot be simply parroted  when defining a homology theory for $\cB$-manifolds, a first obstruction being that $\bman$ does not admit pushouts.  We account for this by distinguishing a class of diagrams in $\bman$ which play the role of pushout diagrams -- these are \emph{collar-gluings} (see Definition~\ref{collar-gluing}).  We go on to define a homology theory for $\cB$-manifolds as a symmetric monoidal functor sending collar-gluings to certain colimits.

As usual, fix a category of basics $\cB$ and a symmetric monoidal quasi-category $\cC^\otimes$ which satisfies condition~($\ast$).  
Let $\mfld(\cB)^{\sqcup} \xra{H} \cC^\otimes$ be a map of symmetric monoidal quasi-categories.  
Through Example~\ref{collar-I}, any collar-gluing of $\cB$-manifolds $X=X_-\cup_{RV} X_+$ 
determines a $\cI$-family of $\cB$-manifolds $\cI\xra{G} \mfld(\cB)$, so there results a map of symmetric monoidal quasi-categories
$
\cI^{\sqcup} \xra{G^{\sqcup}}  \mfld(\cB)^{\sqcup}  \xra{H} \cC^\otimes~.  
$
Consider the universal arrow from the 
Kan extension $\int_{[-\infty,\infty]} HG^{\sqcup} \to H(X)$.  
In light of Theorem~\ref{thm:excision}, this universal arrow can be written as
\begin{equation}\label{hmlgy?}
H(X_-) \bigotimes_{H(RV)} H(X_+)  \to H(X)~.
\end{equation}

\begin{definition}[Homology theory]\label{homology-theory}
Let $\cC^\otimes$ be a symmetric monoidal category satisfying $(*)$.  A \textit{$\cC$-valued homology theory for $\cB$-manifolds} is a map of symmetric monoidal quasi-categories
	\[
	H\colon \bman^{\sqcup} \to \cC^\otimes
	\]
that preserves sequential colimits, and such that for every collar-gluing $X = X_- \cup_{RV} X_+$ of $\cB$-manifolds, the canonical morphism~(\ref{hmlgy?})
is an equivalence in $\cC$.  Denote the full sub-quasi-category 
	\[
	\bH(\bman, \cC^\otimes) \subset \Fun^\otimes(\bman^{\sqcup}, \cC^\otimes)
	\]
spanned by the homology theories.  
\end{definition}

Note that since a homology theory preserves sequential colimits, it is determined entirely by its behavior on $\mfld^{\sf fin}(\cB) \subset \mfld(\cB)$.

\begin{example}
Let $\cB$ be a category of basics and $\cC^\otimes$ be a symmetric monoidal quasi-category satisfying $(*)$.
By Theorem~\ref{thm:excision} there is a canonical factorization
\[
\xymatrix{
\Alg_{\disk(\cB)}(\cC^\ot) \ar[rr]^{\int}  \ar[dr]^{\int}
&
&
\Fun^\otimes(\bman^{\sqcup}, \cC^\otimes)
\\
&
\bH(\bman,\cC^\otimes)  \ar[ur]
&
.
}
\]
\end{example}

\begin{theorem}[Characterization of homology theories]\label{characterization}
Let $\cB$ be a category of basics and $\cC^\otimes$ a symmetric monoidal quasi-category satisfying $(*)$. The (restricted) adjunction~(\ref{adj})
\[
\int \colon \Alg_{\disk(\cB)}(\cC^\ot) \leftrightarrows \bH(\bman , \cC^\otimes) \colon \rho
\]
is an equivalence of quasi-categories.  
\end{theorem}

\begin{proof}
We will make use of the identification $\bman^\cI \cong \bman^{\sf fin}$ of Theorem~\ref{B-collar=fin}.  
Let $X\in \mfld(\cB)^{\sf fin}$ and $A\in \Alg_{\disk(\cB)}(\cC^\ot)$.
Recall the defining expression $\int_X A = \colim \bigotimes_{k\in K} A(U_k)$ where this colimit is over the quasi-category $\disk(\cB)_{/X}$.  If $X = \bigsqcup_{l\in L} V_l \in \disk(\cB)$ is itself a disjoint union of basics, then $\disk(\cB)_{/X}$ has a terminal object $X\xra{=}X$ and so the inclusion $\bigotimes_{l\in L} A(V_l) \xra{\simeq} \int_X A$ is an equivalence.  
It follows that the unit map of the adjunction $1 \xra{\simeq} \rho \int$ is an equivalence.  

For the converse, consider the counit 
\begin{equation}\label{counit}
\int_{(-)} \rho H  \longrightarrow H(-)~.
\end{equation}
If $(-)=U\in \cB$ then~(\ref{counit}) is an equivalence because $\int$ is defined as a (Kan) extension.  We proceed by transfinite induction.  
Consider the collection of intermediary full sub-symmetric monoidal quasi-categories $\disk(\cB)^{\sqcup} \subset \cM\subset (\mfld(\cB)^{\sf fin})^{\sqcup}$ for which the counit map is an equivalence for $(-)$ any object of $\cM$.  
This collection is partially ordered by inclusions.  Clearly $\disk(\cB)^{\sqcup}$ is a member of this collection.  
Let $\cM_0$ be a maximal such member. 
We wish to show $(\mfld(\cB)^{\sf fin})^{\sqcup} \subset \cM_0$.  
Suppose $X = X_-\cup_{RV}X_+$ with $X_\pm~ ,~ RV \in \cM_0$.  
Because $H$ is a homology theory, there is the equivalence
\begin{equation}\label{thing-one}
H(X_-)\bigotimes_{H(RV)} H(X_+) \xra{\simeq} H(X)~.  
\end{equation}
By excision (Theorem~\ref{thm:excision}), there is the equivalence
\begin{equation}\label{thing-two}
\int_{X_-} \rho H \bigotimes_{\displaystyle\int_{RV} \rho H}  \int_{X_+} \rho H \xra{\simeq} \int_X \rho H~.
\end{equation}
The counit map~(\ref{counit}) describes a map from diagram~(\ref{thing-one}) to diagram~(\ref{thing-two}).  
From the invariance of coend under equivalences in the quasi-category $\cC$ and by construction of $\cM_0$, because $X_\pm,  RV\in \cM_0$, the counit map~(\ref{counit}) is an equivalence on the lefthand sides of diagrams~(\ref{thing-one}) and~(\ref{thing-two}).   
It follows that the counit~(\ref{counit}) is an equivalence on the righthand sides of~(\ref{thing-one}) and~(\ref{thing-two}).  We have shown that $\cM_0$ is closed under collar-gluing.  It follows from Theorem~\ref{collar=fin}  that $(\mfld(\cB)^{\sf fin})^{\sqcup} \subset \cM_0$.  
\end{proof}

\subsection{Classical homology theories}\label{classics}
Here we document familiar examples of homology theories.  We will specialize to the stable quasi-category $\cS p$ of spectra, but this is insubstantial, for instance chain complexes would do just as well.  
The quasi-category of spectra $\cS p$ is (the underlying quasi-category of) a symmeric monoidal quasi-category under wedge sum $\vee$, with $\ast$ as the unit, likewise for smash product $\wedge$, with the sphere spectrum $\SS$ as the unit.  
We will exploit the following (related) features of a stable quasi-category $\cM$:
\begin{itemize}
\item
A pullback diagram is a pushout diagram in $\cM$.  In particular, the categorical coproduct and product in $\cM$ are equivalent.
\item
The Yoneda embedding $\cM \to \cP(\cM)$ factors through $\cS p^{\cM^{\op}} \xra{\Omega^\infty} \cP(\cM)$.  In particular, we can regard $\cM(E,E')$ as a spectrum for each pair of objects $E,E'\in \cM$.  
\end{itemize}

Fix a category of basics $\cB$.  
Fix a spectrum $E$.  Consider the map of symmetric monoidal quasi-categories
\[
E_{\sf c} \colon \mfld(\cB)^{\sqcup} \to \cS p^\vee
\]
given on vertices by assigning to $\bigl(J,(X_j)\bigr)$ the data $\bigl(A,(\cS p(\Sigma^\infty(\iota X_j)^\ast , E))\bigr)$;  the value on morphisms (and higher simplices) becomes evident upon observing the canonical equivalence
\[
\bigvee_{j\in J} \cS p\bigl(\Sigma^\infty(\iota X_j)^\ast   , E\bigr)  
\xra{\simeq}
\cS p\Bigl(\Sigma^\infty\bigl(\iota(\bigsqcup_{j\in J} X_j)\bigr)^\ast  , E\Bigr)~.
\]
The homotopy groups $\pi_\ast E_{\sf c}(X) = E^\ast_{\sf c}(\iota X)$ are the compactly supported $E$-cohomology groups of the underlying space $\iota X$.

\begin{lemma}\label{usuals}
For any spectrum $E$, the symmetric monoidal map $E_{\sf c}$ is a homology theory.

\end{lemma}
\begin{proof}
Let $X = X_-\cup_{RV} X_+$ be a collar-gluing of $\cB$-manifolds.  
The resulting map
	\[
	\iota X^\ast \to \iota X_- \times_{\bigl((-1,1)\times \iota V\bigr)^\ast} \iota X_+^\ast
	\] 
induced by collapse maps is a a continuous bijection among compact spaces, and is therefore a homeomorphism.  Moreover, because $[-1,1]\times \iota V \to \iota X_\pm$ is a cofibration, it follows that this pullback is in fact a homotopy pullback.  It follows that $\Sigma^\infty(\iota X)^\ast \simeq \Sigma^\infty (\iota X_-)^\ast \coprod_{\Sigma^\infty (\iota RV)^\ast} \Sigma^\infty (\iota X_+)^\ast$ can be canonically written as a pushout in $\cS p$.  
Applying $\cS p(\Sigma^\infty -,E)$ to this pushout gives a pullback, which is again a pushout in $\cS p$.  That is, the universal arrow
\begin{equation}\label{po}
E_{\sf c}(X_-) \underset{E_{\sf c}(RV)}\amalg E_{\sf c}(X_+) \xra{\simeq} E_{\sf c}(X)
\end{equation}
is an equivalence of spectra.  

We now contemplate the $\disk^{\partial,\fr}_1$-algebra in $\cS p^\vee$ associated to the given collar-gluing $X=X_-\cup_{RV} X_+$.  
Because $(\iota RV)^\ast \cong \Sigma (\iota V)^\ast$ then the value $E_{\sf c}(RV) \cong \Sigma^{-1} E_{\sf c}(V)$ is the desuspension
and the relevant $\disk^{\fr}_1$-algebra is given by the associative multiplication rule
\[
\Sigma^{-1}E_{\sf c}(V) \vee  \Sigma^{-1}E_{\sf c}(V) \to \Sigma^{-1}E_{\sf c}(V)
\]
given by the universal arrow from the coproduct.  The left action is 
\[
E_{\sf c}(X_+) \vee \Sigma^{-1} E_{\sf c}(V)  \to E_{\sf c}(X_+)
\]
which, on the first summand is the identity and on the second is the restriction along the collapse map $\iota X_-^\ast \to \iota RV$.  Because the symmetric monoidal structure of $\cS p^\vee$ is coCartesian, it follows that the universal arrow to the bar construction
\[
E_{\sf c}(X_-)\underset{E_{\sf c}(RV)}\amalg E_{\sf c}(X_+) \xra{\simeq} E_{\sf c}(X_-)\underset{E_{\sf c}(RV)}\vee E_{\sf c}(X_+)
\]
is an equivalence.  Combined with the conclusion~(\ref{po}), this finishes the proof.

\end{proof}

\section{Nonabelian Poincar\'e duality}\label{section:poincare}
In this section we give a substantial generalization of Poincar\'e duality from the classical situation of smooth manifolds.  It is a generalization still of the nonabelian Poincar\'e duality of Lurie~\cite{dag} and Salvator\'e~\cite{salvatore} to the setting of structured singular manifolds displayed in this article.  
One can regard the statement of Poincar\'e duality below as evidence that any homology theory for $\cB$-manifolds which detects more than the stratified homotopy type of the manifolds cannot arise from a ``group-like'' $\disk(\cB)$-algebra
(see Remark~\ref{remark:grouplike} which explains these quotes).

Our discussion of duality is two-fold. We first discuss what one might deem `abelian' Poincar\'e duality.  For this we specialize to factorization homology with coefficients in spectra and construct a dualizing (co)sheaf, but other stable settings such as chain complexes would do.  We then take the coefficients $\cC^\ot = \cS^\times$ to be spaces and discuss a `nonabelian' version of Poincar\'e duality.  
In subsequent work we will develop the theory for more general $\cC^\otimes$ and see Poincar\'e duality as an instance of Koszul duality.  
There it will be shown that Poincar\'e duality characterizes $\cB$-manifolds in an appropriate sense.  

\subsection{Dualizing data}
For simplicity, we specialize our discussion to $\cB = \bsc_n$ and work strictly with \emph{finite} singular $n$-manifolds.  As so, we omit the superscript $\snglr_n^{\sf fin}$ from the notation.

\subsubsection{Duals}

Let us recall the notion of duals for a symmetric monoidal quasi-category from~\textsection4.2.5 of~\cite{dag}.  We specialize our discussion to spectra in as much as it is the initial stable quasi-category receiving a map from spaces, see~\cite{dag},~\S1.4.4, but other such stable quasi-categories, for instance chain complexes, would do just as well.  
For $E$ a spectrum, say it has a (Spanier-Whitehead) dual if there exists
\begin{itemize}
\item a spectrum $E^\vee$ ,
\item a map of spectra ${\sf coev}\colon \SS  \to E\otimes E^\vee$ from the sphere spectrum,
\item a map of spectra ${\sf ev}\colon E^\vee \ot E \to \SS$ to the sphere spectrum,
\item a triangle 
\[
\xymatrix{
E \ar[dr]^-{{\sf coev}\ot 1} \ar[rr]^=
&&
E
\\
&
E\ot E^\vee \ot E  \ar[ur]^-{1\otimes {\sf ev}} 
&
}
\]
which commutes up to a specified homotopy,

\item a triangle
\[
\xymatrix{
E^\vee  \ar[dr]^-{1 \ot {\sf coev}} \ar[rr]^=
&&
E^\vee
\\
&
E^\vee \ot E  \ot E^\vee \ar[ur]^-{{\sf ev}\ot 1} 
&
}
\]
which commutes up to a specified homotopy.
\end{itemize}
We often refer to this data simply as $E^\vee$.  
For a general discussion of duals in monoidal quasi-categories, visit~\S4.2.5~\cite{dag}.

Let $E$ be a spectrum which has a dual.  A choice of the above data canonically determines the map of spectra
\begin{equation}\label{dual-equiv}
\cS p(E , -) \xra{-\ot E^\vee} \cS p(E\ot E^\vee , -\ot E^\vee) \xra{{\sf coev}^\ast} \cS p(\SS, - \ot E^\vee) \simeq (-\ot E^\vee)
\end{equation}
which is an equivalence (here we're using that the Yoneda functor $\cM \to \cP(\cM)$ factors through $\cS p^{\cM^{\op}} \xra{\Omega^\infty} \cP(\cM)$ for any stable quasi-category $\cM$).  
The simplicial set of choices of this data is a contractible Kan complex. 
Clearly, if $E$ and $E'$ both have duals then so does $E\ot E'$ and its dual is $E^\vee \ot (E')^\vee$.  
Denote by $\cS p^\wedge_D\subset \cS p^\wedge$ the full sub-symmetric monoidal quasi-category whose vertices are those spectra which have duals.  
The assignment $E\mapsto E^\vee$ can be made into a functor $(-)^\vee \colon (\cS p^\wedge_D )^{\op} \to \cS p^\wedge_D$ where the opposite is taking place over $\mathsf{Fin}_\ast$.

The lemma below is Atiyah duality~\cite{atiyah} for singular manifolds.  
\begin{lemma}\label{atiyah}
The map $\Sigma^\infty(\iota -)^\ast\colon  (\mfld(\cB)^{\sqcup})^{\op} \to \cS p^\vee$ factors through $\cS p^\vee_{D}$ the dualizable spectra.  
\end{lemma}
\begin{proof}
We need only verify that $\Sigma^\infty(\iota X)^\ast$ has a Spanier-Whitehead dual for each (finite) $\cB$-manifold $X$.  Because the $\cB$-structure is irrelevant here, we can assume $X$ is a (finite) singular $n$-manifold.  
From Lemma~\ref{regular-neighborhood} there is a (conically) smooth proper embedding $\iota X \subset \RR^N$ for some $N$, in addition to a $\delta>0$ for which a $\delta$-neighborhood $\nu$ of $\iota X$ is a regular neighborhood.  
Denote the deformation retraction $p\colon \nu \to \iota X$.  Consider the continuous maps
\[
{\sf coev}\colon S^N \cong (\RR^N)^\ast \to \nu^\ast \wedge \iota X^\ast~,~{}~{}~\Bigl(v\mapsto \bigl(v,p(x)\bigl),\text{ or } v\mapsto \ast\text{ if }v\notin \nu\Bigr)~;
\]
and
\[
{\sf ev}\colon \nu^\ast \wedge \iota X^\ast \to (B_\delta^N)^\ast \cong S^N
\]
given by $(v,x)\mapsto v-x$ if $\lVert v-x \rVert<\delta$ and $(v,x)\mapsto \ast$ otherwise.   
These maps (together with the obvious triangles) exhibit $\Sigma^\infty \nu^\ast$ as a dual of $\Sigma^\infty (\iota X)^\ast$.  
\end{proof}

\subsubsection{Dualizing (co)sheaf}

Define the composite map of symmetric monoidal quasi-categories
\[
\DD^\ast \colon \mfld(\cB)^{\sqcup} \xra{\Sigma^\infty(\iota -)^\ast} (\cS p^\vee_D)^{\op} \xra{(-)^\vee} \cS p^\vee_D
\]
which we refer to as the \emph{dualizing cosheaf (for $\cB$-manifolds)}.  The usage of the term cosheaf is justified by the following lemma.  
\begin{lemma}\label{D-colim}
The map $\DD^\ast$ is a homology theory.

\end{lemma}
\begin{proof}
From~(\ref{dual-equiv}) there is a canonical equivalence $\DD^\ast \simeq \SS_c$ (see~\S\ref{classics} for the notation).  The lemma follows from Theorem~\ref{usuals}.  

\end{proof}

Denote the restriction $\omega = \DD^\ast_|\colon \disk(\cB)^{\sqcup} \to \cS p^\vee$.  Explicitly, for $U\in \cB$ a basic over $U^n_Y\in \bsc_n$ with $Y$ a compact $(k-1)$-manifold, the value
\[
\omega(U) = \Bigl(\Sigma^\infty\bigl(\Sigma^{n-k}S( \iota Y)\bigr)\Bigr)^\vee
\]
where $S$ here denotes the unreduced suspension, regarded as a based space with base point the north pole.  With this formula we can (non-canonically) identify the stalk of $\DD^\ast$ at $x\in \iota X$ as $\omega_x(X):=\lim_{x\in U'\to X} \omega(U') \simeq \omega(U) =  (\Sigma^\infty (\Sigma^{n-k} S (\iota Y)))^\vee$ where $U$ is the unique (up to non-canonical equivalence) basic $U\to X$ whose image contains $x$.

There is the immediate corollary of Theorem~\ref{characterization}.  
\begin{cor}
The universal map of spectra
\[
\int \omega \xra{\simeq} \DD^\ast
\]
is an equivalence.  In particular, for $X$ a finite $\cB$-manifold
there is an equivalence of spectra 
\[
\DD^\ast(X) \simeq  \colim_{U\to X}~ (\Sigma^\infty (\iota U)^\ast)^\vee.
\]  

\end{cor}

\begin{theorem}[Poincar\'e duality]\label{abelian}
Let $X$ be a finite $\cB$-manifold and let $E$ be a spectrum. There is a canonical equivalence
\[
\int_X E\wedge \omega \simeq E_{\sf c}(X)~.  
\]

\end{theorem}
\begin{proof}
This is the observation $E\wedge \omega (U) = E\wedge \DD^\ast (U) \simeq E_{\sf c}(U)$ combined with Lemma~\ref{usuals} that $E_{\sf c}$ is a homology theory so that Theorem~\ref{characterization} can be employed.

\end{proof}

\begin{example}
Denote by $H\ZZ$ the Eilenberg-MacLane spectrum for $\ZZ$.  
Let $X$ be a $\sD_n^{\sf or}$-manifold, that is, an oriented $n$-manifold.  
For each morphism $U\to X$ from a basic the orientation gives a canonical equivalence $H\ZZ\wedge \omega(U) \simeq H\ZZ\wedge S^{-n}$.  
It follows that 
\[
\int_X H\ZZ\wedge \omega \simeq \Sigma^{-n} H\ZZ \wedge \int_X \SS \simeq  \Sigma^{-n} H\ZZ \wedge \Sigma^\infty (\iota X)_+  = \Sigma^{-n} H\ZZ\wedge (\iota X)_+
\]
Theorem~\ref{abelian} then implies $H\ZZ\wedge (\iota X)_+ \simeq \Sigma^n \cS p\bigl((\iota X)^\ast, H\ZZ \bigr)$.  Upon applying homotopy groups we arrive at $H_\ast( \iota X;\ZZ) \cong H_c^{n-\ast}(\iota X;\ZZ)$.

\end{example}

\subsection{Coefficient systems}
We give an abundant source of $\disk(\cB)$-algebras in spaces from the data of a based right fibration over $\cB$.

\begin{definition}
A \emph{coefficient system} on $\cB$ is a pair $(E,z)$ consisting of a right fibration $E\to \cB$
and a section $z\colon \cB \to E$.  We often simply refer to a coefficient system $(E,z)$ by its right fibration $E$.  
\end{definition}

\begin{example}
Consider a $[n]$-stratified space $Z\to [n]$ equipped with a section $z_0\colon[n]\to Z$.  The assignment $\Theta\colon U\mapsto \Top_{[n]}([\iota] U , Z)$, with base point given by $z_0$, gives a coefficient system on $\bsc_n$.  
For clever choices of $(Z,z_0)$, restricting to various categories of basics $\cB \to \bsc_n$ gives interesting examples of coefficient systems.  

\end{example}

\begin{example}\label{ordinary-coefs}
Because the quasi-category $\sD^{\fr}_n\simeq \ast$, a coefficient system is equivalent to the datum of a based Kan complex $Z$.  
Such a coefficient system is connective exactly if $Z$ is $n$-connective (see Definition~\ref{connective-coefs}).  

\end{example}

\begin{example}\label{coefs-corners}
A coefficient system on $\sD^\partial_n$ is map of fibrations 
\[
\xymatrix{
E_{n-1}  \ar[r]  \ar[d]
&
E_n   \ar[d]
\\
{\BO}(n-1) \ar[r]  
&
{\BO}(n)
}
\]
together with a pair of (compatible) sections of each.

Consider the more elaborate example $\sD_{\langle n \rangle}$ of Example~\ref{corners}.  A coefficient system on $\sD_{\langle n \rangle}$ is the data of a fibration $E_S \to\BO(\RR^S)$ for each subset $S\subset \{1,\dots,n\}$, and for each inclusion $S\subset T$ a map $E_S \to E_T$ over the inclusion $\BO(\RR^S) \xra{-\oplus \RR^{T\smallsetminus S}} \BO(\RR^T)$, which respect composition; together with a compatible section of each of these fibrations.

\end{example}

Fix a coefficient system $E$.  By right Kan extension along $\cB\hookrightarrow \mfld(\cB) \subset \cP(\cB)$, the map $E$ determines a map of quasi-categories
$
\Gamma^E\colon \cP(\cB)^{\op} \to \cS_\ast
$
from the coherent nerve.  
We use the same notation for the restriction
\[
\Gamma^E\colon \mfld(\cB)^{\op} \to \cS_\ast~.
\]
This is given explicitly by assigning to $(X,\widehat{X}\xra{g} \cB)$ the Kan complex of maps of right fibrations $\Map_\cB(\widehat{X},E)$ with base point $zg$.

Let $X$ be a $\cB$-manifold and $K\subset \iota X$ a compact subset.  
Denote by $X\smallsetminus K$ the canonically associated sub-$\cB$-manifold of $X$ associated to the open subset $\iota X\smallsetminus K$.  
Denote the fiber
\[
\xymatrix{
\Gamma^E_K(X) \ar[r] \ar[d]
&
\Gamma^E(X)  \ar[d]
\\
\ast  \ar[r]^-{z}
&
\Gamma^E(X\smallsetminus K)~.  
}
\]
Observe that $K\subset K'$ implies $\Gamma^E_K(X) \to \Gamma^E_{K'}(X)$.  Define 
\[
\Gamma_{\!\sf c}^E(X) = \colim_{K\subset \iota X} \Gamma^E_K(X)~.
\]

Let $X\xra{f} Y$ be a morphism of $\cB$-manifolds and let $K\subset \iota X$ be a compact subset.  
There results the open cover $\iota Y = f(\iota X) \cup_{f(\iota X\smallsetminus K)} \iota Y\smallsetminus f(K)$. 
From Lemma~\ref{covers=colims}, combined with Theorem~\ref{no-pre-mans}, the diagram
\[
\xymatrix{
X\smallsetminus K  \ar[r]  \ar[d]^{f_|}
&
X  \ar[d]^f
\\
Y\smallsetminus f(K)  \ar[r]
&
Y
}
\]
is a colimit diagram in $\cP(\cB)$.  
By construction, $\Gamma^E$ on $\cP(\cB)^{\op}$ preserves limits. 
It follows that
\begin{equation}\label{gamma-collars}
\Gamma^E(Y) \xra{\simeq} \Gamma^E(X)\times_{\Gamma^E\bigl(f(X\smallsetminus K)\bigr)} \Gamma^E\bigl(Y\smallsetminus f(K)\bigr)
\end{equation}
is an equivalence in $\cS_\ast$.  
In particular there is a canonical equivalence $\Gamma^E_{f(K)}(Y) \simeq \Gamma^E_K(X)$.  
There results a map of colimits
\[
\Gamma_{\!\sf c}^E(X) \to \Gamma_{\!\sf c}^E(Y)~.  
\]
We have constructed to each edge $f$ of $\mfld(\cB)$ an edge $\Gamma_{\!\sf c}^E(f)$ in $\cS_\ast$. 
Likewise, there is an association of a simplex $\Gamma_{\!\sf c}^E(\sigma)$ of $\cS_\ast$ to each simplex $\sigma$ of $\mfld(\cB)$, which assembles into a map of quasi-categories
\[
\Gamma_{\!\sf c}^E\colon \mfld(\cB) \to \cS_\ast~.
\]
Clearly, for $J$ a finite set and $(X_j)_{j\in J}$ an $J$-indexed sequence of $\cB$-manifolds then $\Gamma_{\!\sf c}^E\Bigl(\bigsqcup_{j\in J} X_j\Bigr) \xra{\simeq} \prod_{j\in J} \Gamma_{\!\sf c}^E(X_j)$.

\subsection{Poincar\'e duality}
Fix a coefficient system $E\to \cB$.  Define the map of symmetric monoidal quasi-categories
\[
\Gamma_{\!\sf c}^E\colon \mfld(\cB)^{\sqcup} \to \cS_\ast^\times
\] 
as follows.  
Assign to the vertex $\bigl(J,(X_j)\bigr)$ the vertex $\bigl(J,\cP(J)^{\op} \to \cS_\ast\bigr)$ given by $(S\subset J)\mapsto \Gamma_{\!\sf c}^E\Bigl(\bigsqcup_{j\in S} X_j\Bigr)$, the value being canonically equivalent to $\prod_{j\in S} \Gamma_{\!\sf c}^E(X_j)$.  
Assign to the edge 
\\
$\bigl(f,(f_k)_{k\in K}\bigr)\colon \bigl(J,(X_j)\bigr) \to \bigl(K,(Y_k)\bigr)$ the edge
\[
(J,\cP(J)^{\op}\to \cS_\ast)\to (K,\cP(K)^{\op} \to \cS_\ast)
\]
over $f$ determined by the coordinates $\Gamma_{\!\sf c}^E(f_k)\colon \Gamma_{\!\sf c}^E\Bigl(\bigsqcup_{\{f(j)=k\}}X_j\Bigr) \to \Gamma_{\!\sf c}^E(Y_k)$.  
The assignment for higher simplicies is nearly identical but notationally intensive.

Denote the $\disk(\cB)$-algebra
\[
A_E\colon \disk(\cB)^{\sqcup} \to \mfld(\cB)^{\sqcup} \xra{\Gamma_{\!\sf c}^E} \cS^\times_\ast
\]
which is the restriction.  Recall from~(\ref{Bk}) and~(\ref{Mk}) the categories of basics $\cB_{<r}$ and the associated quasi-categories $\mfld(\cB_{<r})$ for $0\leq r \leq n+1$, each equipped with a map to $\cB$ and $\mfld(\cB)$ respectively.  
Via pullback, we use the same notation for 
\begin{itemize}
\item $\Gamma^E$ as a functor $\mfld(\cB_{<r})^{\op} \to \cS_\ast$,
\item $\Gamma_{\!\sf c}^E$ as a functor $\mfld(\cB_{<r}) \to \cS_\ast$.  
\end{itemize}
\begin{definition}\label{connective-coefs}
Say the coefficient system $E$ is \emph{connective} if $\Gamma_{\!\sf c}^E(V)$ is connected for every $V\in \cB_{<n}$.  
\end{definition}

\begin{example}
Let us return to Example~\ref{coefs-corners} for the case of $\sD_n^\partial$.   
For simplicity, let us assume the two fibrations are trivialized with (based) fibers $Z_n$ and $Z_{n-1}$ respectively.  
This coefficient system is connective exactly if $Z_n$ is $n$-connective and the map $Z_{n-1} \to Z_n$ is $n$-connective; the last condition being equivalent to saying the homotopy fiber $F$ of the map $Z_{n-1}\to Z_n$ is $(n-1)$-connective.  
Recall from Example~\ref{deligne-boundary} a consolidation of the data of a $\disk_n^\partial$-algebra.
The associated $\disk^\partial_n$-algebra $A_E$ is the data $(\Omega^n Z_n, \Omega^{n-1}F , {\sf a})$ where $\sf a$ is the action of $\Omega^n Z_n$ on $\Omega^{n-1} F$ from the $\Omega$-Puppe sequence of the fibration $F \to Z_{n-1} \to Z_n$.

Let us examine the case of $\sD_{\langle n\rangle}$.  
For simplicity, let us assume each of the said fibrations is trivial with respective fibers $Z_S$.  Denote by $F_T = \mathsf{hofib}(Z_T \to \holim_{S\subsetneq T} Z_S)$ the total homotopy fiber of the $T$-subcube.  
This coefficient system is connective exactly if each $F_T$ is $(n-|T|)$-connective.  
The associated $\disk_{\langle n \rangle}$-algebra $A_E$ is the data $\bigl((\Omega^{n\smallsetminus S} F_S\bigr)_{S\subset \{1,\dots,n\}}; \bigl(a_{S\subset T})\bigr)$ where $a_{S\subset T}$ is the action of $\Omega^{n\smallsetminus S}F_S$ on $\Omega^{n\smallsetminus T} F_T$ from an elaboration of the $\Omega$-Puppe sequence.  
\end{example}

\begin{theorem}[Non-abelian Poincar\'e duality]\label{PD}
Let $\cB$ be a category of basics and let $(E\to \cB~,~z)$ be a coefficient system.  
Suppose $E$ is connective.
Let $X$ be a finite $\cB$-manifold.
Then there is a natural homotopy equivalence of spaces
\[
\int_X A_E~{}~  \simeq ~{}~ \Gamma_{\!\sf c}^E(X)~.
\]
\end{theorem}

\begin{proof}
From Theorem~\ref{characterization}, we must show $\Gamma_{\!\sf c}^E\colon \bman^{\sqcup} \to \cS^\times_\ast$ satisfies excision.  We will show that the canonical map
\[
\Gamma_{\!\sf c}^E(X_-)\times_{\Gamma_{\!\sf c}^E(R V)} \Gamma_{\!\sf c}^E(X_+) \xra{\simeq} \Gamma_{\!\sf c}^E(X)
\]
is an equivalence for every collar-gluing $X= X_-\cup_{R V} X_+$ of $\cB_{<k}$-manifolds for $0\leq k \leq n+1$, the desired case being $k=n+1$.  
It is sufficient to show that in the fiber sequence of spaces
\[
\Gamma_{\!\sf c}^E(X_-)\times \Gamma_{\!\sf c}^E(X_+) \to \Gamma_{\!\sf c}^E(X) \to \Gamma_{\!\sf c}^E(V)
\]
the base is connected.  
We will do this by induction on $0\leq k \leq  n+1$.  

For the $k=0$, then $V$ is a $(-1)$-manifold and $\Gamma_{\!\sf c}^E(V) = \ast$ is a point.  In particular it is connected.

Now suppose $\Gamma_{\!\sf c}^E(W)$ is connected for each $W\in \mfld(\cB_{<k})$.  
Let $V$ be a finite $\cB_k$-manifold.  From Theorem~\ref{collar=fin} $V$ can be written as a finite iteration of collar-gluings of basics.  We prove $\Gamma_{\!\sf c}^E(V)$ is connected by induction on the number $r$ of iterated collar-gluings to obtain $V$.  
If $r=0$ the statement is vacuously true.   
If $V$ is a basic then by the connectivity assumption $\Gamma_{\!\sf c}^E(V)$ is connected.  
If $r\geq 2$, write $V= V_-\cup_{R W} V_+$ for some $W\in \mfld(\cB_{k-1})$ with both $V_\pm$ given by strictly fewer than $r$ collar-gluings.  By induction on $r$ the two spaces $\Gamma_{\!\sf c}^E(V_\pm)$ are connected.  
There is the fiber sequence
\[
\Gamma_{\!\sf c}^E(V_-)\times \Gamma_{\!\sf c}^E(V_+) \to \Gamma_{\!\sf c}^E(V) \to \Gamma_{\!\sf c}^E(W)~.
\]
By induction on $k$ the base is connected.  It follows that $\Gamma_{\!\sf c}^E(V)$ is connected.  
\end{proof}

\begin{remark}\label{remark:grouplike}
In the case of smooth framed $n$-manifolds, where $\cB=\sD_n^{\fr}$, we have $\sD_n^{\fr}\simeq \ast$. So a connective coefficient system is the datum of an $n$-connective based space $Z$, and the associated $\disk^{\fr}_n$-algebra is $\Omega^n Z$. $\Omega^n Z$ is in fact a group-like $\disk^{\fr}_n$-algebra.
By May's recognition principle, $n$-fold loop spaces on $n$-connective based spaces are all the examples of group-like $\disk_n^{\fr}$-algebras. 
Accordingly, one can think of the data of a connective coefficient system as a generalization of the notion of a group-like algebra, but to the singular and structured setting.

Now, the space of stratified continuous maps is a homotopy invariant of the underlying stratified space $[\iota]X$. 
In this sense, Theorem~\ref{PD} tells us that connective coefficient systems (i.e., `group-like' $\disk(\cB)$-algebras) cannot detect more than the stratified proper homotopy type of singular manifolds. 
\end{remark}

\section{Examples of factorization homology theories}\label{section:examples}

In this section we give examples of factorization homology over singular manifolds. To illustrate the relevance to low-dimensional topology, we show that the free $\disk_{3,1}^{\fr}$-algebra can distinguish the homotopy type of link complements, and in particular defines a non-trivial link invariant.

In this section we will not distinguish in notation between a singular manifold $X$ and its underlying space $\iota X$.  

\subsection{Factorization homology of singular 1-manifolds}

When the target symmetric monoidal quasi-category $\cC^\ot$ is $\m_k^\ot$, the category of $k$-modules for some commutative algebra $k$, then factorization homology of closed 1-manifolds gives variants of Hochschild homology.

The simplest and most fundamental example is factorization homology for framed 1-manifolds, $\mfld_1^{\fr}$. In this case, there is an equivalence between framed 1-disk algebras and associative algebras in $\cC$, $\Alg_{\disk_1^{\fr}}(\cC^\ot)\simeq \Alg(\cC^\ot)$, and we have the following immediate consequence of the excision property of factorization homology (Theorem~\ref{thm:excision}). 

\begin{prop}\label{hochschild} For an associative algebra $A$ in $\m_k$, there is an equivalence \[\int_{S^1}A \simeq \hh_*(A)\] between the factorization homology of the circle with coefficients in $A$ and the Hochschild homology of $A$ relative $k$.
\end{prop}

\begin{proof} We have the equivalences \[\int_{S^1}A\simeq \int_{\RR^1}A\underset{{\underset{{S^0\times\RR^1}}\int\!\! A}}\ot\int_{\RR^1}A\simeq A\underset{A\ot A^{\op}}\ot A\simeq \hh_*(A)\]using excision and a decomposition of the circle by two slightly overlapping hemispheres.\end{proof}

\begin{remark} Lurie in~\cite{dag} shows further that the obvious circle action by rotations on $\int_{S^1}A$ agrees with the usual simplicial circle action on the cyclic bar construction.
\end{remark}

It is interesting to probe this example slightly further and see the algebraic structure that results when one introduces marked points and singularities into the 1-manifolds. Recall the quasi-category $\mfld^{\fr}_{1,0}$ of framed 1-manifolds with marked points, and the sub-quasi-category $\sD^{\fr}_{1,0}$ of framed 1-disks with at most one marked point -- its set of objects is the two-element set $\{U^1_{\emptyset^{-1}} , U^1_{S^0}\}$ whose elements we justifiably denote as $\RR^1:= U^1_{\emptyset^{-1}}$ and $(\RR^1,\{0\}) := U^1_{S^0}$.   So $\Alg_{\disk^{\fr}_{1,0}}(\cC^\ot)$ is equivalent to the quasi-category whose objects are pairs $(A_1, A_{\sf b})$ consisting of an algebra $A_1$ and a unital $A_1$-bimodule $A_{\sf b}$, i.e., a bimodule with an invariant map from the unit. Specifically, $A_1 \simeq A(\RR^1)$ and $A_{\sf b}\simeq A(\RR^1, \{0\})$. The proof of the Proposition~\ref{hochschild} extends mutatis mutandis to the following.

\begin{prop} There is an equivalence \[\int_{(S^1,\ast)} A\simeq \hh_*(A_1,A_{\sf b})\] between the factorization homology of the pointed circle $(S^1,\ast)$ with coefficients in $A =(A_1,A_{\sf b})$ and the Hochschild homology of $A_1$ with coefficients in the bimodule $A_{\sf b}$.
\end{prop}

Finally, we mention the example of factorization homology for $\snglr_1^{\fr}$, singular framed 1-manifolds. In this case, the quasi-category of basic opens $\bsc_1^{\fr}$ has as its set of objects $\{\RR\}\coprod \{(C(J), \sigma)\}$ where the latter set is indexed by finite sets $J$ together with an orientation $\sigma$ of the ordinary $1$-manifold $\bigsqcup_J \RR_{>0} = C(J)\smallsetminus \ast$.

An object $A$ in $\Alg_{\disk(\bsc^{\fr}_1)}(\cC^\ot)$ is then equivalent, by evaluating on directed graphs with a single vertex, to the data of an associative algebra $A(\RR)$ in $\cC^\ot$ and for each pair $i,j\geq 0$ an object $A(i,j) \in \cC$ equipped with $i$ intercommuting left $A(\RR)$-module structures and $j$ intercommuting right compatible $A(\RR)$. One can see, for instance, that the factorization homology of a wedge of two circles with a marked point on each circle, ${(S^1\cup_{\{0\}}S^1, \{1, -1\})}$, can be calculated as \[\int_{(S^1\cup_{\{0\}}S^1, \{1, -1\})}A\simeq A(1,1)\underset{A_1\ot A_1^{\op}}\ot A(2,2)\underset{A_1\ot A_1^{\op}}\ot A(1,1)~.\]

\subsection{Intersection homology}

Recall from section~\S\ref{stratifications} that the underlying space of a singular $n$-manifold $X\in \snglr_n$ has a canonical filtration by its strata $X_0 \subset X_1 \subset \ldots \subset X_n= X$ where each $X_i\smallsetminus X_{i-1}$ is a nonsingular $i$-dimensional manifold. As such, the definition of Goresky and MacPherson's intersection homology~\cite{goreskymacpherson} applies verbatim. That is, we restrict to manifolds $X$ have no codimension-one singularities, $X_{n-1}=X_{n-2}$ and for which the $n$-dimensional open stratum $X_n\smallsetminus X_{n-1}$ is nonempty.

For the definition below we use $j^{\rm th}$-stratum functor $(-)_j\colon \snglr_n \to \snglr_{\leq j}$ of section~$\S$\ref{stratifications}.
\begin{definition} 
Denote the left ideal $\bsc_n^{\sf ps} \to \bsc_n$ spanned by those basics $U$ for which $U_{n-1} = U_{n-2}$.  Define the category of \emph{pseudomanifolds} as $\snglr_n^{\sf ps} = \mfld(\bsc_n^{\sf ps})$ -- its objects are those singular $n$-manifolds for which $X_{n-1} = X_{n-2}$.  
\end{definition}

Continuing, choose a perversity function $p$, i.e., a mapping $p:\{2, 3, \ldots, n\} \ra \ZZ_{\geq 0}$ such that $p(2)=0$ and for each $i>2$ either $p(i)=p(i-1)$ or $p(i) = p(i-1)+1$. Recall the following definition.

\begin{definition}[\cite{goreskymacpherson}] A $j$-simplex $g: \Delta^j \ra X$ is $p$-allowable if, for every $i$ the following bounds on the dimensions of intersections hold:
\begin{itemize}
\item ${\sf dim} \bigl(g(\Delta^j)\cap X_i\bigr) \leq i + j -n +p(n-i)$
\item ${\sf dim} \bigl(g(\partial\Delta^j)\cap X_i\bigr) \leq i + j -n +p(n-i)-1$
\end{itemize} 
\end{definition}

The conditions are clearly stable on under the differential $d$ on singular chains, so this gives the following definition of intersection homology with perversity $p$.

\begin{definition}[\cite{goreskymacpherson}] The intersection homology $\sI_p\sC_\ast(X)$ of $X\in \snglr_n^{\sf ps}$ is the complex of all $p$-allowable singular chains.
\end{definition}

The condition of a simplex being $p$-allowable is clearly preserved by embeddings of singular manifolds: if $f:X \ra Y$ is a morphism in $\snglr_n^{\sf ps}$ and $g:\Delta^j \ra X$ is $p$-allowable, then $f\circ g: \Delta^j\ra Y$ is $p$-allowable. Further, being $p$-allowable varies continuously in families of embeddings. That is, there is a natural commutative diagram:
\[\xymatrix{
\snglr_n(X,Y) \ar[r]\ar[d]&\Map(X,Y)\ar[d]\\
\Map\bigl(\sI_p\sC_\ast(X), \sI_p\sC_\ast(Y)\bigr) \ar[r]&\Map\bigl(\sC_*(X), \sC_*(Y)\bigr)\\}\]
Consequently, intersection homology is defined on the quasi-category $\snglr_n^{\sf ps}$ of $n$-dimensional pseudomanifolds. 
Obviously $\sI_p\sC_\ast (X\sqcup Y) \cong \sI_p\sC_\ast (X) \oplus \sI_p\sC_\ast (Y)$.  
We have the following:
\begin{prop} The intersection homology functor
\[\sI_p\sC_\ast: \snglr_n^{\sf ps} \longrightarrow {\sf Ch}\] defines a homology theory in $\bH( \snglr_n^{\sf ps} , {\sf Ch}^\oplus)$.
\end{prop}
The proof is exactly that intersection homology satisfies excision, or has a version of the Mayer-Vietoris sequence for certain gluings.
\begin{proof} Let $X\cong  X_- \cup_{\RR\times V}X_+$ be a collar-gluing. Then 
\[\xymatrix{
\sI_p\sC_\ast(\RR\times V)\ar[d]\ar[r]&\sI_p\sC_\ast(X_+)\ar[d]\\
\sI_p\sC_\ast(X_-)\ar[r]&\sI_p\sC_\ast(X)\\}\]
is a pushout diagram in the quasi-category of chain complexes. I.e., the natural map 
\\
$\sI_p\sC_\ast(X_-)\oplus_{\sI_p\sC_\ast(\RR\times V)}\sI_p\sC_\ast(X_+) \ra \sI_p\sC_\ast(X)$ is a quasi-isomorphism.
\end{proof}

\subsection{Link homology theories and $\disk^{\fr}_{n,k}$-algebras}

We now consider one of the simplest, but more interesting, classes of singular $n$-manifolds -- that of $n$-manifolds together with a distinguished properly embedded $k$-dimensional submanifold. 
While we specialize to this class of singular manifolds, the techniques for their analysis are typical of techniques that can be used for far more general classes.

Recall from Example~\ref{Ekn-framed} the quasi-category $\mfld_{n,k}^{\fr}$ whose objects are framed $n$-manifolds $M$ with a properly embedded $k$-dimensional submanifold $L\subset M$ together with a splitting of the framing along this submanifold, and the full sub-quasi-category $\disk^{\fr}_{n,k}\subset \mfld_{n,k}^{\fr}$ generated under disjoint union by the two objects $\RR^n := U^n_{\emptyset^{-1}}$ and $(\RR^k\subset \RR^n):= U^n_{S^{n-k-1}}$ with their standard framings.

\subsubsection{Explicating $\disk^{\fr}_{n,k}$-algebras}
Fix a symmetric monoidal quasi-category $\cC^\ot$ satisfying~($\ast$).  
Recall from Example~\ref{product-bundles} the map of quasi-categories $\int_Y\colon \Alg_{\disk(\cB)}(\cC^\ot) \to \Alg_{\disk(\cB_{k+1})}(\cC^\ot)$
defined for any quasi-category of basics $\cB$ of dimension $n$ and any $\cB_{n-k-1}$-manifold $Y$.

\begin{prop}\label{nk} There is a pullback diagram:

\[\xymatrix{
\Alg_{\disk^{\fr}_{n,k}}(\cC^\ot)\ar[d]&\ar[l]{\Alg_{\disk^{\fr}_{k+1}}}\Bigl(\displaystyle\int_{S^{n-k-1}}A~,~ \hh^\ast_{\sD^{\fr}_k}(B)\Bigr)\ar[d]\\
\Alg_{\disk^{\fr}_{n}}(\cC^\ot)\times\Alg_{\disk^{\fr}_{k}}(\cC^\ot)&\ar[l]\{(A,B)\}\\}\]
\end{prop}
That is, the space of compatible $\disk^{\fr}_{n,k}$-algebra structures on the pair $(A,B)$ is equivalent to the space of $\disk^{\fr}_{k+1}$-algebra maps from $\int_{S^{n-k-1}}A$ to the Hochschild cohomology $\hh^\ast_{\sD^{\fr}_k}(B)$; the datum of a $\disk^{\fr}_{n,k}$-algebra is equivalent to that of a triple $(A,B, \mathsf{a})$, where $A$ is a $\disk^{\fr}_n$-algebra, $B$ is a $\disk^{\fr}_k$-algebra, and $\mathsf{a}$ is a map of $\disk^{\fr}_{k+1}$-algebras
	\[
	\mathsf{a}: \int_{S^{n-k-1}} A \longrightarrow \hh^*_{\sD^{\fr}_k} (B)
	\]
-- this is an $S^{n-k-1}$ parametrized family of central $\disk^{\fr}_k$-algebra actions of $A$ on $B$. In essence, Proposition~\ref{nk} is a parametrized version of the higher Deligne conjecture, and in the proof we will rely on the original version of the higher Deligne conjecture.

\begin{proof}

The quasi-category $\disk_{n,k}^{\fr}$ has a natural filtration by the number of components which are isomorphic to the singular manifold $(\RR^k\subset \RR^n)$:
\[
\disk_n^{\fr}=(\disk_{n,k}^{\fr})_{\leq 0} \ra (\disk_{n,k}^{\fr})_{\leq 1} \ra \ldots \colim_i ~(\disk_{n,k}^{\fr})_{\leq i} \simeq \disk^{\fr}_{n,k}
\] 
Consider the second step in this filtration, the full subcategory $(\disk_{n,k}^{\fr})_{\leq 1}$ of $\disk_{n,k}^{\fr}$ whose objects contain at most one connected component equivalent to $(\RR^k\subset \RR^n)$.

Disjoint union endows $(\disk_{n,k}^{\fr})_{\leq 1}$ with a \emph{partially defined} symmetric monoidal structure. 
This partially defined symmetric monoidal structure can be articulated as follows.  Consider the pullback $(\disk^{\fr,\sqcup}_{n,k})_{\leq 1} := (\disk^{\fr}_{n,k})_{\leq 1} \times_{\disk^{\fr}_{n,k}} \disk^{\fr,\sqcup}_{n,k}$ where here we are using the map from the right factor $\sqcup\colon \disk^{\fr,\sqcup}_{n,k} \to \disk^{\fr}_{n,k}$.  
The coCartesian fibration $\disk^{\fr, \sqcup}_{n,k} \to \mathsf{Fin}_\ast$ restricts to a map $(\disk^{\fr,\sqcup}_{n,k})_{\leq 1} \to \mathsf{Fin}_\ast$ which is an inner fibration and for each edge $f$ in $\mathsf{Fin}_\ast$ with a lift $\w{J}_+$ of its source in $(\disk^{\fr,\sqcup}_{n,k})_{\leq 1}$ there is either a coCartesian edge over $f$ with source $\w{J}_+$ or the simplicial set of morphisms over $f$ with source $\w{J}_+$ is empty.  
In this way, by a symmetric monoidal functor from $(\disk^{\fr,\sqcup}_{n,k})_{\leq 1}$ over $\mathsf{Fin}_\ast$ it is meant a map over $\mathsf{Fin}_\ast$ which sends coCartesian edges to coCartesian edges.  

It is immediate that such a symmetric monoidal functor $F$ is equivalent to the data of a $\disk^{\fr}_n$-algebra $F(\RR^n)$ and a $\disk^{\fr}_{n-k}$-$F(\RR^n)$-module given by $F(\RR^k\subset \RR^n)$. Extending such a symmetric monoidal functor $F$ to $\disk_{n,k}^{\fr}$ is thus equivalent to giving a $\disk^{\fr}_k$-algebra structure on $F(\RR^k\subset \RR^n)$ compatible with the $\disk^{\fr}_{n-k}$-$F(\RR^n)$-module. That is, the following is a triple of pullback squares of quasi-categories
\[\xymatrix{
\Alg_{\disk_{n,k}^{\fr}}(\cC^\ot)\ar[d]&\ar[l]\Alg_{\disk^{\fr}_{k}}\bigl(\m_A^{\disk^{\fr}_{n-k}}(\cC)\bigr)\ar[d]\\
\Fun^\ot\bigl((\disk_{n,k}^{\fr})_{\leq 1},\cC\bigr)\ar[d]&\ar[l]\m_A^{\disk^{\fr}_{n-k}}(\cC^\ot)\ar[d]\\
\Alg_{\disk^{\fr}_{n}}(\cC^\ot)&\ar[l]\{A\}~.
\\}\]

Using the equivalence $\m_A^{\disk^{\fr}_{n-k}}(\cC)\simeq \m_{\int_{S^{n-k-1}}A}(\cC)$ of~\cite{cotangent}, we can then apply the higher Deligne conjecture to describe an object of the quasi-category \[\Alg_{\disk^{\fr}_{k}}\Bigl(\m_A^{\disk^{\fr}_{n-k}}(\cC)\Bigr) \simeq \Alg_{\disk^{\fr}_{k}}\Bigl(\m_{\int_{S^{n-k-1}}A}(\cC)\Bigr)~.\] That is, to upgrade a $\disk^{\fr}_k$-algebra $B$ to the structure of a $\disk^{\fr}_k$-algebra in $\int_{S^{n-k-1}}A$-modules is equivalent to giving a $\disk^{\fr}_{k+1}$-algebra map $\mathsf{a}: \int_{S^{n-k-1}}A\ra \hh^*_{\sD^{\fr}_k}(B)$ to the $\sD^{\fr}_k$-Hochschild cohomology of $B$.
\end{proof}

\subsubsection{Hochschild cohomology in spaces}

We now specialize our discussion of $\disk^{\fr}_{n,k}$-algebras to the case where $\cC=\cS^\times$ is the quasi-category of spaces with Cartesian product, but any $\infty$-topos would do just as well. In this case, the $\sD^{\fr}_n$-Hochschild cohomology of an $n$-fold loop space has a very clear alternate description which is given below.

\begin{prop}\label{HH-Aut}
Let $Z=(Z,\ast)$ be a based space which is $n$-connective.  
In a standard way, the $n$-fold based loop space $\Omega^n Z$ is a $\disk^{\fr}_n$-algebra.  
There is a canonical equivalence of $\disk^{\fr}_{n+1}$-algebras in spaces \[\hh^*_{\sD^{\fr}_n}(\Omega^n Z) \simeq \Omega^n \Aut(Z)\] between the $\sD^{\fr}_n$-Hochschild cohomology space of $\Omega^nZ$ and the $n$-fold loops, based at the identity map, of the space of homotopy automorphisms of $Z$.
\end{prop}

\begin{proof} 
In what follows, all mapping spaces will be regarded as based spaces, based at either the identity map or at the constant map at the base point of the target argument -- the context will make it clear which of these choices is the appropriate one.

There are equivalences 
\[
\m_{\Omega^n Z}^{\disk^{\fr}_n}(\cS) \simeq \m_{\int_{S^{n-1}}\Omega^n Z}(\cS) \simeq \m_{\Omega Z^{S^{n-1}}}(\cS)\simeq \cS_{/Z^{S^{n-1}}}
\] 
sending the object $\Omega^n Z$ with its natural $\disk^{\fr}_n$-$\Omega^nZ$-module self-action to the space $Z$ with the natural inclusion of constant maps $Z\ra Z^{S^{n-1}}$. 
Thus, to describe the mapping space \[\hh^*_{\sD^{\fr}_n}(\Omega^nZ) \simeq {\m_{\Omega^nZ}^{\disk^{\fr}_n}}(\Omega^nZ,\Omega^nZ)\] it suffices to calculate the equivalent mapping space $\Map_{/Z^{S^{n-1}}}(Z,Z)$. By definition, there is the (homotopy) pullback square of spaces
\[\xymatrix{
\Map_{/Z^{S^{n-1}}}(Z,Z)\ar[r]\ar[d]& \ast\ar[d]\\
\Map(Z,Z)\ar[r]&\Map(Z,Z^{S^{n-1}})~.
\\}\]

Choose a base point $p\in S^{n-1}$.  The restriction of the evaluation map $ev_p^\ast\colon\Map(Z,Z) \ra \Map(Z,Z^{S^{n-1}})$ is a map of based spaces.   
Thus, the pullback diagram above factorizes as the (homotopy) pullback diagrams
\[\xymatrix{
\Map_{/Z^{S^{n-1}}}(Z,Z)\ar[r]\ar[d]& \ast\ar[d]\\
\Map(Z,Z)^{S^n}\ar[r]\ar[d]&\Map(Z,Z)\ar[d]\\
\Map(Z,Z)\ar[r]&\Map(Z,Z^{S^{n-1})}\\
}\]
where the space of maps from the suspension $S^n= \Sigma S^{n-1}$ to $\Map(Z,Z)$ is realized as the homotopy pullback of the two diagonal maps $\Map(Z,Z) \ra \Map(Z,Z)^{S^{n-1}}$; this is a consequence of the fact that the functor $\Map(-,Z)$ sends homotopy colimits to homotopy limits, applied to the homotopy colimit $\colim(\ast \la S^{n-1}\ra \ast)\simeq S^n$.

Applying the adjunction between products and mapping spaces, we obtain that the Hochschild cohomology space $\Map_{/Z^{S^{n-1}}}(Z,Z)$ is the homotopy fiber of the map $\Map\bigl(S^n, \Map(Z, Z)\bigr) \ra \Map(Z,Z)$ over the identity map of $Z$, which recovers exactly the definition of the based mapping space $\Map_*\bigl(S^n,\Map(Z,Z)\bigr) \simeq \Map_*\bigl(S^n,\Aut(Z)\bigr)$, where the last equivalence follows by virtue of $S^n$ being connected.

\end{proof}

\begin{cor}
Let $Z$ and $W$ be pointed spaces.  Suppose $Z$ is $n$-connective and $W$ is $k$-connective. A $\disk^{\fr}_{n,k}$-algebra structure on the pair $(\Omega^n Z,\Omega^k W)$ is equivalent to the data of a pointed map of spaces
	\[
	Z^{S^{n-k-1}} \longrightarrow {\sf BAut}(W)~.
	\]

\end{cor}

\begin{proof} 
Proposition~\ref{nk} informs us that giving the structure of a $\disk^{\fr}_{n,k}$-algebra on $(\Omega^nZ,\Omega^kW)$ is equivalent to defining a $\disk^{\fr}_{k+1}$-algebra map \[\int_{S^{n-k-1}}\Omega^nZ \longrightarrow \hh^*_{\sD^{\fr}_k}(\Omega^kW)~.\] 
By way of nonabelian Poincar\'e duality (Theorem~\ref{PD}), the factorization homology $\int_{S^{n-k-1}}\Omega^nZ$ is equivalent as $\disk^{\fr}_{k+1}$-algebras to the mapping space $\Omega^{k+1}Z^{S^{n-k-1}}$.  Proposition~\ref{HH-Aut} gives that the Hochschild cohomology $\hh^*_{\sD^{\fr}_k}(\Omega^kW)$ is equivalent to the space of maps to $\Omega^k \Aut(W)$.

Finally, a $(k+1)$-fold loop map $\Omega^{k+1}Z^{S^{n-k-1}} \ra \Omega^k \Aut(W)$ is equivalent to a pointed map between their $(k+1)$-fold deloopings. The $(k+1)$-fold delooping of $\Omega^{k+1}Z^{S^{n-k-1}}$ is $Z^{S^{n-k-1}}$, since $Z$ is $n$-connective; the $(k+1)$-fold delooping of $\Omega^k \Aut(W)$ is $\tau_{\geq k+1}{\sf BAut}(W)$, the $k$-connective cover of ${\sf BAut}(W)$. However, since $Z^{S^{n-k-1}}$ is already $(k+1)$-connective, the space of maps from it into $\tau_{\geq k+1}{\sf BAut}(W)$ is homotopy equivalent to the space of maps into ${\sf BAut}(W)$.
\end{proof}

\subsubsection{Free $\disk^{\fr}_{n,k}$-algebras}\label{free-examples}

Fix a symmetric monoidal quasi-category $\cC^\ot$ satisfying assumption ($\ast$) whose underlying quasi-category is presentable.  
We analyze the factorization homology theory resulting from one of the simplest classes of $\disk^{\fr}_{n,k}$-algebras, that of freely generated $\disk^{\fr}_{n,k}$-algebras. 
That is, there is a forgetful functor
	\begin{equation}\label{forget-two}
	\Alg_{\disk^{\fr}_{n,k}}(\cC^\ot) \to \cC \times \cC~,
	\end{equation}
given by evaluating on the objects $\RR^n$ and $(\RR^k\subset \RR^n)$, and this functor admits a left adjoint $\free_{n,k}$. 
To accommodate more examples, we modify~(\ref{forget-two}).  
Consider the maximal sub-Kan complex $\cE\subset \sD_{n,k}$.  
In light of Lemma~\ref{only-equivalences}, $\cE$ is a coproduct $\End_{\sD_{n,k}}(\RR^n)\coprod \End_{\sD_{n,k}}(\RR^k\subset \RR^n) \simeq \sO(n) \coprod \sO(n,k)$, here $\sO(n,k):= \sO(n-k)\times \sO(k)$.  
There results a map of quasi-categories $\cE \to \sD_{n,k} \to \disk_{n,k} \to \disk_{n,k}^\sqcup$, restriction along which gives the map of quasi-categories
\[
\Alg_{\disk_{n,k}}(\cC^\ot) \to \Map(\cE,\cC) \simeq \cC^{\sO(n)}\times \cC^{\sO(n,k)}
\]
to the quasi-category of pairs $(P,Q)$ consisting of an $\sO(n)$-object in $\cC$ and an $\sO(n,k)$-object in $\cC$.  
We will denote the left adjoint to this map as $\free_\cE$.  
Denote the inclusion as $\delta \colon \cC\times \cC \to \Map(\cE,\cC)$ as the pairs $(P,Q)$ whose respective actions are trivial.

For $X$ a $\sD_{n,k}$-manifold (not necessarily framed) define $\int_X \free^{(P,Q)}_{n,k}:= \int_X \free^{\delta(P,Q)}_\cE$.  
When $X$ is framed (i.e., is a $\sD^{\fr}_{n,k}$-manifold) the lefthand side of this expression has already been furnished with meaning as the factorization homology of $X$ with coefficients in the free $\disk^{\fr}_{n,k}$-algebra generated by $(P,Q)$.  
The following lemma ensures that the two meanings agree.
Recall the forgetful map $\disk^{\fr, \sqcup}_{n,k} \to \disk_{n,k}^\sqcup$.  

\begin{lemma}\label{frame-or-not}
Let $(P,Q)$ be a pair of objects of $\cC$.  
Then the universal arrow
\[
\free_{n,k}^{(P,Q)} \xra{\simeq} \bigl(\free_\cE^{\delta(P,Q)}\bigr)_{|\disk^{\fr, \sqcup}_{n,k}}
\]
is an equivalence.  

\end{lemma}

\begin{proof}

Denote the pullback quasi-category $\cE^{\fr} = \cE\times_{\sD_{n,k}} \sD^{\fr}_{n,k}$.  
Like $\cE$, $\cE^{\fr}\subset \sD^{\fr}_{n,k}$ is the maximal sub-Kan complex (=$\infty$-groupoid).  
The projection $\cE^{\fr} \to \cE$ is a Kan fibration with fibers $\sO(n)$ or $\sO(n,k)$, depending on the component of the base.  
As so, the inclusion of the two objects with their standard framings $\{\RR^n\}\coprod \{\RR^k\subset \RR^n\} \xra{\simeq} \cE^{\fr}$ is an equivalence of Kan complexes.  
Therefore $\Map(\cE^{\fr},\cC) \xra{\simeq} \cC\times \cC$.  

Let us explain the following diagram of quasi-categories
\[
\xymatrix{
\Alg_{\disk_{n,k}}(\cC^\ot)   \ar[rr]^-{\rho}  \ar@(l,l)[dd]^-{\rho}
&&
\Map(\cE,\cC)   \ar@(r,r)[dd]^-{\rho}    \ar@(u,u)[ll]^-{\free_\cE} 
\\
&&
\\
\Alg_{\disk^{\fr}_{n,k}}(\cC^\ot)  \ar[rr]^-{\rho}  \ar[uu]^-{\mathsf{Ran}}  
&&
\Map(\cE^{\fr},\cC)~.   \ar[uu]^-{\mathsf{Ran}}  \ar@(d,d)[ll]^-{\free_{n,k}}
}
\]
Each leg of the square is an adjunction.
All maps labeled by $\rho$ are the evident restrictions.  
The maps denoted as $\mathsf{Ran}$ are computed as point-wise right Kan extension.  
(That is, $\mathsf{Ran}(A)\colon U\mapsto \lim_{U\to U'} A(U')$ where this limit is taking place in $\cC^\ot$ and is indexed by the appropriate over category.  
We emphasize that, unlike the case for left extensions, this point-wise right Kan extension agrees with operadic right Kan extension.)
As so, the straight square of right adjoints commutes.
It follows that the outer square of left adjoints also commutes.

The right downward map is equivalent to that which assigns to a pair of objects $(P,Q)$ with respective actions of $\sO(n)$ and $\sO(n,k)$, the pair $(P,Q)$.  
The map $\delta\colon \cC\times \cC \to \cC^{\sO(n)}\times \cC^{\sO(n,k)}$ is a section to this right downward map $\rho$.    
We have established the string of canonical equivalences
\[
\free_{n,k} \simeq \free_{n,k}\circ\bigl(\rho \circ \delta\bigr) = \bigl(\free_{n,k}\circ \rho\bigr)\circ \delta \simeq \bigl(\rho\circ \free_\cE\bigr)\circ \delta~.
\]
This completes the proof.  

\end{proof}

In order to formulate our main result, we first give the following definition.

\begin{definition} For $M$ a topological space and $P$ an object of $\cC$, the \emph{configuration object of points in $M$ labeled by $P$} is 
\[
\conf^P(M) = \coprod_{j\geq 0} \conf_j(M)\underset{\Sigma_j}\ot P^{\ot j} \in \cC
\] 
where $\conf_j (M)\subset M^{\times j}$ is the configuration space of $j$ ordered and distinct points in $M$.
\end{definition}

For the remainder of the section, assume that the monoidal structure of $\cC^\ot$ distributes over small colimits.

\begin{prop}\label{prop:conf} Let $(P,Q)$ be a pair of objects of $\cC$.  
Let $(L\subset M)$ be a $\sD_{n,k}$-manifold, which is to say, a smooth $n$-manifold and a properly embedded $k$-submanifold.  
There is a natural equivalence
	\[
	\int_{(L\subset M)} \free_{n,k}^{(P,Q)}
	\simeq
	\conf^{P}(M\smallsetminus L) \ot \conf^{Q}(L)\]
between the factorization homology of $(L\subset M)$ with coefficients in the $\disk_{n,k}$-algebra freely generated by $(P,Q)$ and the tensor product of the configurations objects of the link complement $M\smallsetminus L$ and the link $L$ labeled by $P$ and $Q$, respectively.
\end{prop}

We make several remarks before proceeding with the proof of this result.

\begin{remark} We see from this result with $(n,k)=(3,1)$ that factorization homology can distinguish knots. For instance, the unknot, whose knot group is $\ZZ$, and the trefoil knot, whose knot group is presented by $\langle x,y | x^2 = y^3 \rangle$, give rise to different factorization homologies.
\end{remark}

\begin{remark} Specializing to the case where the link $L$ is empty, we obtain the equivalence $\int_M \free_n^P \simeq \conf^P(M)$. Consequently, factorization homology is not a homotopy invariant of $M$, in as much as the homotopy types of the configuration spaces $\conf_j(M)$ are sensitive to the homeomorphism (or, at least, the simple homotopy) type of $M$, see~\cite{simple}. This is in contrast to the case in which the $\disk^{\fr}_n$-algebra $A$ comes from an $n$-fold loop space on an $n$-connective space, in which case nonabelian Poincar\'e duality~(Theorem~\ref{PD}) implies that factorization homology with such coefficients is a proper homotopy invariant. However, note that the factorization homology $\int_M \free_n^P$ is independent of the framing on $M$; this is a consequence of the fact that the $\disk_n^{\fr}$-algebra structure on $\free_n^P$ can be enhanced to a $\disk_n$-algebra.
\end{remark}

Recall the maps of symmetric monoidal quasi-categories $\disk^{\fr, \sqcup}_n \to \disk^{\fr,\sqcup}_{n,k}$ and $\disk^{\fr,\sqcup}_k \to \disk^{\fr,\sqcup}_{n,k}$ indicated by the assignments $\RR^n\mapsto \RR^n$ and $\RR^k \mapsto (\RR^k\subset \RR^n)$, respectively.
The following lemma describes the free $\disk_{n,k}^{\fr}$-algebras in terms of free $\disk^{\fr}_n$-algebras and free $\disk^{\fr}_k$-algebras. 

\begin{lemma}\label{free-free}
Let $(P,Q)$ be a pair of objects of $\cC$. Then the universal arrows to the restrictions
	\[
	 \free_n^P\xra{\simeq} \bigl(\free_{n,k}^{(P,Q)}\bigr)_{|\disk^{\fr,\sqcup}_n} 
	\qquad \&
	\qquad
	 \free_k^Q \ot\int_{S^{n-k-1}\times \RR^{k+1}}\free_n^P~{}~\xra{\simeq} ~{}~\bigl(\free_{n,k}^{(P,Q)}\bigr)_{|\disk^{\fr,\sqcup}_k}
	\]
are equivalences.  
\end{lemma}

\begin{proof} 
Recall from the proof of Proposition~\ref{nk} that a $\disk^{\fr}_{n,k}$-algebra structure $A$ on $(A_n, A_k)$, where $A_n$ is a $\disk^{\fr}_n$-algebra and $A_k$ is a $\disk^{\fr}_k$-algebra, is equivalent to the structure of a $\disk_{n-k}^{\fr}$-$A_n$-module structure on $A_k$ in the quasi-category $\Alg_{\disk^{\fr}_k}(\cC^\ot)$. 
The forgetful functor factors as the forgetful functors
\[
\Alg_{\disk^{\fr}_{n,k}}(\cC^\ot)\longrightarrow\Alg_{\disk^{\fr}_n}(\cC^\ot)\times\Alg_{\disk^{\fr}_k}(\cC^\ot)\longrightarrow\cC\times\cC
\] 
and thus, passing to the left adjoints, we can write the free algebra $A$ on a pair $(P,Q)$ as the composite of the two left adjoints, which gives the free $\disk^{\fr}_n$-algebra on $P$ and the free $\disk^{\fr}_{n-k}$-module on the free $\disk^{\fr}_k$-algebra on $Q$; the latter is calculated by tensoring with the factorization homology $\int_{S^{n-k-1} \times \RR^{k+1}}\free_n^P$, which is a special case of the equivalence between ${\disk^{\fr}_j}$-$R$-modules and left modules for $\int_{S^{j-1}}R$, see Proposition 3.16 of~\cite{cotangent}, applied to $R =\free_n^P$ and $j=n-k$.

\end{proof}

\begin{proof}[Proof of Proposition ~\ref{prop:conf}]

Recall the construction of the $\infty$-operad $\cE^{\amalg}_{\sf inert}$ of Example~\ref{free-guy} -- it is the free $\infty$-operad on $\cE$.  That is, the map 
\[
\mathsf{Fun}^\ot(\cE^{\amalg}_{\sf inert},\cC^\ot) \xra{\simeq} \mathsf{Fun}(\cE,\cC)\simeq \cC^{\sO(n)}\times \cC^{\sO(n,k)}~, 
\]
induced by restriction along the inclusion of the underlying quasi-category $\cE \to \cE^\amalg_{\sf inert}$, is an equivalence of Kan complexes -- here we are using exponential notation for simplicial sets of maps.
Explicitly, a vertex of $\cE^\amalg_{\sf inert}$ is a pair of finite sets $(J_n,J_k)$ while an edge is the data of a pair of based maps $(J_n)_+ \xra{a} (J_n')_+$ and $(J_k)_+ \xra{b} (J_k')_+$ which are \emph{inert}, which is to say, the fibers over non-base points $a^{-1}(j')$ and $b^{-1}(j'')$ are each singletons, together with a pair of elements $\alpha\in \sO(n)^{J_n'}$ and $\beta\in \sO(n,k)^{J_k'}$.  
We will denote a typical object of $\cE^\amalg_{\sf inert}$ as $E=(J_n,J_k)$.

The the standard inclusion $\cE \to \mfld_{n,k}$ then induces the map of $\infty$-operads $\cE^\amalg_{\sf inert} \xra{i} \mfld^{\sqcup}_{n,k}$ whose value on vertices is
\[
i\colon (J_n,J_k)\mapsto \bigl(\bigsqcup_{J_n} \RR^n\bigr)\sqcup \bigl(\bigsqcup_{J_k} (\RR^k\subset \RR^k)\bigr)~.
\] 
Likewise, let $(\sO(n)\xra{\w{P}}\cC~,~\sO(n,k)\xra{\w{Q}} \cC)$ be a pair $(P,Q)$ of objects in $\cC$ each equipped with actions of $\sO(n)$ and $\sO(n,k)$, respectively.  
These data then determine the solid diagram of $\infty$-operads
\[
\xymatrix{
\cE^{\amalg}_{\sf inert}  \ar[rrrrrd]^-{(P^\bullet,Q^\bullet)}   \ar[d]^i
&&&&&
\\
\mfld^{\sqcup}_{n,k}  \ar@{.>}[rrrrr]^-{\mathsf{Free}_\cE^{(\w{P},\w{Q})}}
&&&&&
\cC^\ot
}
\]
in where the value of $(P^\bullet,Q^\bullet)$ on $(J_n,J_k)$ is canonically equivalent to $ P^{\ot J_n} \ot Q^{\ot J_k}$ as a $\bigl(\sO(n)^{J_n} \times \sO(n,k)^{J_k}\bigr)$-objects.  
The filler $\mathsf{Free}_\epsilon^{(\w{P},\w{Q})}$ is the desired free construction, and is computed as operadic left Kan extension.  Explicitly, for $X\in \mfld_{n,k}$, the value
\begin{equation}\label{free-colimit}
\mathsf{Free}_\cE^{(\w{P},\w{Q})}(X) = \colim_{E\xra{\mathsf{act}} X}  (P^\bullet,Q^\bullet)(E) = \colim_{(J_n,J_k)\xra{\sf act} X} P^{\ot J_n} \ot Q^{\ot J_k}
\end{equation}
where the colimit is over the quasi-category $\cE^\amalg_X  := \cE^\amalg_{\sf inert} \times_{\mfld^{\sqcup}_{n,k}} \bigl((\mfld^{\sqcup}_{n,k})_{\sf act}\bigr)_{/X}$ of active morphisms in $\mfld^{\sqcup}_{n,k}$ from the image under $i$ of $\cE^\amalg_{\sf inert}$ to the object $X$.

We will now compute the colimit in~(\ref{free-colimit}).  
By construction, the projection $\cE^\amalg_X \to \cE^\amalg_{\sf inert}$ is a right fibration whose fiber over $E$ is the Kan complex $\mfld_{n,k}(i(E),X)$. 
Consider the subcategory of isomorphisms $\cE^\amalg_{\sf iso}\subset \cE^\amalg_{\sf inert}$ -- it is isomorphic to the category of pairs of finite sets and pairs of bijections among them. 
Denote $\cG = \cE^\amalg_{\sf iso}\times_{\cE^\amalg_{\sf inert}} \cE^\amalg_X$.  
Because the inclusion $\cE^\amalg_{\sf iso}\subset \cE^\amalg_{\sf inert}$ is cofinal, so is the inclusion $\cG \subset \cE^\amalg_X$.  So the colimit~(\ref{free-colimit}) is canonically equivalent to the colimit of the composite $\cG \subset \cE^\amalg_X \xra{(P^\bullet,Q^\bullet)} \cC^\ot$.  
Because $\cE^\amalg_{\sf iso}$ is a coproduct of $\infty$-groupoids (Kan complexes) indexed by isomorphism classes of its objects, then $\cG$ is a coproduct of $\infty$-groupoids indexed by isomorphism classes of objects of $\cE^\amalg_{\sf iso}$.  As so, the colimit~(\ref{free-colimit}) breaks up as a coproduct over isomorphism classes of objects of $\cE^\amalg_{\sf iso}$.

We will now understand the $[E]^{th}$ summand of this colimit.  
Choose a representative $E=(J_n,J_k)\in \cE^\amalg_{\sf iso}$ of this isomorphism class. 
We point out that the Kan complex of $\Aut(E)$ fits into a Kan fibration sequence $\sO(n)^{ J_n}\times \cO(n,k)^{ J_k} \to \Aut(E) \to \Sigma_{J_n}\times \Sigma_{J_k}$.  
Consider the right fibration $(\cE^\amalg_{\sf iso})_{/E} \to \cE^\amalg_{\sf iso}$ whose fiber over $E'$ is the Kan complex $\Iso(E',E)$ which is a torsor for the Kan complex $\Aut(E)$ is $E'$ if isomorphic to $E$ and is empty otherwise.  
Denote the resulting right fibration $\cG_E = (\cE^\amalg_{\sf iso})_{/E} \times_{\cE^\amalg_{\sf iso}} \cG~\longrightarrow~ \cG$ whose fibers are either a torsor for $\Aut(E)$ or empty.  
The composite $\cG_E \to \cG \subset \cE^\amalg_X \xra{(P^\bullet,Q^\bullet)} \cC^\ot$ is canonically equivalent to the constant map at $P^{\ot J_n}\ot Q^{\ot J_k}$.  
It follows from the definition of the tensor over spaces structure, that the colimit of this composite  is 
\[
\Bigl(\mfld_{n,k}\bigl(i(J_n, J_k),X \bigr)\Bigr) \ot \bigl(P^{\ot J_n} \ot Q^{\ot J_k}\bigr)~.
\]  
We conclude from this discussion that the colimit of the composite $\cG \subset \cE^\amalg_X \xra{(P^\bullet,Q^\bullet)} \cC^\ot$ is 
\begin{equation}\label{tensor-time}
\coprod_{[(J_n,J_k)]} \Bigl(\mfld_{n,k}\bigl(i(J_n,J_k),X \bigr)\Bigr) \ot_{\Aut(J_n,J_k)} \bigl(P^{\ot J_n} \ot Q^{\ot J_k}\bigr)~.
\end{equation}

We make expression~(\ref{tensor-time}) more explicit for the case that $(\w{P},\w{Q}) = \delta(P,Q)$ is a pair of objects with trivial group actions.  
As so, the map $\cG \subset \cE^\amalg_X \xra{(P^\bullet,Q^\bullet)} \cC^\ot$ factors through the projection $\cG \to \cE^\amalg_{\sf iso} \to (\mathsf{Fin}_\ast)_{iso}$ the groupoid of finite sets and bijections -- we denote this groupoid as $\Sigma$.

Recall that the $\sD_{n,k}$-manifold $X=(L\subset M)$ is the data of a framed $n$-manifold $M$, a properly embedded smooth submanifold $L$, and a splitting of the framing along $L$.  Evaluation at the origins of $i(J_n,J_k) = \bigl(\bigsqcup_{J_n} \RR^n\bigr)\sqcup \bigl(\bigsqcup_{J_k} (\RR^k\subset \RR^n)\bigr)$ gives a map
\[
\mfld_{n,k}\bigl((\bigsqcup_{J_n} \RR^n)\sqcup (\bigsqcup_{J_k} \bigl(\RR^k\subset \RR^n)\bigr)~,~(L\subset M)\bigr) ~\longrightarrow~ \mathsf{Conf}_{J_n}(M\smallsetminus L)\times \mathsf{Conf}_{J_k}(L)~.  
\]
This map is evidently natural among morphisms among the variable $(J_n,J_k)\in \cG$ where the action of $\cG$ on the righthand side factors through the projection $\cG \to \Sigma$.   
There results a $\Sigma$-equivariant map
\[
\Bigl(\mfld_{n,k}\bigl(i(J_n,J_k),(L\subset M)\bigr)\Bigr)_{/\sO(n)^{ J_n}\times \sO(n,k)^{ J_k}} ~\xra{\simeq}~ \mathsf{Conf}_{J_n}(M\smallsetminus L)\times \mathsf{Conf}_{J_k}(L)~;
\]
and for standard reasons it is an equivalence of Kan complexes.  
We conclude that
\[
\free_\cE^{\delta(P,Q)}(L\subset M) \xra{\simeq} \conf^P(M\smallsetminus L)\ot \conf^Q(L)~.
\]
Finally, the formula 
\[
\free_\cE^{\delta(P,Q)}(L\subset M) = \int_{(L\subset M)} \free_\cE^{\delta(P,Q)}
\]
is a formal consequence of commuting left Kan extensions (here we are using the same notation for $\free_\cE^{\delta(P,Q)}$ and its restriction to $\disk^{\sqcup}_{n,k}$).  With Lemma~\ref{frame-or-not}, this completes the proof of the proposition.

\end{proof}

\begin{remark}
The methods employed here in~\S\ref{free-examples} have been intentionally presented to accommodate much greater generality.  
For instance, with appropriate modifications of the statements, the role of $\sD_{n,k}$ (or its framed version) could be replaced by any category of basics $\cB$.  Likewise, the maximal sub-Kan complex $\cE\subset \sD_{n,k}$ could be replaced by any map $\cE\to \cB$ of quasi-categories.  

\end{remark}

\section{Differential topology of singular manifolds}\label{snglr-facts}

The proofs of Proposition~\ref{tubular-neighborhood}, Lemma~\ref{creation}, and Theorem~\ref{pushforward-formula} require some foundational theory of singular manifolds akin to the classical theory of differential topology.  
We give a brief tour of such theory and analyze the singular generalizations of orthogonal transformations, tangent bundles, vector fields and flows, and embedding theorems.  
It is here that we fully exploit the deliberate cone-structure of singular manifolds.   
 
Through the course of this section we will not distinguish between a singular manifold $X$ and its underlying space $\iota X$.  
In the case $X = U^n_Y$ is a basic of depth $k$, as usual we identify the subspace $\RR^{n-k}\times \ast \subset \RR^{n-k}\times C(Y)$ with $\RR^{n-k}$.  

\begin{notation}
For $U,V \in \bsc_n$, we denote by ${\sf Iso}(U,V) \subset \bsc_n(U,V)$ the space of invertible morphisms. For $X \in \snglr_{n}$ we let $\Aut(X) \subset \snglr_n(X,X)$ denote the automorphisms of $X$.
\end{notation}

\subsection{Endomorphisms of basics}\label{orthogonal}
In this section we prove basic facts about the space of endomorphisms of basics -- these facts are useful for understanding specific examples of $\cB$-manifolds and also serve a larger purpose in the present work. For a basic open $U = U^n_X$ of depth $k$, we prove that the endomorphism monoid of $U$ is homotopy equivalent to $\sO(\RR^{n-k}) \times \Aut_{\snglr_{k-1}}(X)$, and in particular that the endormorphisms of a fixed object $U_X \in \ob \bsc_n(X)$ form an $\infty$-groupoid.

\begin{notation}
Fix a basic $U= U^n_X \in \bsc_n$ with $X$ of dimension $k-1$. In what follows, we write $\RR^{n-k} \subset U_X$ to mean $\RR^{n-k} \times \{\ast\} \subset \RR^{n-k} \times C(X) = \iota U_X$. In particular, we will abbreviate the element $(0,\ast) \in \iota U_X$ by simply writing 0.
\end{notation}
We are about to analyze the topological monoid $\bsc_n(U,U)$ of endomorphisms of $U$.  
From the definition of $\bsc_n(U,U)$, restriction to the subspace $\RR^{n-k} \subset U$ gives a map of topological monoids
\begin{equation}\label{rest-to-locus}
(-)_{|\RR^{n-k}}\colon \bsc_n(U,U) \to \Emb(\RR^{n-k},\RR^{n-k})~.
\end{equation}
If $k=0$ then $U=\RR^n$, note that the map~(\ref{rest-to-locus}) is an isomorphism.

\begin{remark} 
There is a sequence of topological submonoids
\[
\Emb(\RR^n,\RR^n)\supset \Emb^0(\RR^n,\RR^n)\supset \Aut^0(\RR^n) \supset {\sf GL}(\RR^n) \supset \sO(\RR^n)  
\]
defined as follows:
\begin{itemize}
\item
$\Emb^0(\RR^n,\RR^n)$ consists of smooth self-embeddings which preserve the origin.  
\item
$\Aut^0(\RR^n)$  consists of origin preserving, smooth self-embeddings which are isomorphisms.
\item
${\sf GL}(\RR^n)$ consists of those automorphisms which are linear.
\item
$\sO(\RR^n)$ consists of those linear automorphisms which are orthogonal -- it is a finite dimensional compact group.
\end{itemize}
Each of these inclusions is a homotopy equivalence of topological spaces.  This can be seen through the following deformation retractions:
\begin{itemize}
\item
There is the deformation retraction of $\Emb(\RR^n,\RR^n)$ onto $\Emb^0(\RR^n,\RR^n)$ given by $(t,f)\mapsto \bigl((f(-)-tf(0)\bigr)$.
\item
There is a deformation retraction of $\Emb^0(\RR^n,\RR^n)$ onto ${\sf GL}(\RR^n)$ given by $(t,f)\mapsto \bigl(\frac{1}{t}f(t-)\bigr)$. For $t=0$ this is understood to be $D_0f$, the derivative of $f$ at the origin.
That the image lies in ${\sf GL}(\RR^n)$ is obvious. The continuity of this assignment is a consequence of the definition of the weak Whitney $C^\infty$ topology.  
\item
There is a deformation retraction of ${\sf GL}(\RR^n)$ onto $\sO(\RR^n)$ given by the Graham-Schmidt algorithm $\mathsf{GrSm}_t\colon {\sf GL}(\RR^n) \to \sO(\RR^n)$.  
\end{itemize}
\end{remark}

The methods above are classical.  We now imitate these methods for the case $k>0$.  
So assume $X$ is a non-empty singular manifold of dimension $k-1>-1$. 
Consider the continuous map
\begin{equation}\label{gamma}
\gamma \colon \RR_{>0} \times \RR^{n-k} \to \bsc_n(U,U)~,~{}~{}~{}~{}~(t,v)\mapsto \bigl([u,s,x]\mapsto [tu-v,ts,x]\bigr)~.
\end{equation}
Here $[u,s,x]$ is an element of $\iota U(X) = \RR^{n-k} \times \RR_{\geq 0} \times \iota X_{/\sim}$. 
Write the value of $\gamma$ on $(t,v)$ as the endomorphism $U\xra{\gamma_{t,v}} U$.  
Notice $\gamma_{t,v}\circ \gamma_{t',v'} = \gamma_{tt',v+tv'}$ and $\gamma_{t,v}^{-1} = \gamma_{\frac{1}{t},-\frac{1}{t}v}$. An important case of these identities is $v=0$.

\begin{remark}
The map $\gamma$ embodies the concept of scaling and translating -- we see it as capturing what remains of a vector space structure on a basic.  
\end{remark}

\begin{definition}
We now define the topological submonoids 
\begin{equation}\label{basic-equivalences}
\bsc_n(U,U) \supset \bsc_n^0(U,U) \supset \Aut^0(U)\supset {\sf GL}(U) \supset \sO(U)~.
\end{equation}
\begin{itemize}
	\item
	$\bsc_n^0(U,U)$ consists of those endomorphisms $f$ of $U$ which preserve the origin $0\in \RR^{n-k}\subset U$, which is to say $f_{|\RR^{n-k}}(0)=0$.  
	\item
	$\Aut^0(U)$ consists of those origin preserving self-maps which admit a strict inverse (in the category $\bsc_n$). I.e., these are the origin-preserving automorphisms.
	\item
	${\sf GL}(U)$ consists those endomorphisms $T$ of $U$ which satisfy the identity
	\[
	T\circ \gamma_{t,v} = \gamma_{t,Tv}\circ  T
	\]
	for all $(t,v)\in \RR_{>0}\times \RR^{n-k}$. (The map $Tv$ is defined by noting that any $T \in \bsc_n{(U_X,U_X)}$ defines an embedding $\RR^{n-k} \to \RR^{n-k}$.)
	That such $T$ are examples of origin preserving automorphisms is Lemma~\ref{linear-is-iso}.
	\item
	The last is defined as the product of topological groups $\sO(U) := \sO(\RR^{n-k}) \times \Aut(X)$. This is regarded as a submonoid of $\bsc_n^0(U,U)$ through the homomorphism given by $(T,f)\mapsto \bigl( [u,s,x]\mapsto [T(u),s,f(x)] \bigr)$. That its image lies in ${\sf GL}(U)$ is an easy exercise.
	\end{itemize}
\end{definition}

\begin{lemma}\label{linear-is-iso}
There is an inclusion of subspaces ${\sf GL}(U) \subset \Aut^0(U)$.  

\end{lemma}

\begin{proof}
We need to show that each morphism $U\xra{T} U$ in ${\sf GL}(U)$ is a surjective map of spaces.
Let $[u,s,x]$ be a point in $\iota U=\RR^{n-k}\times C(X)$.  

We know the image of $T$ is an open subset of $\iota U$ and contains the origin $0\in \RR^{n-k}\subset \RR^{n-k}\times C(X)$.  
The sequence $t\mapsto \gamma_{t,0}([u,s,x])$ converges to $0$ as $t\to 0$.  
So there is a $t_0>0$ for which $\gamma_{t_0,0}([u,s,x])$ is in the image of $T$.
That is, there is a point $[u',s',x']\in \iota U$ with $T[u',s',x'] = \gamma_{t_0,0}[u,s,x]$.  
Then $T\gamma_{\frac{1}{t_0},0}[u',s',x'] = [u,s,x]$.  

\end{proof}

\begin{lemma}\label{derivative}
Each of the inclusions in~(\ref{basic-equivalences}) is a homotopy equivalence of spaces.  
\end{lemma}

\begin{proof}
It is enough to witness the following deformation retractions:
\begin{itemize}
\item
There is a deformation retraction of $\bsc_n(U,U)$ onto $\bsc_n^0(U,U)$ given by $(t,f)\mapsto \gamma_{1,tf(0)} \circ f$.

\item
There is a deformation retraction of $\bsc_n^0(U,U)$ onto ${\sf GL}(U)$ given by $(t,f)\mapsto \bigl(\gamma_{\frac{1}{t},0} \circ f\circ \gamma_{t,0}\bigr)$ -- we continuously extend this expression to $t=0$ as 
\begin{equation}\label{proof:derivative}
(0,f)~{}~\mapsto ~{}~D_0f := \Bigl( [u,s,x]~\mapsto~ \Bigl[D_0f_{|\RR^{n-k}}(u)  ,  D_0f_\RR[s,x]  ,  f_X[0,0,x]\Bigr]  \Bigr)
\end{equation}
which we now explain.

The first coordinate is the derivative of $\RR^{n-k}\xra{f_{|\RR^{n-k}}}\RR^{n-k}$ at the origin.
Choose a lift $\bigl(\RR^{n-k}\times \RR \times X ~ \xra{\w{f}} ~ \RR^{n-k}\times \RR\times X\bigr)\in \w{\bsc}_n$ of $f$.  
Denote the projections of $\w{f}$ onto the $\RR$- and $X$-factors as $f_\RR$ and $f_X$, respectively.  
The third coordinate is the value of $f_X$ at $[0,0,x]$.
By adjointness, regard $f_\RR$ as a map $\RR^{n-k}\times X \to \Map^0(\RR,\RR)$ to the space of smooth origin-preserving maps of the real line.  
Notate the composition $D_0f_\RR \colon X \xra{\{0\}\times 1_X} \RR^{n-k}\times X \xra{f_\RR} \Map^0(\RR,\RR) \xra{D_0} \Map^0(\RR,\RR)$ -- by adjointness we can write it as a map $D_0f_\RR\colon \RR\times X \to \RR$.  
This defines the second coordinate.

That expression~(\ref{proof:derivative}) is independent of $\w{f}$ is immediate from the definition of morphisms in $\bsc_n$.  
Given existence, the expression for $D_0f$ is forced upon us by continuity:
\[
D_0f = \lim_{s\to 0} \gamma_{\frac{1}{s},0} f \gamma_{s,0}~.
\]
That this expression describes an element of ${\sf GL}(U)$ is the string of equalities
\begin{eqnarray}
D_0f \circ \gamma_{t,v}
& = &
\bigl(\lim_{s\to 0} \gamma_{\frac{1}{s},0} f \gamma_{s,0}\bigr)\circ \gamma_{t,v} 
\nonumber
\\
& = & 
\lim_{s\to 0} \gamma_{\frac{1}{s},0} f \gamma_{ts,sv}  
\nonumber
\\
& = & 
\lim_{s\to 0} \Bigl(\bigl( \gamma_{1,-\frac{1}{s}f(-sv)} \gamma_{1,\frac{1}{s}f(-sv} \bigr) \circ \bigl( \gamma_{\frac{1}{s},0} f \gamma_{st,sv}\bigr)\Bigr)
\nonumber
\\
& = & 
\lim_{s\to 0} \Bigl( \bigl(\gamma_{1,-\frac{1}{s} f(-sv)}\bigr)  \circ \bigl( \gamma_{\frac{1}{s},0} \gamma_{1,\frac{1}{s}f(-sv)} f \gamma_{st,sv} \bigr)\Bigr)
\nonumber
\\
& = & 
\Bigl(\lim_{r\to 0} \gamma_{1,\frac{1}{r}f(rv)} \Bigr) \circ \Bigl( \lim_{s\to 0} \bigl( \gamma_{\frac{1}{s},0} \circ ( \lim_{p\to 0} \gamma_{1,f(-pv)} f \gamma_{st,pv}) \bigl) \Bigr)
\nonumber
\\
& = & 
\gamma_{1,D_0f(v)} \circ \bigl( \lim_{s\to 0} \gamma_{\frac{1}{s},0} f \gamma_{st}\bigr)
\nonumber
\\
& = & 
\gamma_{1,D_0f(v)} \circ \bigl(\lim_{s'\to 0} \gamma_{\frac{t}{s},0} f \gamma_{s,0} \bigr)
\nonumber
\\
& = & 
\gamma_{1,D_0f(v)}\gamma_{t,0} \circ \bigl( \lim_{s'\to 0} \gamma_{\frac{1}{s},0} f \gamma_{s,0}\bigr)
\nonumber
\\
& = & 
\gamma_{t,D_0f(v)}\circ  D_0f.  
\nonumber
\end{eqnarray}
Moreover, $T\in {\sf GL}(U)$ implies $\gamma_{\frac{1}{t},0} T \gamma_{t,0} = T$ for all time $t>0$, and so expression~(\ref{proof:derivative}) is indeed a deformation retract.

\item
There is a deformation retraction of ${\sf GL}(U)$ onto $\sO(U)$ given by
\[
(t,T)\mapsto \Bigl( [u,s,t]\mapsto \Bigl[ \mathsf{GrSm}_t(D_0f_{|\RR^{n-k}})(u) ,(1-t) + t D_0f_{\RR}[s,x], f_X[0,0,x]\Bigr] \Bigr)~.
\]

\end{itemize}

\end{proof}

In the proof of Lemma~\ref{derivative} we discovered the 
\begin{definition}[Derivative]\label{derivative-map}
There is the continuous map
\[
D\colon \RR^{n-k}\times \bsc_n(U,U) \to {\sf GL}(U)~,~{}~{}~(v,f)~\mapsto ~\lim_{t \to 0} \gamma_{\frac{1}{t},\frac{1}{t}f(v)} f \gamma_{t,-v}~.
\]
We write the value of $D$ on $(v,f)$ as $D_vf$ and refer to it as the \emph{derivative of $f$ at $v$}.  

\end{definition}

\begin{cor}\label{deriv-homo}
The derivative map $D_0\colon \bsc_n^0(U,U) \to {\sf GL}(U)$ is a homomorphism.
\end{cor}

\begin{proof}
This follows easily from the classical chain rule for maps between Euclidean spaces.  
\end{proof}

\begin{lemma}\label{only-equivalences}
Let $U,V\in \bsc_n$ be basics of depths $k$ and $l$ respectively.  Suppose the space of morphisms $\bsc_n(U,V)\neq \emptyset$ is non-empty.  The following are equivalent.
\begin{enumerate}
\item The depths $k=l$ are equal.
\item There is a morphism $f\colon U \to V$ for which the intersection $\emptyset \neq f(\iota U)\cap \RR^{n-l}\subset V$ is non-empty.
\item The inclusion of the isomorphisms
\[
{\sf Iso}(U,V) \xra{\simeq} \bsc_n(U,V)
\]
is a homotopy equivalence.
\item There is an isomorphism $U\cong V$.  
\end{enumerate}
\end{lemma}
\begin{proof}
We show $(3)\implies (4) \implies (1) \implies (3)$ and $(1)\iff(2)$.
The first implication is obvious.

Write $U=U_X$ and $V=V_Y$ for $X$ and $Y$ compact singular manifolds of dimensions $(k-1)$ and $(l-1)$ respectively.  Inspecting the definition of a morphism $U\xra{g} V$ in $\bsc_n$ we conclude $U\cong V \implies X\cong Y$.  Because $X\cong Y$ implies $\iota X \cong \iota Y$ are homeomorphic, by invariance of domain the dimensions must be equal $k=l$.  So $(4) \implies (1)$.
Moreover, knowing $\bsc_n(U,V)$ is nonempty, we conclude the following.
\begin{itemize}
\item If $k=l$ then $X \cong Y$, hence $U\cong V$.  We can thus assume $U=V$.  From Lemma~\ref{derivative} there is a deformation retraction of both $\bsc_n(U,U)$ and ${\sf Iso}(U,U)$ onto ${\sf GL}(U)$.   Thus $(1)\implies(3)$.  
\item If $k\neq l$ then the image $ g(\iota U)\subset \RR_{>0}\times C(\iota V)$.  In particular this image is disjoint from $\RR^{n-l}$.  Thus $\neg (1) \implies \neg (2)$. 
\item If the image of $g$ is disjoint from $\RR^{n-l}\subset \iota V$, then $k\neq l$.  Thus $\neg(2)\implies\neg(1)$.  
\end{itemize}
\end{proof}

\begin{remark}
We conclude from Lemma~\ref{only-equivalences} that for each point $x$ in a singular manifold $X$ there is a unique isomorphism class of coordinate charts $(U,0) \to (X,x)$ about $x$.
\end{remark}

There are a number of immediate corollaries of Lemma~\ref{only-equivalences}.
\begin{cor}
\begin{enumerate}
\item[~]
\item
Let $U$ be an objet of $\bsc_n$.  Then the coherent nerve of the subcategory of endomorphisms
$N^c \bigl(\End_{\bsc_n}(U)\bigr)$
is a Kan complex (i.e., $\infty$-groupoid).

\item
Let $U\xra{f} V$ be a morphism in $\bsc_n$.
Then exactly one of the following is true:
\begin{enumerate}
\item The edge $f$ is an equivalence in the coherent nerve $N^c\bsc_n$~.  
\item The depth of $U$ is strictly less than the depth of $V$.
\end{enumerate}

\item

Denote the full subcategory $\bsc_{n,=j}\subset \bsc_n$ spanned by the basics of depth exactly $j$.   
Then $\bsc_{n,=(n-k)}$ has a skeleton which is the coproduct (as categories) of topological monoids 
\[
\coprod_{[Y^{k-1}]} \bsc_n(U^n_Y,U^n_Y)~, 
\]
with one summand for each isomorphism class of non-empty compact singular $(k-1)$-manifolds.  
Moreover, each such monoid group-like.

\item

The quasi-category $N^c \bsc_{n,=(n-k)}$ is a Kan complex.

\item
The assignment $U\mapsto \mathsf{depth}(U)$ describes a functor
\[
\bsc_n \to \NN
\]
to the poset natural numbers.  
Upon applying the coherent nerve, this functor is conservative.  
\\
(It follows that $N^c\bigl( \bsc_n \bigr) \to \NN$ is a fibration of quasi-categories, though this will not be used.)

\end{enumerate}

\end{cor}

Recall the topological monoid $\w{\bsc}_n(U,U)$ which appearred in Definition~\ref{singular-manifolds}.  
\begin{lemma}\label{extensions}
The map of topological monoids
\[
\w{\bsc}_n(U,U) \to \bsc_n(U,U)
\]
is a homotopy equivalence of spaces.  
\end{lemma}
\begin{proof}
First notice the apparent lift  $\sO(U) \to \w{\bsc}_n(U,U)$ of the inclusion $\sO(U) \subset \bsc_n(U,U)$.  
The defining expression~(\ref{gamma}) of $\gamma$ reveals that it factors through $\w{\bsc}_n(U,U) \to \bsc_n(U,U)$.
The methods just established for $\bsc_n(U,U)$ carry through verbatum and give that the inclusion of topological monoids
\[
\sO(U) \xra{\simeq} \w{\bsc}_n(U,U)
\]
is a homotopy equivalence of spaces.  

\end{proof}

\subsection{Dismemberment}\label{section:dismemberment}
We give a useful construction which functorially assigns to each singular $n$-manifold $X$ of depth $k$ a singular $n$-manifold $\w{X}$ of depth less than $k$, equipped with a conically smooth map $\w{X} \to X$ which is an isomorphism of singular $n$-manifolds under $X\smallsetminus X_{n-k}$.  We refer to $\w{X}$ as the \emph{dismemberment} of $X$ (along $X_{n-k}$) and observe useful relations between properties of $X$ and its dismemberment.  
The intuition is that $\w{X}$ is obtained from $X$ by tearing out its deepest stratum then extending by a small collar along the torn locus. The map $\w{X}\to X$ is given by collapsing this collar extension appropriately.  

\begin{lemma}\label{dism}
There are topological functors
\[
\w{(-)}_{\leq 0} \colon \snglr_{n,\leq k} \to \snglr_{n,\leq k}~{}~{}~{}~{}~{}~{}~{}~\w{(-)}\colon \snglr_{n,\leq k} \to \snglr_{n,<k}
\]
equipped with 
\begin{itemize}
\item a natural transformation
\[
i \colon \iota \w{(-)}_{\leq 0} \leftrightarrows \iota \w{(-)}\colon \sigma
\]
with $\sigma i = 1$ the identity transformation,

\item 
a natural transformation
\[
\iota \w{(-)}_{\leq 0} \to \iota
\]
which, for $\partial_k(-):=\bigl(\w{(-)}_{\leq 0}\bigr)_{|(-)_{n-k}}$, restricts as a fiber bundle
\[
\iota \partial_k(-) \to \iota (-)_{n-k}
\]
whose fibers are in $\snglr_{n,k-1}^{\sf cmpt}$ (in particular the fibers are not empty).  

\end{itemize}

\end{lemma}

\begin{remark}

The intuition is that $\w{X}_{\leq 0}$, and $\w{X}$, are obtained from $X$ by deleting an open, respectively closed, tubular neighborhood of the depth $k$-stratum of $X$.  
If $X$ happens to have empty depth $k$-stratum, then $\w{X}_{\leq 0} = \w{X} = X$.  

\end{remark}

\begin{remark}\label{boundary-smooth}

The natural transformations of Lemma~\ref{dism} are defined on underlying spaces.  
This is because we have not yet discussed maps among singular manifolds of unequal dimensions.  
We will do this in~\S\ref{conically-smooth}, afterwhich it is manifestly the case that these natural transformations canonically lift to natural transformations by \emph{conically smooth} maps.  

\end{remark}

\begin{remark}
The intuition is that $\w{X}_{\leq 0}$ is obtained from $X$ by deleting tubular neighborhoods of the singular strata, and that $\w{X}$ is obtained by deleting closed tubular neighborhoods of these strata.  
The exact construction is an imporvement on this intuition so that no information is lost.
Slightly more specifically, where no `deletion' actually takes place, but rather something more like `real blow-up' and `real blow-up with collar extension', respectively.

\end{remark}

\begin{proof}[Proof of Lemma~\ref{dism}]
We construct the functors $\w{(-)}_{\geq 0}$ and $\w{(-)}$, as well as the said natural transformations, by induction on the depth parameter.
For the base case of depth zero, declare $\w{(-)}_{\geq 0} = \w{(-)}$ to be the identity functors -- here we are using the canonical embedding $\mfld_n \subset \snglr_n$.  
The natural transformations are tautological.

Suppose both $\w{(-)}_{\geq 0}$ and $\w{(-)}$ have been defined on $\snglr_{n,<k}$, as well as the natural transformations.
We first define $\w{(-)}_{\geq 0}$ and $\w{(-)}$, as well as the natural transformations, on $\bsc_{n,\leq k}$.
Let $U\in \bsc_{n,\leq k}$ be a basic of depth at most $k$.  
If the depth of $U$ is strictly less than $k$, then $\w{U}_{\geq 0}$ and $\w{U}$, as well as the natural transformations, have already been defined.  
Let $U=U_Y$ be a basic of depth $k$.  
Declare 
\[
\w{(U_Y)}_{\geq 0} := \RR^{n-k} \times \RR_{\geq 0} \times Y~{}~{}~{}~{}~{}~{}~{}~{}~{}~ \w{(U_Y)}:= \RR^{n-k}\times \RR \times Y~.
\]
By induction, $\w{(U_Y)}_{\geq 0}$ is a sngular $n$-manifold of depth $k$, and $\w{(U_Y)}$ is a singular $n$-manifold of depth greater than $k$.
From the inductively defined natural transformations there are the canonical maps
\begin{equation}\label{unzip-trans}
i\colon \iota \w{(U_Y)}_{\geq 0}  =  \RR^{n-k} \times \RR_{\geq 0} \times \iota Y \hookrightarrow \RR^{n-k}\times \RR \times \iota Y = \iota \w{(U_Y)}~,
\end{equation}
\begin{equation}\label{unzip-retract}
\sigma \colon \iota \w{(U_Y)} = \RR^{n-k}\times \RR \times \iota Y \xra{1\times |-| \times \sigma} \RR^{n-k}\times \RR_{\geq 0} \times \iota Y  = \iota \w{(U_Y)}_{\geq 0}~,
\end{equation}
and 
\begin{equation}\label{dism-trans}
\iota \w{(U_Y)}_{\geq 0} =  \RR^{n-k} \times \RR_{\geq 0} \times \iota Y \to \RR^{n-k}\times \RR_{\geq 0} \times \iota Y \to \RR^{n-k} \times C(\iota Y) = \iota U_Y~.
\end{equation}
This latter map restricts to the projection $\partial_k (U_Y) = \RR^{n-k}\times \{0\}\times Y \to \RR^{n-k}$ which is a fiber bundle with fibers in $\snglr_{k-1}^{\sf cmpt}$.  

We now define $\w{(-)}_{\geq 0}$ and $\w{(-)}$, as well as the natural transformations, on morphisms of $\bsc_{n,\leq k}$. 
Let $U \xra{f} V$ be such a morphism.  
If the depth of $U$ or $V$ is less than $k$, then these data have already been defined.  
Assume $U=U_Y$ and $V=V_Z$ have depth $k$.  
Write $f\colon \RR^{n-k}\times C(\iota Y) \to \RR^{n-k}\times C(\iota Z)$ in coordinates as $f[u,s,y] = [f^{s,y}(u),f^{u,y}(s),f^{u,s}(y)]$.  
Exactly from the definition of morphisms between basics, the expression $\w{f}_{\geq 0}(u,s,y) = \bigl(f^{s,y}(u),f^{u,y}(s),\w{f}^{u,s}(y)\bigr)$ is well-defined for $s\geq 0$ (in particular, independent of the choice of representatives for $f$ and for $[u,s,y]$), and the resulting map of sets
\[
\w{(-)}_{\geq 0} \colon \bsc_n(U_Y,V_Z) \to \snglr_n
\]
is continuous.  
Likewise, define the continuous map
\[
\w{(-)}\colon \bsc_n(U_Y,V_Z) \to \snglr_n
\]
in coordinates as
\[
\w{f}^{s,y}(u) = 
\begin{cases}
f^{s, y}(u)
&
s\geq 0
\\
f^{-s, y}(u)
&
s\leq 0
\end{cases}
\] 
\[
\w{f}^{u,y}(s) = 
\begin{cases}
f^{u, y}(s)
&
s\geq 0
\\
-f^{u,y}(-s)
&
s\leq 0
\end{cases}
\]
\[
\w{f}^{u,s}(y) = 
\begin{cases}
\w{f}^{u,s}(y)
&
s\geq 0
\\
\w{f}^{u,s}(y)
&
s\leq 0~.
\end{cases}
\]
This expression is well-defined and it is simple to check that $\w{f}$ is `smooth', which is to say it is an element of $\snglr_n\bigl(\w{(U_Y)},\w{V_Z}\bigr)$ -- the salient point to notice is that the two limits
\[
\lim_{s\to 0^+} \frac{f^{u,y}(s)}{s} ~=~ \lim_{s \to 0^-} \frac{f^{u,y}(s)}{s}
\]
exist and agree for all $u$ and $y$.
It is immediate that these continuous maps of morphisms spaces respect composition and are compatible with the canonical maps~(\ref{unzip-trans}),~(\ref{unzip-retract}), and~(\ref{dism-trans}).  

We now define $\w{(-)}_{\geq 0}$ and $\w{(-)}$, as well as the natural transformations, on $\snglr_{n,\leq k}$.  
For $X$ a singular manifold of depth at most $k$, define the underlying spaces
\[
\iota \w{X}_{\geq 0} := \colim_{ U\to X} \iota \w{U}_{\geq 0} ~{}~{}~{}~{}~{}~{}~{}~{}~ \iota \w{X} := \colim_{U\to X} \iota \w{U}~.
\]
It is routine to verify that these spaces are second countable and Hausdorff, both of these statements follow because morphisms from basics to $X$ form a basis for its topology.  
By construction, $\iota \w{X}_{\geq 0}$ has the open cover $\{\iota \w{U}_{\geq 0}\mid U\xra{\phi} X\}$, and each inclusion $\w{(\psi^{-1} \phi)}_{\geq 0}\colon \iota \w{U}_{\geq 0} \subset \iota \w{V}_{\geq 0}$ is a morphism of singular $n$-manifolds.  
This determines an atlas for $\w{X}_{\geq 0}$, which determines a maximal atlas.  
That $\w{(-)}_{\geq 0}$ is a topological functor is formal from its construction (as a left Kan extension).  
Likewise for $\w{(-)}$.  

It is formal to check that these definitions of $\w{(-)}_{\geq 0}$ and $\w{(-)}$ agree with the inductively defined definitions on $\snglr_{n,<k}$.  
The natural transformations are equally formal, being defined on $\bsc_{n,\leq k}$.  
The conditions on these natural transformations are local conditions and their statements thus reduce to the statements already shown on basics.

\end{proof}

We make some immediate observations about these functors.
\begin{observation}
\begin{itemize}
\item[~]
\item
If $X$ is a smooth $n$-manifold then $\w{X}_{\geq 0} = \w{X} = X$.  
In particular, $\w{(-)}$ is idempotent.  

\item
For $X$ an arbitrary singular $n$-manifold, there is a canonical morphism $X_n \to \w{X}_{\geq 0}$ and $X_n \to \w{X}$ from the top $n$-dimensional stratum.

\item
If $X$ is a singular $n$-manifold, then there is a canonical isomorphism $X \cong \w{X}_{\geq 0}$, this is given by explicitly constructing such on basics by induction.  
In particular, $\w{(-)}_{\geq 0}$ is idempotent.  

\item
There is a canonical isomorphism $\w{(\w{(-)}_{\geq 0}}) \cong \w{(-)}$, this is given by explicitly constructing such on basics by induction.  

\item 
For $X$ an ordinary smooth manifold
\[
\w{X\times Y} \xra{\cong} \w{X}\times \w{Y}~.
\]

\end{itemize}

\end{observation}

\subsection{Tangent bundle}
We define the tangent bundle $TX$ associated to a singular manifold $X$.  
To do so, we must enlarge our notion of a singular manifold in the following way.
Recall that a basic $U$ can be canonically written as $U^n_Y = \RR^{n-k}\times C(Y)$ with $Y$ compact.    
We now allow the possibility that $Y$ admits a finite atlas, yet need not be compact. 
A second countable Hausdorff topological space equipped with a maximal atlas by such basics will be referred to as a \emph{protracted} singular $n$-manifold.  
We will not develop this larger class of protracted singular manifolds, but many of the basic results and definitions for singular manifolds are valid in this protracted setting.

As we will see, the tangent bundle $TX \to X$ of a singular manifold will not be a fiber bundle of topological spaces.
Indeed, the fiber over $x\in X$ is homeomorphic to a small neighborhood of $x$, any of which is homeomorphic to a cone $C(S^{n-k-1}\star Y)$.  
(See Remark~\ref{remark:TXisNotABundle}.) 
However, for each open stratum $X_{j-1,j}$ the restriction $TX_{|{X_{j-1,j}}} \to X_{j-1,j}$ does have the structure of a fiber bundle equipped with a fiberwise $\RR_{\geq 0}$-action in line with the cone structure of these cone-shaped neighborhoods.

\begin{definition}
Let $X$ be a topological space and let $\RR_{\geq 0} = (\RR_{\geq 0},\times)$ be the topological semi-monoid.  A \emph{$\RR_{\geq0}$-space over $X$} is a topological space $E\to X$ over $X$ together with an action of the semi-monoid over $X$
\[
(\RR_{\geq 0} \times X) \times_X E \to E 
\]
whose restriction to $\{1\}\times E$ is the identity on $E$.
We say a $\RR_{\geq 0}$-space over $X$ is \emph{free} if the restriction of the action to $\RR_{>0}$ is free and the map from the quotient $E_{/\RR_{\geq 0}} \to X$ is a homeomorphism.  
We call the resulting section $z\colon X \cong 0\cdot E \to E$ the \emph{zero-section}.  

\end{definition}

\begin{definition}
We simultaneously define: 
\begin{itemize}
\item The \emph{tangent bundle} functor
\[
X\mapsto (TX \to X)
\]
which assigns to each protracted singular manifold $X$ a free $\RR_{\geq 0}$-space $TX$ over $X$. This assignment is natural with respect to conically smooth maps.
\item An isomorphism of free $\RR_{\geq 0}$-spaces $T(X\times Y) \cong TX \times TY$ whenever $Y$ has depth zero. This isomorphism is natural with respect to conically smooth maps in $X$ and $Y$.  
\end{itemize}
We proceed by induction on the depth $k$ of $X$.  For the base case $k=0$ -- when $X$ is an ordinary smooth manifold -- we define $TX \to X$ as the familiar tangent bundle equipped with the free $\RR_{\geq 0}$-action by scaling.  
Record the standard canonical isomorphism $T(X\times Y) \cong TX \times TY$.  

Consider the singular $1$-manifold $\RR_{\geq 0} = U^1_{\ast}$.  
Declare its tangent space $T(\RR_{\geq 0}) = (T\RR)_{|\RR_{\geq 0}} \smallsetminus \{(0,t)\mid t>0\}$.

Suppose the functor $X'\mapsto TX'$ has been defined for $X'$ of depth less than $k$.  
Let $U= U^n_Y$ be a basic of depth $k$.   
Denote the singular $n$-manifold $\w{U} = R^{n-k}R Y$ -- it has depth less than $k$ and there is the standard quotient map of topological spaces $\w{U} \to U$.  
Denote by $\w{U}_{\leq 0}$ the product singular $n$-manifold $R^{n-k} Y \times \RR_{\leq 0}$.  There is the standard inclusion of underlying spaces $\w{U}_{\leq 0} \subset \w{U}$. 
By induction we have defined $T\w{U}$ and a canonical isomorphism $T\w{U} \cong T\RR^{n-k} \times T\RR \times TY$.  
Via this isomorphism, define the subspace $T\w{U}_{\leq 0} \subset T\w{U}$ as 
\[
T\w{U}_{\leq 0} \cong  T\RR^{n-k} \times T\RR_{\leq 0} \times TY  
\]
which is a free $\RR_{\geq 0}$-subspace over $\w{U}_{\leq 0}\subset \w{U}$.

There is the obvious $\RR_{\geq 0}$-equivariant projection $T\w{U}_{\leq 0} \to TU_{n-k}$ over the quotient map $\w{U}_{\leq 0} = \RR^{n-k}\times \RR_{\leq 0} \times Y \to \RR^{n-k} = U_{n-k}$.
Define the topological space $TU$ as the pushout
\begin{equation}\label{solids}
\xymatrix{
T\w{U}_{\leq 0} \ar[r]  \ar[d]
&
T\w{U} \ar@{.>}[d]
\\
TU_{n-k} \ar@{.>}[r]
&
TU
}
\end{equation}
The solid diagram is evidently $\RR_{\geq 0}$-equivariant over $U$, thus giving the $\RR_{\geq 0}$-map $TU\to U$, this $\RR_{\geq 0}$-action is free as desired.  
Provided $X$ has depth zero, the isomorphism $T(X\times U) \cong TX \times TU$ is immediate from the description of the product of singular manifolds (see Lemma~\ref{facts}).

The standard inclusion $R^{n-k}R_{>0} Z \to U^n_Z$ induces an $\RR_{\geq 0}$-equivariant map $T(R^{n-k}R_{>0}Z) \to TU$ over this inclusion.  
Notice also that for each $j$ the inclusion of the $j^{\rm th}$ stratum $U_j \to U$ induces a continuous $\RR_{\geq 0}$-equivariant map $TU_j \to TU$ over the standard inclusion $U_j\subset U$.  
We point out that a fiber of $TU\to U$ over $x\in U_{n-k}$ is canonically isomorphic to $\RR^{n-k}\times (\RR_{\geq 0} \times Y)_{/\sim} = \RR^{n-k} \times C(Y) =  U$ as a $\RR_{\geq 0}$-space, whereas a fiber over a point $x\in U\smallsetminus U_{n-k}$ is a fiber of $T\w{U} \to \w{U}$.

Let $V^n_Z \xra{f} U^n_Y$ be a morphism between basics of depths $l$ and $k$ respectively.  
If $l<k$ then $f$ factors through $R^{n-k}R_{>0} Y$ and we define the map $Df\colon TV\to T(R^{n-k}R_{>0} Y \subset TU$ where the first map is defined by induction.  
Suppose $l=k$.  From the definition of a morphism between equal depth basics,
there is diagram
\[
\xymatrix{
\w{V}  \ar@{.>}[r]^{\w{f}}  \ar[d]
&
\w{U} \ar[d]
\\
V \ar[r]^f 
&
U
}
\]
where the vertical maps are the standard quotient maps and with $\w{f}$ a morphism of protracted singular $n$-manifolds which restricts to a map $\w{V}_{\leq 0} \to \w{U}_{\leq 0}$ over $V_{n-k} \to U_{n-k}$.  
By induction, there is a map $D\w{f}\colon T\w{V} \to T\w{U}$, and $D\w{f}$ gives a map of the solid diagrams depicted in~(\ref{solids}).  There results a map of pushouts $TV\to TU$.  This map is $\RR_{\geq 0}$-equivariant because everything in sight is.  

Define $TX = \colim TU$ where the colimit is over the category $\{U\to X\}$ of basics over $X$.  
That $TX$ is a free $\RR_{\geq 0}$ space over $X$ follows because each term in colimit was so, and each morphism in the colimit respected this structure.  
Clearly the assignment $X\mapsto TX$ is functorial.  Moreover, for $Y$ of depth zero the canonical isomorphism $T(X\times Y) \cong TX\times TY$ is apparent. This completes the definition of the free $\RR_{\geq 0}$-space $TX$ over $X$.  
\end{definition}

The space $TX$ associated to a protracted singular $n$-manifold $X$ has a canonical structure of a protracted singular $2n$-manifold -- the atlas is given from the open cover $TU$, which is evidently a protracted singular $2n$-manifold, as $U\to X$ ranges through coordinate charts.  
Moreover, for $X\xra{f} Y$ a morphism of protracted singular $n$-manifolds, then $Df\colon TX \to TY$ is a morphism of protracted singular $2n$-manifolds for which $TX \xra{\cong} f^\ast TY$ is an isomorphsim over $X$.  

\begin{remark}\label{remark:TXisNotABundle}
We warn the reader that while we call $TX \to X$ the tangent \emph{bundle}, it is only a fiber bundle in the case that $X$ has depth zero.  In the general case, the fibers over different points are not isomorphic.  For instance, in the case that $X$ is a nodal surface, the fiber over a node is isomorphic to $C(S^1\sqcup S^1)$ (as an $\RR_{\geq 0}$-space) whereas the fiber over a smooth point is isomorphic to $\RR^2$ (as an $\RR_{\geq 0}$-space).  
In general, the fiber of $TX \to X$ over $x\in X$ is a singular $n$-manifold which is isomorphic to  a chart containing $x$ in its deepest stratum - the isomorphism class of this chart is well-defined in light of Lemma~\ref{only-equivalences} (this isomorphism is not as $\RR_{\geq 0}$-spaces).  
(This suggests the existence of an exponential map -- while the theory developed here is likely amenable to such methods, we do not develop them here.)  
In any event, while $TX\to X$ is not a fiber bundle of topological spaces, for each open stratum $X_{j-1,j}\subset X$ the restriction $TX_{|X_{j-1,j}}\to X_{j-1,j}$ is a fiber bundle.  
\end{remark}

\subsection{Conically smooth maps}\label{conically-smooth}
Until this moment, we have only considered ``smooth'' open embeddings $X\to Y$.  
We now define the space of ``smooth'' maps $X\to Y$ which are not necessarily open embeddings.  
This larger class of maps is necessary for discussing vector fields, flows, and (positive codimension) embeddings.  
We refer to this notion of smoothness as \emph{conical smoothness} and say that such a map is \emph{conically smooth}.

Let $X$ and $Y$ be protracted singular manifolds.  
We define the set $C^\infty(X,Y)$ of conically smooth maps from $X$ to $Y$ as follows.  
We use induction on the depth $l$ of $Y$ with the base case $l=0$ so that $Y$ is an ordinary smooth manifold.  
If $X$ has depth zero define $C^\infty(X,Y)$ as the familiar set of smooth maps.  
Suppose we have defined $C^\infty(X,Y)$ for $X$ of depth less than $k$.  
Let $U=U^n_Z$ be a basic of depth $k$.  
Define $C^\infty(U,Y) \subset C^\infty(R^{n-k}RZ, Y)$ as the subset of those $f$ which factor through the quotient map $\RR^{n-k}\times (\RR \times Z) \to \RR^{n-k}\times C(Z)$.  
For $X$ of depth $k$, define $C^\infty(X,Y)$ as the subset of those continuous maps $f$ for which for each chart $U\xra{\phi} X$ the composition $f \phi\in C^\infty(U,Y)$ is conically smooth.
Assume then that we have defined $C^\infty(X,Y')$ whenever $Y'$ has depth less than $l$.  
Suppose $Y=V^m_Q$ is a basic of depth $l$ and $X = U^n_Z$ is a basic.  
Define $C^\infty_0(U,V)$ as the set of those continuous maps $f$ for which the diagram can be filled
\[
\xymatrix{
\RR^{n-k}\times \RR \times Z  \ar@{.>}[r]^{\w{f}}  \ar[d]
&
\RR^{m-l}\times \RR \times Q  \ar[d]
\\
\RR^{n-k}\times CZ \ar[r]^f 
&
\RR^{m-l}\times CQ 
}
\]
with $\w{f}\in C^\infty(R^{n-k}RZ,R^{m-l}RQ)$ where the vertical maps are the standard quotient maps.  
Define $C^\infty(X,Y)$ as the set of those continuous maps $f$ for which for each $x\in X$ there is a pair of charts $(U,0)\xra{\phi} (X,x)$ and $(V,0)\xra{\psi} \bigl(Y,f(x)\bigr)$ for which $f\bigl(\phi(U)\bigr) \subset \psi(V)$ and the composition $\psi^{-1} f \phi \in C^\infty_0(U,V)$.  
These sets $C^\infty(-,-)$ are closed under composition.  

\begin{example}
The inclusion of a stratum $X_j\to X$ is a conically smooth map.  

\end{example}

\begin{remark}
The following example serves to warn the reader two-fold.
Consider the open `Y'-graph $Y =U^1_{\bf{3}}= C(\{a,b,c\}) = (\RR\times \{a,b,c\})_{/\sim}$.  
Consider the map $\RR\xra{f} Y$ given by $t\mapsto (\exp(\frac{-1}{1+t}),a)$ for $t<-1$, as $t\mapsto (\exp(\frac{-1}{t-1}),b)$ for $t>1$, and as $t\mapsto [(0,a)] = [(0,b)]$ otherwise.  
Then $f$ is conically smooth in addition to being a proper map with $0$ mapping to the $Y$-point.  
This illustrates that a conically smooth map need not preserve stratifications in any sense.  
It also illustrates that the obvious inclusion $C^\infty_0(U,V)\subset C^\infty(U,V)$ is not an equality -- indeed, the depicted conically smooth map $f$ is not an element of $C^\infty_0$.  
\end{remark}

It is routine that $TX \to X$ and the action $\RR_{\geq 0}\times TX \to TX$ are conically smooth maps.  
It is likewise routine that a conically smooth map $X\xra{f} Y$ induces a conically smooth map $Df\colon TX \to TY$ of $\RR_{\geq 0}$-spaces over $X\xra{f} Y$.
Endow $C^\infty(X,Y)$ with the weakest topology so that the iterated derivative map $D^r\colon C^\infty(X,Y) \to \Top(T^r X, T^r Y)$ is continuous for each $r\geq 0$.  
With this topology, the composition maps $C^\infty(X,Y)\times C^\infty(Y,Z) \to C^\infty(X,Z)$ are continuous.  
It is an exercise to verify that, for $X$ and $Y$ singular $n$-manifolds, the apparent set map $\snglr_n(X,Y)\to C^\infty(X,Y)$ is the inclusion of a subset, and that for $X$ and $Y$ smooth, this space of maps agrees with the weak Whitney $C^\infty$ topology.

\begin{lemma}\label{part-o-1}
There are conically smooth partitions of unity.  
\end{lemma}
\begin{proof}
This follows the classical arguments founded on the existence of smooth bump functions -- here the essential observation is that $\{f(\iota V)\subset \iota U \mid V\xra{f} U\}$ forms a basis for $\iota U$, and for $\phi'\colon \RR_{\geq 0} \to [0,1]$ a smooth map which is $1$ on $[0,\epsilon/2]$ and $0$ outside $[0,\epsilon)$ for some $\epsilon>0$, then the composition $\phi\colon \iota U\approx C(S^{k-1}\star \iota Y) \xra{pr} \RR_{\geq 0} \xra{\phi'} [0,1]$ is a conically smooth map whose support is in an (arbitrarily) small neighborhood of $0\in \RR^{n-k}\subset \iota U$.  The remaining points are typical given that $\iota X$ is paracompact.  
\end{proof}

\subsection{Vector fields and flows}

Consider the dense $\RR_{\geq 0}$-subspace over $X$
\[
PX= \bigcup_{0\leq j \leq n} TX_{j-1,j} \subset TX
\]
which is the union of the tangent spaces of the open strata (which are ordinary smooth manifolds).  Refer to $PX$ as the \emph{parallel tangent space of $X$}.  

\begin{definition}
A \emph{parallel vector field} on $X$ is a conically smooth section $V\colon X \to TX$ which factors through $PX \subset TX$.
\end{definition}

Vector sum gives the commutative diagram of $\RR_{\geq 0}$-spaces over $X$
\[
PX\times_X TX 
\xra{+}
TX 
\]
whose restriction to $PX\times_X PX$ factors through $PX$.  
This operation $+$ is defined through local coordinates $(U,0)\mapsto (X,x)$ at a point of depth $j$ by the expression 
\[
(TU_j)\times_{U_j} (U_j\times U) \xra{+} (U_j\times U)~{},~{}~  \bigl((u,v),(u,[w,s,y])\bigr)\mapsto \bigl(u,[v+w,s,y]\bigr)~.
\]
By inspection $+$ is associative and commutative, thereby making $PX$ into a vector space over $X$ and $TX$ a module of $PX$ over $X$.  
As so, the space of parallel vector fields on $X$ is a vector space.

\begin{lemma}\label{flowing}
Let $V$ be a parallel vector field on $X$.   
Then $V$ can be integrated to a conically smooth flow $\eta\colon \RR\times X \dashrightarrow X$ defined on a neighborhood of $\{0\}\times X$.  
\end{lemma}
\begin{proof}
Let $k$ be the depth of $X$.  For $k=0$ the result is classical -- from the existence, uniqueness, and smooth dependence on initial conditions for ODE's.  Assume $k>0$.

By definition of $V$ being conically smooth there is a collection of charts $\{U_\alpha \xra{\phi_\alpha} X\}$ for which $\{\phi_\alpha\bigl((U_\alpha)_k\bigr)\}$ covers $X_{n-k}$, and for each $\alpha$ there is a filling 
\[
\xymatrix{
\w{U}_\alpha  \ar@{.>}[r]^-{\w{V}_\alpha}  \ar[d]
&
T\w{U}_\alpha  \ar[d]
\\
U_\alpha  \ar[r]^-{V_{|U_\alpha}}
&
TU_\alpha.
}
\]
From the construction of $\partial \w{X}_{n-k}$, the images of the collection $\{\partial(\w{U}_k)_\alpha\to \partial \w{X}_{n-k} \}$ is a cover.   
Upon possibly shrinking $\w{X}$ over $X$, and using the canonical isomorphism $\w{X}\smallsetminus \w{X}_{\leq 0} \xra{\cong} X\smallsetminus X_{n-k}$, the collection $\{\w{\phi}_\alpha(\w{U}_\alpha)\}\amalg  \{\phi(U)\mid U\xra{\phi} X\text{ of depth $>k$}\}$ is an open cover of $\w{X}$.
Upon choosing a partition of unity $\{\psi_{\alpha'}\}$ subbordinate to this open cover, the expression $\w{V} = \Sigma_{\alpha'} \psi_{\alpha'} \cdot \w{V}_{\alpha'}$ gives a filling
\[
\xymatrix{
\w{X}  \ar@{.>}[r]^{\w{V}}  \ar[d]
&
T\w{X}  \ar[d]
\\
X  \ar[r]^V
&
TX
}
\]
after possibly shrinking $\w{X}$ over $X$ -- here $\w{V}_{\alpha'} = V_{|U_{\alpha'}}$ if $\alpha'\notin\{\alpha\}$ which makes sense through $\w{X}\smallsetminus \w{X}_{\leq 0} \xra{\cong} X\smallsetminus X_{n-k}$.

Evidently, $\w{V}$ factors through $PX\times_{TX} T\w{X}$.  
Straight from the definition of $\w{U}$ and $P(-)$, the canonical morphism $PU\times_U T\w{U}\to T\w{U}$ factors through $P\w{U}$, so likewise with $U$ replaced by $X$.  
By induction on $k$, there is a flow $\w{\eta}\colon \RR\times \w{X} \dashrightarrow \w{X}$ defined on a neighborhood of $\{0\}\times \w{X}$.  

It is immediate to see (using local coordinates for instance) that this morphism is an isomorphism onto its image, which is characterized as the subobject $I\subset P\w{X}$ for which $I_{|X_{n-k}} = ker\bigl(T(\partial \w{X}_{n-k}) \to I_{|\partial X_{n-k}}\bigr)$ where the arrow $\partial \w{X}_{n-k} \to \w{X}$ is a conically smooth embedding.  
As so, the restriction of the flow $\w{\eta}\colon \RR\times \partial\w{X}_{n-k} \dashrightarrow \w{X}$ factors through $\partial\w{X}_{n-k}$.  
Because $\partial \w{X}_{n-k}$ separates $\w{X}$, it follows likewise that the restriction factors $\w{\eta}\colon \RR\times \w{X}_{\leq 0} \dashrightarrow \w{X}_{\leq 0}$.  
There results the map on quotients $\eta\colon \RR\times X \dashrightarrow X$ as desired.

\end{proof}

\subsection{Proof of Proposition~\ref{tubular-neighborhood}}

The statement of Proposition~\ref{tubular-neighborhood} immediately appeals to using Morse theoretic arguments.  However, such a discussion would require the existence and utility of Morse functions our context of singular manifolds, which would require some careful development of transversality in this singular setting.  
While the presentation of singular manifolds featured in this article is likely amenable to a theory of transversality, we find a proof of Proposition~\ref{tubular-neighborhood} which involves far less development.

Let $X$ be a singular $n$-manifold. 
Choose a conically smooth partition of unity $\{\psi_\alpha\}$ subbordinate to the open cover $\{\w{U}_\alpha\mid U_\alpha\xra{\phi_\alpha} X\}$ of $\w{X}$.

For each $U_\alpha\xra{\phi_\alpha} X$ of depth $k$, written $U_\alpha = U^n_Y$, consider the flow $\eta^\alpha\colon \RR\times \w{U}_\alpha \to \w{U}_\alpha$ given by $(t,(v,s,y)) = (v,s+t,y)$ -- it is a conically smooth map having no stationary points and $\eta^\alpha_0$ is the identity.  
For each $x\in \w{U}_\alpha$, denote the restriction  $\eta^\alpha(x)\colon \RR\times \{x\} \to \w{U}_\alpha$ which is a conically smooth path.  This determines the vector field $V'_\alpha\colon \w{U}_\alpha \to P\w{U}_\alpha$ given by $V_\alpha'(x)= D_0\eta^\alpha(x)(1)$.  This vector field is parallel and its flow is $\eta^\alpha$.  
For $U_\alpha$ of depth greater than $k$, define $V_\alpha'\colon \w{U}_\alpha \to P\w{U}_\alpha$ as the zero vector field.

For each $\alpha$ consider the parallel vector field $V_\alpha\colon \widetilde{X} \to P\w{X}$ given by $V_\alpha(x) = \psi_\alpha(x)\cdot D_{\w{\phi}_\alpha^{-1}x} V_\alpha'(\w{\phi}_\alpha^{-1}x)$ if $x\in \w{\phi}_\alpha(\w{U}_\alpha)$ and $V_\alpha(x) = 0$ otherwise -- this indeed defines a \emph{conically smooth} vector field and it is parallel because its support is a scale of a parallel vector field.  
Define the parallel vector field $V\colon \widetilde{X} \to P\widetilde{X}$ by
\[
V(x) = \sum_\alpha V_\alpha~.
\]
Because the sum is locally finite and each summand is parallel, this indeed describes a conically smooth parallel vector field.

Note that the restriction of $V_\alpha$ to $\w{\phi}_\alpha\bigl((\w{U}_\alpha)_{\leq 0}\bigr)$ does not vanish and projects to the zero vector field under the projection $\w{U}_{\leq 0} \to \RR^{n-k} \times Y$ whose fibers are the oriented singular $1$-manifold $\RR_{\geq 0}$.    
It follows from its defining expression that the restriction of $V$ to $\w{X}_{\leq 0}$ is non-vanishing and projects as the zero vector field under $\w{X}_{\leq 0}\to \partial \w{X}_{n-k}$ whose fibers are singular $1$-manifolds.

Consider the flow $\eta\colon \RR\times \w{X} \dashrightarrow \w{X}$ of the parallel vector field $V$ -- it exists by way of Lemma~\ref{flowing}.  
From the above considerations, the restriction of the flow $\eta\colon \RR_{\leq 0} \times \w{X}_{\leq 0} \to \w{X}_{\leq 0}\subset \w{X}$ is defined for all non-positive time; and the restriction of this flow gives a conically smooth map $\eta\colon \RR_{\leq 0} \times \partial \w{X}_{n-k} \to \w{X}_{\leq 0}$ which is an isomorphism of singular $n$-manifolds.  
From the openness of the domain of $\eta$, there is a conically smooth map $\epsilon \colon \partial \w{X}_{n-k} \to \RR_{>0}$ for which the flow $\eta\colon \RR_{< \epsilon} \times \partial \w{X}_{n-k} \to \w{X}$ gives a morphism of singular $n$-manifolds with $\eta_0$ the standard inclusion $\partial \w{X}_{n-k} \hookrightarrow \w{X}$.  
Because the fibers of the projection $\partial \w{X}_{n-k}\to X_{n-k}$ are compact, by shrinking $\epsilon$ as necessary we can assume $\epsilon$ is constant along fibers; that is, $\epsilon \colon X_{n-k} \to \RR_{>0}$.  

Denote the projection $p\colon \w{X}_{\leq 0} \to \partial \w{X}_{n-k}$.  
Consider the equivalence relation $\sim$ on $\RR_{<\epsilon}\times \partial \w{X}_{n-k}$ which is the subset of the product $\{(t,x),(t',x')\mid t,t'\leq 0~,~ p(x) = p(x')\}$.  
Define $\w{X}_{n-k} :=(\RR_{<\epsilon} \times \partial\w{X}_{n-k})_{/\sim}$.  
The atlas of $\RR_{<\epsilon}\times \partial \w{X}_{n-k}$ manifestly gives an atlas of $\w{X}_{n-k}$ making it a singular $n$-manifold.  
Note the standard inclusion $X_{n-k} \to \w{X}_{n-k}$ as $x\mapsto [0,p^{-1}x]$ and the canonical isomorphism of singular $n$-manifolds $\w{X}_{n-k}\smallsetminus X_{n-k} \cong (0,\epsilon)\times \partial \w{X}_{n-k}$.  
This equivalence relation is such that the restriction of the map $\w{X} \to X$ 
to $\RR_{<\epsilon} \times \partial \w{X}_{n-k}$ factors through the morphism of singular $n$-manifolds $\w{X}_{n-k} \to X$ under $X_{n-k}$.  This finishes the proof of the proposition.

\subsection{Regular neighborhoods}\label{embeddings}

We establish a version of Whitney's embedding theorem for singular manifolds, in addition to the existence of tubular neighborhoods.  We imitate classical arguments (for instance, as found in~\cite{spivak}) so we will only indicate the arguments.

\begin{lemma}\label{embedding}
Let $X$ be a finite singular $n$-manifold. There is a proper conically smooth embedding 
\[
X\hookrightarrow \RR^N
\]
for some $N$. 
\end{lemma}
\begin{proof}

Let $\cA' = \{(U'_i\xra{f'_i} X)\}_{0\leq i \leq r}$ be a finite atlas for $X$.  Replace this atlas by another $\cA=\{(U_i\xra{f_i} X)\}$ for which $f_i(U_i)\subset \ov{f'_i(U'_i)}$ is contained in the closure.  Choose a conically smooth partition of unity $\{\psi_i\}$ subordinate to $\cA$.  
Denote $U_i = U^n_{Y_i}$ for $Y_i$ a compact singular $(k_i-1)$-manifold of strictly less depth than $X$.  Because each $Y_i$ is compact and therefore finite, inductively, choose an conically smooth (proper) embedding $e'_i\colon Y_i \to \RR^{N_i-1}\subset S^{N_i-1}$.  There results a conically smooth embedding $e_i\colon U_i = \RR^{k_i}\times C(Y_i) \hookrightarrow \RR^{k_i}\times C(S^{N_i-1}) \approx \RR^{k_i+N_i}$ where this last map is given via polar coordinates by the radial map $s \mapsto  \mathsf{exp}(-\frac{1}{s})$ for $s\in \RR_{>0}$ and $s\mapsto 0$ else.  
Then the expression $(\psi_1\cdot e_1 , \dots , \psi_r \cdot e_r , \psi_1,\dots , \psi_r)$ described a conically smooth embedding $e''\colon X \hookrightarrow \RR^{N-1}$ where $N-1= r+ \sum_{1\leq i \leq r} k_i+ N_i$.  

We now modify $e''$ to a proper map.  
Choose a sequence of compact subsets $K_1\subset K_2\subset \dots\subset \iota X$ such that the union $\bigcup_i K_i = \iota X$ and $K_i$ is contained in the interior of $K_{i+1}$.  Using conically smooth bump functions, for each $i$ choose a conically smooth map $b_i\colon \iota X \to [0,1]$ which is $0$ on $K_i$ and $1$ on $K_{i+1}$.  Define $B\colon \iota X \to \RR $ as $\sum_i b_i$; this sum makes sense because for each $x\in \iota K$ there is a neighborhood on which $b_i(x)\neq 0$ for only finitely many $i$.  The map $e=B\times e''\colon X \to \RR^N$ is a conically smooth proper embedding.  
\end{proof}

Fix a properly embedded singular $n$-manifold of depth $k$.  
The $(n-k)^{\rm th}$ stratum $X_{n-k}\subset \RR^N$ is then a properly embedded smooth $(n-k)$-submanifold, with unit sphere bundle denoted $SX_{n-k} \to X_{n-k}$.  Denote the intersection with the normal bundle as $N_k X =  \{(x,v)\in TX_{|X_{n-k}}\mid v\perp TX_{n-k}\}$.  
For $\epsilon\colon X_{n-k} \to \RR_{>0}$ a smooth map  
denote the $\RR_{\geq 0}$-subspace $N^\epsilon_k  X = \{(x,v)\in T\RR^N_{|X_{n-k}}\mid \rVert v- N_k X\rVert < \epsilon(x) \lVert v \rVert\}$ over $X$.

\begin{lemma}\label{normal-space}
Let $X\subset \RR^N$ be a conically smoothly properly embedded finite singular $n$-manifold.  Let $\epsilon\colon X_{n-k} \to \RR_{>0}$ be an arbitrary conically smooth map.  
Then the intersection with the exponential spray
\[
\exp(N^\epsilon_k X)\cap X= \{x+v\mid (x,v)\in N^\epsilon_k X\}\cap X 
\]
is an open neighborhood of $X_{n-k}\subset X$, and therefore canonically inherits the structure of a singular $n$-manifold.

\end{lemma}

\begin{proof}
It is sufficient to prove the statement for a conically smoothly embedded basic $i\colon U\subset \RR^N$.  Write $U=U^n_Y$, with $Y$ a compact singular $(k-1)$-manifold, so that $\iota U = \RR^{n-k} \times C(\iota Y)$.  Upon shrinking $U$ if necessary, by performing a diffeomorphism of $\RR^N$ we can assume $\RR^{n-k}\subset U \subset \RR^N$ is the standard embedding.  
It is enough to show that for any $[u,s,y]\in U$ there is a $t_0\geq 0$ for which for all $0\leq t\leq t_0$ the point $i[u,ts,y]$ is an element of $\mathsf{exp}(N^\epsilon_k X)$; that is, for each $0<t\leq t_0$ there is a vector $w_t\in (N_kX)_u$ for which $\lVert (i[u,ts,y]-u) - w_t\rVert < \epsilon(u)\lVert w_t \rVert$.  

Let $[u,s,y]\in U$ be arbitrary.  
Denote the vector $w=\ov{\alpha}_{0,u}(i)[u,s,y]\in T_u X\subset \RR^N$.  If $w=0$ then $s=0$ and there is nothing to prove.  Assume $w\neq 0$.
From the definition of $\ov{\alpha}_{0,u}(i) = \lim_{t\to 0} \gamma_{\frac{1}{t},u} \circ i\circ  \gamma_{t,u}$, there is a $t_0\geq 0$ for which for all $0<t\leq t_0$ there is the inequality
\[
\lVert  (\frac{i[u,ts,y]-u}{t}+u) - (w+u) \rVert~ <~\epsilon(u)\lVert w \rVert 
\]
which is to say $\lVert (i[u,ts,y]-u)   - tw \rVert <t\epsilon(u)\lVert w \rVert$ and we are finished.

\end{proof}

\begin{lemma}\label{regular-neighborhood}
Let $X\subset \RR^N$ be a conically smoothly properly embedded singular $n$-manifold.  
Then there is a conically smooth function $\delta\colon X\to \RR_{>0}$ for which the $\delta$-neighborhood $\nu$ of $X\subset \RR^N$ is a regular neighborhood -- denote the deformation retraction $r_t\colon \nu \to \nu$.  
Moreover, if $X$ is finite then there is an isotopy of $X\hookrightarrow \RR^N$ through proper conically smooth embeddings to one for which such a $\delta$ is bounded below by $1$.  
\end{lemma}

\begin{proof}
Again, we imitate classical methods so only indicate the proof.  
Once and for all choose a smooth family, as $\epsilon\in\RR_{\geq 0}$, of smooth bump functions (distributions) $c_\epsilon\colon \RR_{\geq 0} \to [0,1]$ for $\epsilon> 0$ which takes the value $1$ on the interval $[\frac{5\epsilon}{4}, \frac{7\epsilon}{4}]$ and takes the value $0$ outside the interval $[\epsilon , 2\epsilon]$.  
Choose a smooth function $d\colon \RR^N \to \RR_{\geq 0}$ for which $d$ is bounded by the continuous functions $\frac{2}{3} \mathsf{dist}(-,X) \leq d \leq \mathsf{dist}(-,X)$.  
Consider the vector field $\VV\colon \RR^N \to \RR^N$ given by 
\[
v~\mapsto~ \frac{ \int_X  c_{d(v)} v' }{\int_X c_{d(v)}} ~-~ v~.
\]
Because the support of $c_\epsilon$ is compact and $X$ is closed, these integrals converge.  
Moreover, being comprised of smooth functions, the vector field $\VV$ is smooth.  
The intuition for $\VV$ is that $\VV(v)$ points towards the center of mass of the locus of points of $X$ which are nearest to $v$.  
For instance, if the set $\{x'\in X \mid d(v,x') = d(v,X)\}$ is a singleton, then $\VV(v)$ is approximately $x'-v$, this approximation smoothly improving the closer $v$ is to $X$.  
In particular, if $x\in X$ then $d(x)=0$ and thus $\VV(x) = 0$.  

By existence and uniqueness of solutions to first order ODE, in addition to smooth dependence on initial conditions, this vector field can be integrated to a smooth function $\chi\colon \RR\times \RR^N \to \RR^N$.  
If $\{x'\in X\mid d(v,x') = d(v,X)\}$ is a singleton, then $\chi_t(v)$ approximates the straight-line flow from $v$ to $x'$.  
In particular, from the tubular neighborhood theorem for ordinary submanifolds, for $v$ close enough to a point in the open $n$-stratum $X_{n-1,n}\subset X$, the flow $\chi_t(v)$ for $t\geq 0$ approximates the straight-line path from $v$ to its nearest point in $X$.  

Let $x\in X$.  From Lemma~\ref{only-equivalences} there is a unique (up to isomorphism) compact singular manifold $Y$ for which a neighborhood of $x\in X$ is of the form $U^n_Y$.  
Consider replacing $X$ by $\exp(T_xX)\cong T_x X\subset T_x \RR^N \cong \RR^N$ to obtain a new vector field $\VV'$.  
Because $T_xX\subset \RR^N$ is $\RR_{>0}$-invariant, in other words \emph{conical}, then  $v\notin T_x X$ implies $v\in \RR^N\smallsetminus \{0\} \cong \RR_{>0}\times S^{N-1}$ and the projection of $v$ onto $\RR_{>0}$ is strictly negative.  
Moreover, from the compactness of $Y$, for $v\in T_xX$ with $\lVert v\rVert = 1$ then the magnitude of the image of $v$ under this projection is bounded below by a positive number.  
In particular, the set $\{v\in \RR^N \mid \lim_{t\to \infty} \chi'_t(v) \in T_xX \} = \RR^N$ is everything and the assignment $v\mapsto \lim_{t\to \infty} \chi'_t(v)$ describes a continuous map $\RR^N \to T_xX$ -- likewise, for $B_{\delta_x}(x)$ in place of $\RR^N$.   
From Lemma~\ref{normal-space} applied to the singular $n$-manifold of depth $n$ which is $(X,x)$ -- the marked point $x$ regarded as a singularity of depth $n$ -- we can choose a $\delta_x$ for which a $\delta_x$-neighborhood in $X$ of $x$ is within an arbitrarily small neighborhood of $\exp(T_xX)\subset \RR^N$.  
As so, the original vector field $\VV$ is arbitrarily close to $\VV'$ in a small enough neighborhood of $x$.  
In particular, for a small enough $\delta_x>0$, the subset $\{v\in B_{\delta_x}(x)\mid \lim_{t\to \infty} \chi(v) \in X\} = B_{\delta_x}(x)$ is everything and furthermore, the assignment $v\mapsto \lim_{t\to \infty} \chi_t(v)$ describes a continuous map $B_{\delta_x}(x) \to X$.  
In this way, because $X\subset \RR^N$ is properly embedded, we can choose a continuous map $\delta\colon X\to \RR_{\geq 0}$ such that for every $v$ in the $\delta$-neighborhood $\nu$ of $X$ in $\RR^N$, the assignment $v\mapsto \lim_{t\to \infty} \chi_t(v)$ is a continuous function $\nu \to X$.  
Upon reparametrizing the flow we obtain a map $r' \colon [0,1)\times \nu \to \nu$ which extends continuously to $r\colon [0,1]\times \nu$ witnessing a deformation retraction of $\nu$ onto $X$.  

Suppose $X$ is finite.  
Choose a finite atlas $\cA=\{(U_i,\phi_i)\}$ for $X$.  For each $U_i$ in this collection, choose a sequence of endomorphisms $f_{ij}\colon U_i\to U_i$ for which the closure of the image $\ov{f_{ij}(\iota U_i)} \subset \iota U_i$ is compact and the union $\bigcup_j f_{ij}(\iota U_i)  = \iota U_i$.  
Consider the singular submanifold $X_j = \bigcup_i \phi_i\bigl(f_{ij}(U_i)\bigr)\subset X\subset \RR^N$.  
Then the closure of $X_j$ in $\RR^N$ is compact and contained in $X$.  
So the restriction $\delta_{|X_j}$ is bounded below by a positive number.  
Let $R_j = \mathsf{Sup}\{\lVert x\rVert \mid x\in X_j\}<\infty$.  
By iteratively (and smoothly) scaling outside $R_{j-1}$ through an isotopy rel $B_{R_{j-1}}(0)$, we can assume $\delta(x)$ is bounded below by $1$ for $x\in X_j$.  
This completes the proof.

\end{proof}

\subsection{Homotopical aspects of singular manifolds}
We record the following technical results which are referenced in the course of the proof of Theorem~\ref{pushforward-formula}.

\begin{lemma}\label{smush}
Let $X$ be a finite singular $n$-manifold. There is a conically smooth map $\beta\colon [0,1] \times X \to X$ with $\beta_0 = 1_X$ the identity map and such that the collection of closures $\{\ov{\beta_t(X)}\mid t\in (0,1]\}$ is a compact exhaustion of $X$. 

Furthermore, for each $t\in (0,1]$, the inclusion $\ov{\beta_t(X)} \to X$ is a trivial cofibration of topological spaces.  

\end{lemma}

\begin{proof}

We give a parallel vector field $V$ on $X$ whose flow reparametrizes to the desired $\beta\colon [0,1]\times X\to X$.  
For this statement we use induction on the dimension $n$.  If $n<0$ the statement is vacuously true.  
Suppose the statement is true for $n'<n$.  
If $X = \emptyset$ the statement is vacuously true.  Assume otherwise.
From Theorem~\ref{collar=fin}
we can write $X$ as a finite iteration of collar-gluings from basics.  Let this number of iterations be $r$.  We proceed by induction on $r$.

If $r=1$ then $X=U^n_Y$ is a basic.  
Consider the path of endomorphisms $\gamma\colon \RR_{>0}\times U \to U$ given by $(t,[u,s,y])\mapsto [\frac{u}{t}, \frac{s}{t}, y]$.
Consider the vector field $W = \frac{d}{dt}\gamma_t |_{t=1}$ on $U$.  
Choose a diffeomorphism $\RR \cong \RR_{<1} $ which is the identity on $\RR_{\leq 0}$, thereby implementing an endomorphism $U\to U$ upon identifying $U=\RR^{n-k}\times CY = C(S^{n-k-1}\star Y) = (\RR\times S^{n-k-1}\star Y)_{/\sim}$.  
Define $V$ as the pullback of $W$ along this endomorphism.  
It is immediate that $V$ has the desired properties.  

Assume $r>1$ and write $X=X_-\cup_{RP} X_+$ as a collar-gluing in where both $X_\pm$ themselves can be written as $<r$-times iterated colar-gluings from basics.  
By induction there are $V_\pm$ and $V_0$ as in the statement for $X_\pm$ and $P$ respectively.

Choose orientation preserving diffeomorphisms $e_+\colon \RR \cong (1,\infty)$ and $e_-\colon \RR\cong  (-\infty,-1)$ which are the identity on $[2,\infty)$ and $(-\infty,-2]$ respectively.   These diffeomorphisms implement morphisms $E_\pm \colon X_\pm \to X$ such that  $\{E_\pm(X_\pm) , (-2,2)\times P\}$ is an open cover of $X$ with $E_-(X_-) \cap (-1,1)\times P = E_+(X_+) \cap (-1,1)\times P =\emptyset =  E_-(X_-)\cap E_+(X_+)$.  
Choose a partition of unity $\{\psi_\pm , \psi_0\}$ subordinate to this open cover.  
Denote the vector field $W$ on $X$ given as $\psi_\pm(x)\cdot D_x(E_\pm)\bigl(V_\pm(E_\pm^{-1}(x))\bigr)$ for $x\in E_\pm(X_\pm)$ and as zero vector field elsewhere -- $W$ is well-defined because the two sets $E_\pm(X_\pm)$ are disjoint.  
Denote the vector field $W_0$ on $X$ given as $\bigl(0,\psi_0(x)\cdot V_0(pr_V(x))\bigr)$ for $x\in \RR\times P$ and as the zero vector field elsewhere.  
These vector fields $W$ and $W_0$ are indeed conically smooth, and, being scalings of locally parallel vector fields, are parallel.  
Define $V = W+W_0$.

There is a conically smooth map $\delta \colon X\to\RR_{>0}$ for which the flow of $V$ gives a well-defined conically smooth map $\eta \colon [0,\delta)\times X \to X$. 
Moreover, by construction, for each $\epsilon\colon X\to \RR_{>0}$ there is a compact subset $K\subset X$ for which $\eta_\epsilon(X)\subset K$.  Indeed, take $K=E_-(K_-)-\cup [-2,2]\times K_0 \cup E_+(K_+)\subset E_-(X_-)\cup \RR\times V \cup E_+(X_+) = X$ where the other subscripted $K$'s exist by induction.

We point out that for each $t\in [0,1]$ the map $\beta_t$ is an open embedding, and in particular an isomorphism onto its image.  
Denote the closure $\ov{X}:=  \ov{\beta_{\frac{1}{2}}(X)}$ and the boundary $\partial \ov{X} = \ov{X} \smallsetminus \beta_{\frac{1}{2}}(X)$.  
Each point $x\in\partial \ov{X}$ lies on a unique flow line of $V$.  As so, $\beta$ extends to a continuous map $\ov{\beta}\colon [0,1]\times \ov{X} \to \ov{X}$ whose restriction to $[0,1]\times \partial \ov{X}$ is an embedding, and whose restriction to $X\cong \beta_{\frac{1}{2}}(X)$ is a an open embedding for each $t\in [0,1]$.  
With this, one can conclude that $\partial \ov{X} \to \ov{X}$ is a cofibration of topological spaces, and one can construct a homeomorphism $X \cong  (0,1/2]\times \partial \ov{X}    \coprod_{\partial \ov{X}} \ov{X}$ with the pushout.  It follows that $\ov{X} \to X$ is a cofibration.

\end{proof}

Recall the map of quasi-categories $\widehat{(-)}\colon \psnglr_n \to \cP(\bsc_n)$ given by the restricted Yoneda map.

\begin{lemma}\label{covers=colims}
Let $X$ be a finite singular $n$-manifold and let $\cU\subset \cO(X)$ be a countable subposet of open subsets.  Suppose each $O\in \cU$ is finite as a singular $n$-manifold, that the union $\bigcup_{O\in \cU} O = X$, and for each finite subset $\{O_i\}\subset \cU$ the collection $\{ O\in \cU \mid O\subset \bigcap O_i\}$ is an open cover of $\bigcap O_i$.  
Denote the functor $\cU^\triangleright \to \psnglr_n$ given by $O\mapsto O$ and $\infty \mapsto X_{\cU}$, the singular premanifold determined by the open cover $\cU$ (see Example~\ref{cover-refinement}).  
Then the composite 
\[
\cU^\triangleright \xra{\ov{p}} \psnglr_n \xra{\widehat{(-)}} \cP(\bsc_n)
\]
is a colimit diagram.

\end{lemma}

\begin{proof}
Denote by $\ov{(-)}\colon \psnglr_n \to \Top^{\bsc_n^{\op}}$ the restricted Yoneda map of $\Top$-enriched categories which factors $\widehat{(-)}$ through the coherent nerve construction.  
We will show
\begin{equation}\label{first-hocolim}
\hocolim_{O\in \cU} \ov{O} \xra{\simeq} \ov{X}_{\cU}
\end{equation}
is an equivalence in the (enriched) projective model structure on $\Top^{\bsc_n^{\op}}$.  
For this it is sufficient to show that~(\ref{first-hocolim}) is an equivalence upon evaluating at each $U\in \bsc_n$.  

Fix $U\in \bsc_n$.  
We will explain the factorization of the arrow in~(\ref{first-hocolim}) as the following zig-zag of equivalences:
\begin{eqnarray}
\hocolim_{O\in \cU} \ov{O}(U) 
& = &
\hocolim_{O\in \cU} \psnglr_n(U,O)
\nonumber
\\
&\xla{\simeq}& \label{second}
\hocolim_{O\in \cU} \ov{\sS}(U,O)
\\
&\xla{\simeq}& \label{fourth}
\hocolim_{O\in \cU} \hocolim_{t} \ov{\sS}^{K^O_{t}}(U,O)
\\
&\simeq&\label{fifth}
\hocolim_{\{t\}} \hocolim_{O\in \cU} \ov{\sS}^{K^O_t}(U,O)
\\
&\xra{\simeq} &\label{sixth}
\hocolim_{\{t\}} \colim_{O\in \cU} \ov{\sS}^{K^O_t}(U,O)
\\
&\cong &\label{seventh}
\hocolim_{\{t\}}  \ov{\sS}^{K^O_t}(U,X_\cU)
\\
&\xra{\simeq} & \label{ninth}
\ov{\sS}(U,X_\cU)
\\
&\xra{\simeq}& \label{tenth}
\psnglr_n(U,X_\cU)
\end{eqnarray}

Write 
	$U=\RR^{n-k}\times C(Y)\cong C(S^{n-k-1}\star Y) = \bigl(\RR\times (S^{n-k-1}\star Y)\bigr)_{/\sim}$.  Denote by $U'$ the singular $n$-manifold which is the image of $U$ under the inclusion $\bsc_n \subset \snglr_n$.   
Fix a diffeomorphism $\RR \cong \RR_{<1}$ whose restriction to $\RR_{\leq 0}$ is the identity.  This then fixes a morphism $U\to U'$ by acting on the first coordinate.   We will identify $U$ with its image in $U'$.  

Notice the closed subspace $\ov{U} \subset U'$ consisting of those $[(t,(u,y))]$ with $t\leq 1$ -- this inclusion is a (trivial) cofibration of topological spaces.  
Because $Y$ is compact, so is $\ov{U}$.  The inclusion $U\to U'$ factors through $U\subset \ov{U}$ as a dense open subspace.  
Moreover, we have seen in the proof of Lemma~\ref{smush} the standard map $\beta\colon [0,1]\times \ov{U} \to \ov{U}$ with $\beta_0 = 1$ the identity and with $\beta_t(\ov{U})\subset U$ for $t>0$.  We point out that the linearly ordered set $[0,1]$, regarded as a category, is (homotopy) filtered.

Denote by $\ov{\sS}(U,-)\subset \psnglr_n(U,-)$ the subfunctor consisting of those maps from $U$ which extend to a map from $U'$.  By continuity, there is a continuous map $\ov{\sS}(U,-) \to \Top(\ov{U},-)$ to the space of continuous maps from $\ov{U}$.  Moreover, this continuous map factors through those maps from $\ov{U}$ which are cofibrations -- this follows from the inclusion $\ov{U} \to U'$ being a (trivial) cofibration and the maps from $U'$ being open maps.  
From the last two sentences of the previous paragraph, the inclusion $\ov{\sS}(U,-) \to \psnglr_n(U,-)$ is a point-wise equivalence.  
This establishes the equivalences~(\ref{second}) and~(\ref{tenth}).  

For $K\subset O$ a compact subspace, denote by $\ov{\sS}^K(U,O)\subset \ov{\sS}(U,O)$ the (open) subspace consisting of those maps $U\xra{f} O$ for which $f(U)\subset K$.  From Lemma~\ref{smush}, $\ov{\sS}(U,O) = \colim_{t\in (0,1]} \ov{\sS}^{K^O_t}(U,O)$ can be written as a (homotopy) filtered colimit = filtered homotopy colimit, where here $K^O_t = \ov{\beta_t(O)}\subset O$ is the closure of the image, which is compact by construction.  
This establishes the equivalences~(\ref{fourth}) and~(\ref{ninth}).  

The equivalence~(\ref{fifth}) is obtained by rewriting the nested colimit with the outer colimit still (homotopy) filtered.  The isomorphism~(\ref{seventh}) is the definition of $X_\cU$.  

It remains to explain the equivalence~(\ref{sixth}).  
From the construction of $\beta$, the inclusion $K^O_t \subset O$ is a (trivial) cofibration of topological spaces and for $0<t'<t\leq 1$ then $K^O_t\subset Int(K^O_{t'})$ is contained in the interior.  As so, for each pair $O\subset O'\in \cU$ and each $t\in (0,1]$, the map $K^O_t\subset O \subset O'$ is a cofibration of topological spaces.  
It follows that for each $O\subset O'\in \cU$ the map $\ov{\sS}^{K^O_t}(U,O) \to \ov{\sS}^{K^{O'}_{t'}}(U,O')$ is a cofibration for $t'\in (0,1]$ small enough.  
As so, the assignment $(\{t\},O)\mapsto \ov{\sS}(U,O)$ describes a diagram in $\Top$ where each arrow is a cofibration.  Because each finite intersection $\bigcap O_i$ of elements of $\cU$ is covered by elements of $\cU$, this diagram is cofibrant (in the projective model structure on diagrams in $\Top$).  It follows that the universal arrow~(\ref{sixth}) from the homotopy colimit to the colimit is an equivalence.

\end{proof}

Recall from Example~\ref{cover-refinement} that the canonical morphism $X_\cU \to X$ is a refinement.

\begin{theorem}\label{no-pre-mans}
Let $\cB$ be a quasi-category of basics.
The inclusion 
\[
\mfld(\cB) \xra{\simeq} \pmfld(\cB)
\]
is an equivalence of quasi-categories.
\end{theorem}

\begin{proof}
Because $\mfld(\cB) = \snglr_n\times_{\psnglr_n} \pmfld(\cB)$ and $\pmfld(\cB) \to \psnglr_n$ is a right fibration, it is sufficient to prove the theorem for the case $\cB = \bsc_n$.  
It is enough to show that a refinement is an equivalence in $\psnglr_n$.  
Let $r\colon \dddot{X} \to X$  be a refinement.  
From Lemma~\ref{refs-pullback}, for every $U\xra{f}X$ there is the pullback diagram of singular premanifolds
\[
\xymatrix{
{\dddot{U}} \ar[d]^{r_|}  \ar[r]^{\dddot{f}}
&
{\dddot{X}}    \ar[d]^r
\\
U \ar[r]^f
&
X
}
\]
in where the morphism $r_|$ is a refinement.  To show that $r$ is an equivalence is to show $r_|$ is an equivalence for every $U\xra{f} X$.  
We can therefore reduce to the case $X=U$ is a basic.  When appropriate, we will denote this refinement with a subscript $r_U$.

Choose $p_0\in \RR^{n-k}\subset  U$.  
By translating as necessary, we can assume $p_0=0\in \RR^{n-k}$.  
From the definition of a refinement, there is a morphism $V\xra{a} \dddot{U}$ for which $0 \in f( V)$.  
Recall from Lemma~\ref{local-basis} the map $\beta$ which we denote as $\beta^U$ to remember its dependence.  
Recall that $\{\beta_t^U(U)\mid t\geq 0\}$ forms a local base for the topology about $0\in\RR^{n-k}$, that $\beta^U_0=1_U$ is the identity, and that $t<t'$ implies the closure $\ov{\beta^U_{t'}(U)}\subset \beta^U_{t}(U)$.
There is some $t_0\in \RR_{>0}$ for which $\beta^U_{t_0}( U) \subset  V$. 
We have witnessed a right inverse $g_U:=\beta^U_{t_0}$ to $r$ in the quasi-category $\psnglr_n$.  

Use the notation $P(Q) = \psnglr_n(P,Q)$.  
To show $g$ is a left inverse we show the map of spaces $\dddot{U}(W) \xra{r_W} U(W)\xra{g_W} \dddot{U}(W)$ is equivalent to the identity for any $W\in \bsc_n$.  Specifically, we exhibit for each $W\xra{b} \dddot{U}$ a path from $b$ to $g_Wr_W(b)$ which is canonical up to a contractible space of choices.   
Let $W\xra{b} \dddot{U}$ be arbitrary.  Write $l$ for the depth of $W$ and let $q\in \RR^l\subset  W$. 
There is some $t_b\in [0,\infty)$ for which the dragged image of $q$
\[
\beta^U_{t'}(b(q)) \subset a( V)
\]
for all $t'> t_b$ -- we point out that such a $t_b$ is unique.  Choose such a $t'$.  
Because $[0,t']$ is compact, and the collection $\{ f( V')\mid V'\xra{f'} \dddot{U}\}$ forms a basis for the topology of $ \dddot{U}\cong  U$,  there is a finite collection of morphisms $V'_i\xra{f'_i} \dddot{U}$ whose underlying spaces cover the image of the path $\beta^U_t(b(q)) \colon [0,t'] \to  U$.  
As so, we can choose a $T\in [0,\infty)$ so that for any $t\in [0,t']$ and any $t''>T$ there is some $V_i$ for which there is the containment $(\beta^U_t b)( \beta^W_{t''}\bigl( W)\bigr)\subset  f'_i( V'_i)$.  We point out that such a $T$ is unique.  Choose such a $t''$. 
The choice of a $t'$ and $t''$ as above determines the desired path $[0,t'+t'']\to \dddot{U}(W)$ from $b$ to $p_Wr_W(b)$.  This path is given by a smoothly reparametrizing the path $t\mapsto b\beta^W_t$ for $t\in [0,t'']$ and $t\mapsto \beta^U_{(t-t'')}b\beta^W_{t''}$ for $t\in [t'',t'+t'']$.  

\end{proof}

\begin{remark}
Theorem~\ref{no-pre-mans} says that in a weak sense there is no distinction between a local invariant of $\cB$-premanifolds and one of $\cB$-manifolds.  Indeed, in so much as the difference between the quasi-category of $\cB$-premanifolds and that of $\cB$-manifolds is measured by refinements, which we think of as covers, the theorem implies that an invariant of $\cB$-manifolds which is specified by its values on basics, when evaluated on a $\cB$-manifold,
 is determined by is values on open cover of that $\cB$-manifold.  
Said another way, an invariant of $\cB$-basics automatically satisfies a (co)sheaf condition.  This has important consequences, as we have seen.  
\end{remark}


\begin{thebibliography}{99}

\bibitem[A]{atiyah} Atiyah, Michael. Thom complexes, Proc. London Math. Soc. (3) , no. 11 (1961), 291--310.

\bibitem[Ba]{baas} Baas, Nils. On formal groups and singularities in complex cobordism theory. Math. Scand. 33 (1973), 303Ð313 (1974).

\bibitem[BD]{bd} Beilinson, Alexander; Drinfeld, Vladimir. Chiral algebras. American Mathematical Society Colloquium Publications, 51. American Mathematical Society, Providence, RI, 2004.

\bibitem[BoVo]{bv} Boardman, J. Michael; Vogt, Rainer. Homotopy invariant algebraic structures on topological spaces. Lecture Notes in Mathematics, Vol. 347. Springer-Verlag, Berlin-New York, 1973. x+257 pp.

\bibitem[B\"o]{bodig} B\"odigheimer, C.-F. Stable splittings of mapping spaces. Algebraic topology (Seattle, Wash., 1985), 174--187, Lecture Notes in Math., 1286, Springer, Berlin, 1987.

\bibitem[CG]{kevinowen} Costello, Kevin; Gwilliam, Owen. Factorization algebras in perturbative quantum field theory. Preprint. Available at http://www.math.northwestern.edu/~costello/renormalization

\bibitem[F1]{cotangent} Francis, John. The tangent complex and Hochschild cohomology of $\cE_n$-rings. To appear, Compositio Mathematica.

\bibitem[F2]{facthomology} Francis, John. Factorization homology of topological manifolds. Preprint.

\bibitem[Ga]{galatius:graphs} Galatius, S\o ren. Stable homology of automorphism groups of free groups, Ann. of Math. 173 (2011), 705Ð768. math/0610216 MR2784914

\bibitem[GM1]{goreskymacpherson} Goresky, Mark; MacPherson, Robert. Intersection homology theory. Topology 19 (1980), no. 2, 135--162.

\bibitem[GM2]{goreskymacpherson2} Goresky, Mark; MacPherson, Robert. Intersection homology. II. Invent. Math. 72 (1983), no. 1, 77--129.

\bibitem[Jo]{joyal} Joyal, Andr\'e. Quasi-categories and Kan complexes. Special volume celebrating the 70th birthday of Professor Max Kelly. J. Pure Appl. Algebra 175 (2002), no. 1-3, 207--222.

\bibitem[LS]{simple} Longoni, Riccardo; Salvatore, Paolo. Configuration spaces are not homotopy invariant. Topology 44 (2005), no. 2, 375--380. 

\bibitem[Lu1]{topos} Lurie, Jacob. Higher topos theory. Annals of Mathematics Studies, 170. Princeton University Press, Princeton, NJ, 2009. xviii+925 pp.

\bibitem[Lu2]{dag} Lurie, Jacob. Higher algebra. http://www.math.harvard.edu/$\sim$lurie/

\bibitem[Lu3]{cobordism} Lurie, Jacob. On the classification of topological field theories. Current developments in mathematics, 2008, 129--280, Int. Press, Somerville, MA, 2009.

\bibitem[May]{may} May, J. Peter. The geometry of iterated loop spaces. Lectures Notes in Mathematics, Vol. 271. Springer-Verlag, Berlin-New York, 1972. viii+175 pp.

\bibitem[Mat]{mather} Mather, John. Notes on topological stability. Mimeographed notes, Harvard University, 1970.

\bibitem[Mc]{mcduff} McDuff, Dusa. Configuration spaces of positive and negative particles.  Topology 14 (1975), 91--107. 

\bibitem[Sa]{salvatore} Salvatore, Paolo. Configuration spaces with summable labels. Cohomological methods in homotopy theory (Bellaterra, 1998), 375--395, Progr. Math., 196, Birkh\"auser, Basel, 2001.

\bibitem[Se1]{segal} Segal, Graeme. Configuration-spaces and iterated loop-spaces. Invent. Math. 21 (1973), 213--221. 

\bibitem[Se2]{segallocal} Segal, Graeme. Locality of holomorphic bundles, and locality in quantum field theory. The many facets of geometry, 164--176, Oxford Univ. Press, Oxford, 2010. 

\bibitem[Sp]{spivak} Spivak, Michael. A comprehensive introduction to differential geometry, Volume 1. Publish or Perish Inc, Wilmington, DE, 1979. xiv+668 pp.

\bibitem[Th]{thomas} Thomas, Justin. Kontsevich's Swiss Cheese Conjecture. Thesis (PhD) -- Northwestern University. 2010.

\bibitem[V]{voronov} Voronov, Alexander. The Swiss-cheese operad. Homotopy invariant algebraic structures (Baltimore, MD, 1998), 365--373, Contemp. Math., 239, Amer. Math. Soc., Providence, RI, 1999. 

\end{thebibliography}
\end{document}